\documentclass{amsart}

\usepackage{color,hyperref,amssymb,braket} %mathrsfs
\usepackage{amsrefs} %\usepackage[alphabetic]{amsrefs}
\usepackage{epsfig}
\usepackage{makecell} %line breaks within table cell
\usepackage{newtxtext,newtxmath}
\usepackage{amsmath,amsfonts,tikz} 

\usetikzlibrary{shapes.geometric, arrows}
\usepackage{pgfplots}
\pgfplotsset{compat=1.18}
\usepgfplotslibrary{fillbetween}

\usepackage{geometry}
\theoremstyle{plain}
\newtheorem{theorem}{Theorem}[section]
\newtheorem{corollary}[theorem]{Corollary}
\newtheorem{lemma}[theorem]{Lemma}
\newtheorem{proposition}[theorem]{Proposition}
\newtheorem{example}[theorem]{Example}
\newtheorem{remark}[theorem]{Remark}

\newtheorem{assumption}[theorem]{Assumption}

\numberwithin{equation}{section}

\newcommand{\abs}[1]{\lvert#1\rvert}

\newcommand{\nm}[2]{\|\,#1\,\|_{#2}}

\newcommand\al{\alpha}

\newcommand\la{\langle}
\newcommand\ra{\rangle}

\newcommand{\mc}[1]{\mathcal{#1}}
\newcommand{\mb}[1]{\mathbb{#1}}
\newcommand{\ms}[1]{\mathscr{#1}}
\newcommand{\mr}[1]{\mathrm{#1}}

\newcommand{\wh}[1]{\widehat{#1}}

\newcommand{\mk}[1]{\mathfrak{#1}}
\begin{document}
\title[Barron regularity of many particle Schr\"odinger eigenfunctions]{Barron regularity of many particle Schr\"odinger eigenfunctions}
\author[P. B. Ming \and H. Yu]{Pingbing Ming \and Hao Yu}
\address{SKLMS, Institute of Computational Mathematics and Scientific/Engineering Computing, AMSS, Chinese Academy of Sciences, Beijing 100190, China}
\address{School of Mathematical Sciences, University of Chinese Academy of Sciences, Beijing 100049, China}
\email{mpb@lsec.cc.ac.cn, yuhao@amss.ac.cn}
\thanks{The work of Ming and Yu was supported by the National Natural Science Foundation of China under the grants 12371438.}
\date{\today}

\begin{abstract}
This work investigates the regularity of Schr\"odinger eigenfunctions and the solvability of Schr\"odinger equations in spectral Barron space $\mc{B}^{s}(\mb{R}^{nN})$, where neural networks exhibit dimension-free approximation capabilities. Under assumptions that the potential $V$ consists of one-particle and pairwise interaction parts $V_{i},V_{ij}$ in Fourier-Lebesgue space $\mc{F}L_{s}^{1}(\mb{R}^{n})+\mc{F}L_{s}^{\alpha^{\prime}}(\mb{R}^{n})$ and an additional part $V_{\operatorname{a d}} \in \mc{F}L_{s}^{1}(\mb{R}^{nN})$, we prove that all eigenfunctions $\psi\in \bigcap_{\gamma<s+2-n/\alpha} \mc{B}^{\gamma}(\mb{R}^{nN})$ and $\psi\in \mc{B}^{s+2}(\mb{R}^{nN})$ if $\alpha=\infty$, where $1/\alpha+1/\alpha^{\prime}=1$ and $2+s-|s|-n/\alpha>0$. The assumption accommodates many prevalent singular potentials, such as inverse power potentials. Moreover, under the same assumption or a stronger assumption $V\in\mc{B}^{s}(\mb{R}^{nN})$, we establish the solvability of Schr\"odinger equations and derive compactness results for $V\in\mc{B}^{s}(\mb{R}^{nN})$ with $s>-1$.
\end{abstract}
\maketitle

\section{Introduction} \label{Section: Introduction}
The many-particle Schr\"odinger equation stands as a cornerstone in quantum mechanics, enabling the precise description of complex systems. Its applications span a vast array of scientific fields and engineering disciplines, including condensed matter physics, quantum chemistry and quantum computing. While traditional numerical methods have achieved great success for low-dimensional problems, the curse of dimensionality (CoD) remains a major challenge when solving high-dimensional problems like Schr\"odinger equations. The computational costs increase exponentially with the dimensionality. Recently, machine learning has emerged as a promising approach to mitigate the CoD. Significant progress has been made to solve Schr\"odinger equations and eigenvalue problems~\cite{cai2018approximating, carleo2017solving, choo2020fermionic, gao2017efficient, eigenvalue_han2020solving, han2019solvingusingDNN, hermann2020deep, li2022semigroup, Xie2022}, among many others. 

Established lower bounds indicate that the neural network approximations of H\"older continuous functions~\cite{Yarotsky2017error,LuShenYangZhang:2021} or solutions to PDEs in Sobolev spaces~\cite{lu2022machine} may suffer from CoD.
Classical function spaces like H\"older spaces, Sobolev spaces and Besov spaces do not seem suitable to analyze the performance of machine learning methods for high-dimensional problems. In a series of seminal work~\cite{Barron:1992,Barron:1993,Barron:1994}, Bassron has proved the Monte Carlo approximation rates have been achieved by the neural networks for the functions in the so-called spactral Barron space $\mc{B}^{s}(\mb{R}^d)$, which is defined for  all $f \in \ms{S}'(\mb{R}^d)$ with the finite norm given by
\[
\nm{f}{\mc{B}^{s}(\mb{R}^d)}:=\int_{\mb{R}^d}(1+|\xi|^{2})^{s/2}\abs{\wh{f}(\xi)}\mr{d}\xi,
\]
where 
$\wh{f}$ is the Fourier transform of $f$ in the sense of distribution. 
%It has been widely employed to study the approximation and generalization properties of neural networks, notably for  dimension-free convergence rates that mitigate the CoD. 
%Particularly,
When the solution or the eigenfunction resides in a spectral Barron space, dimension-free generalization rates have been proved for solving Schr\"odinger equations and eigenvalue problems~\cite{DRMlu2021priori,lu2021priori,Guo2024gen}. The optimization dynamics for shallow networks in solving Schr\"odinger eigenvalue problems have been analyzed recently in~\cite{Dus2024two} assuming that the the eigenfunctions belongs to a Barron space. 

%Since the numerical analyses are based on the Barron regularity of eigenfunctions or solutions to PDEs, natural regularity questions would be: When is the eigenfunction in Barron space? When is the solution of PDE in Barron space? If in, what is the magnitude of the smoothness index?
%Given that numerical analyses rely on the Barron regularity of eigenfunctions or PDE solutions, natural questions arise: When does the eigenfunction belong to a Barron space? When does the PDE solution lie in a Barron space? If it does, how large is the smoothness index $s$? 
This study focuses on the Barron regularity of the many-particle Schr\"odinger equations. Let $\{x_{i}\}_{i=1}^{N} \subset \mb{R}^{n}$ and the configuration $x = (x_{1}, x_{2},\ldots, x_{N})$. Consider the Hamiltonian operator
\begin{equation*}
\mc{H} = -\sum_{i=1}^{N}\frac{1}{2\mu_{i}}\Delta_{i} + V ,
\end{equation*}
where $\mu_{i}>0$ are the masses of the particles and $V$ is a potential function
\begin{equation} \label{eq: interaction potential V}
V(x)=\sum_{i=1}^{N} V_{i}(x_{i})+\sum_{i<j} V_{i j}(x_{i}-x_{j})+V_{\operatorname{a d}}(x).
\end{equation}
The three terms in $V$ describes the external field on each single particle, the pairwise interactions between particles and some more complex interactions, respectively.
We shall prove the Barron regularity of eigenfunctions of $\mc{H}$ and the solvability of the equation $(\mc{H}+\rho I)u=f$ in Barron spaces. 
\subsection{Related works}
The study of the regularity of Schr\"odinger eigenfunctions can be traced back to \textsc{Kato's} groundbreaking work~\cite{kato1957eigenfunctions}, which established the H\"older continuity of eigenfunctions. %$V_{\operatorname{a d}}\in L^{\infty}(\mb{R}^{3})$, $\alpha\ge2$ $V_{i},V_{i j}\in (L^{\alpha}(\mb{R}^{3}))_{0}$, the $L^{\alpha}$ functions of bounded support.
Simon~\cite{Simon1974Pointwise} recovered the same H\"older continuity under a slightly different assumption on $V$ in the momentum space. Restricted to electronic wave functions, \cite{Fournais2005Sharp,Fournais2009Analytic} proved more refined far-reaching H\"older continuity results. It is worth mentioning that \textsc{Yserentant}~\cite{Yserentant2004On,Yserentant2005Sparse} established the regularity in Hilbert spaces of mixed derivatives for electronic wave functions, promoting the application of the sparse grid method to solve many particle Schr\"odinger equations. %, which has been successfully used in medium high dimension. 
Recently, motivated by the analysis of machine learning-based PDE solvers, Barron regularity has been proved for solutions to Schr\"odinger equation~\cite{DRMlu2021priori} and the ground state of the Schr\"odinger operator~\cite{lu2021priori} on the hypercube $(0,1)^{d}$. The authors in~\cite{Guo2024gen} later extended these results to all Schr\"odinger eigenfunctions with Dirichlet boundary condition. However, these results rely on the assumption $V\in\mc{B}^{s}$ with $s\ge0$, which excludes many commonly used potentials such as the Coulomb potential. Moreover, the compactness of the hypercube  simplifies the proof, as it allows for the representation of functions as Fourier series and enabling the compact embedding of a Barron space with higher smoothness into one with lower smoothness. Recently, \cite{Yserentant2025regularity} proved that the electronic wave functions with eigenvalues below the essential spectrum lie in $\mc{B}^{s}(\mb{R}^{3N})$ for $s<1$. This is the first Barron regularity result with a singular potential, but its multiplier estimate for $V$ \cite[Lemma 2.3]{Yserentant2025regularity} essentially depends on the specific inverse power form of the Fourier transform of the Coulomb potential. 

For static Schr\"odinger equation $(\mc{H}+\rho I)u=f$ on $\mb{R}^{d}$, \cite{Chen2023regularity} proved the existence and uniqueness of the solution in $\mc{B}^{s+2}(\mb{R}^d)$ provided that $\rho>0$, $V+\rho\ge0$ and $V\in\mc{B}^{s}(\mb{R}^d)$ with $s\ge0$. The assumptions therein rule out terms $V_{i}$ and $V_{ij}$, as well as the possibility of singular potentials.  For general singular potentials, the Barron regularity of the eigenfunctions and the solvability of Schr\"odinger equations remain open. 

It is worthwhile to mention that Barron regularity has also been proved for solutions to the screened Poisson equation, some time-dependent equations in \cite{E2022Some} and Hamilton-Jacobi-Bellman equation in~\cite{Feng2025HJB}, which is also a typical high-dimensional PDE. 
%The complexity estimates in~\cite{Marwah2021Parametric,Chen2021Representation} lie in a similar spirit.
Compared to the classical function spaces, regularity theories in Barron space seem more difficult, as addressed in~\cite{E2022Some,lu2021priori,Chen2023regularity}.
\subsection{Our contributions}
In this paper, we establish the Barron regularity for the eigenfunctions of $\mc{H}$ and the solvability of the equation $(\mc{H}+\rho I)u = f$ in the spectral Barron spaces. The main results and their connections to related studies are visually summarized in Figure \ref{Figure: structure of main results in this paper}.

The first result is the regularity of the eigenfunctions. We show in Theorem \ref{Theorem: Barron regularity of the eigenfunctions} that under Assumption \ref{Assumption: V_i, v_ij in mcFL_s^1+mcFL_s^^(alpha^prime), 2+s-|s|-n/alpha>0}, any eigenfunction $\psi$ of $\mc{H}$ has Barron regularity
\begin{equation} \label{eq: regularity of eigenfunctions, summary}
\begin{aligned}
\psi &\in \bigcap_{\gamma<s+2-n/\alpha}\mc{B}^{\gamma}(\mb{R}^{Nn}),\  
&&\text{if\;} \alpha<\infty,\\
\psi&\in \mc{B}^{s+2}(\mb{R}^{Nn}),\  &&\text{if\;} \alpha=\infty,
\end{aligned}
\end{equation}
with the explicit norm estimates. The case $\alpha=\infty$ reduces to a shift estimate in Barron space; see Theorem~\ref{Theorem: Barron regularity lift of the eigenfunctions}. We recover Simon's estimate~\cite[Theorem $1^{\prime}$]{Simon1974Pointwise} from Theorem~\ref{Theorem: Barron regularity of the eigenfunctions} with $s=0$; see Theorem \ref{Theorem: regularity of eigenfunction when Fourier of V_i, v_ij in L^1+L^(alpha^prime)}. Then, as an application of Theorem~\ref{Theorem: regularity of eigenfunction when Fourier of V_i, v_ij in L^1+L^(alpha^prime)}, when $V_{i}$, $V_{ij}$ are inverse power potentials $|x|^{-t}$, we obtain $\psi\in\mc{B}^{\gamma}(\mb{R}^{Nn})$ for $\gamma < 2-t$ and naturally recover Yserentant’s result~\cite{Yserentant2025regularity} for electronic wave functions with $t=1$; see 
Corollary~\ref{Corollary: regularity of eigenfunctions under inverse power potential}. The sharpness of 
Theorem~\ref{Theorem: Barron regularity of the eigenfunctions} near the borderline case is illustrated by an example from radial Schr\"odinger equation in Example~\ref{Example: Sharpness of regularity of eigenfunctions}. 

The second result concerns the solvability of the equation $(\mc{H}+\rho I)u = f$. We show in Theorem \ref{Theorem: solvability of (mcH+rho I)u = f under two assumptions} that under Assumption \ref{Assumption: V_i, v_ij in mcFL_s^1+mcFL_s^^(alpha^prime), 2+s-|s|-n/alpha>0} and if $\rho$ is large enough, then, for any $f\in \mc{B}^{\gamma-2}(\mb{R}^{nN})\cap H^{-1}(\mb{R}^{nN})$, there exists a unique weak solution $u^{*}\in\mc{B}^{\gamma}(\mb{R}^{nN})\cap H^{1}(\mb{R}^{nN})$. This solvability result only requires $V$ to satisfy the weakest assumption in this paper. If we weaken the requirements for $\rho$ and $f$, and strengthen the assumption on $V$, then we obtain that when $V\in\mc{B}^{s}(\mb{R}^{nN})$ with $s>-1$ and $\rho>0$, either the homogeneous problem $(\mc{H}+\rho I)u=0$ has a nonzero solution $u\in\mc{B}^{s+2}(\mb{R}^{nN})$ or for any $f\in\mc{B}^{s}(\mb{R}^{nN})$, there exists a unique solution $u^{*}\in\mc{B}^{s+2}(\mb{R}^{nN})$. 
These results fill the gap in the solvability of the Schr\"odinger equation with singular potentials in Barron spaces. 

To seek conditions for $V$ as broad as possible ensuring the Barron regularity, we work on Fourier-Lebesgue spaces, which are wider than the spectral Barron space. Assumption \ref{Assumption: V_i, v_ij in mcFL_s^1+mcFL_s^^(alpha^prime), 2+s-|s|-n/alpha>0} accommodates many commonly used singular potentials. The key is the multiplier estimates in Fourier-Lebesgue spaces that is suitable for studying Schr\"odinger equations, and proving compactness results including the situation $V\in\mc{B}^{s}(\mb{R}^{nN})$ with negative indices $s$.
\thispagestyle{empty}
\tikzstyle{startstop} = [rectangle, rounded corners, minimum width=0.1cm, minimum height=0.1cm,text centered, draw=black, fill=green!10]
\tikzstyle{Assum} = [rectangle, rounded corners, minimum width=0.1cm, minimum height=0.1cm, text centered, draw=black, fill=orange!13]
\tikzstyle{arrow} = [thick,->,>=stealth]
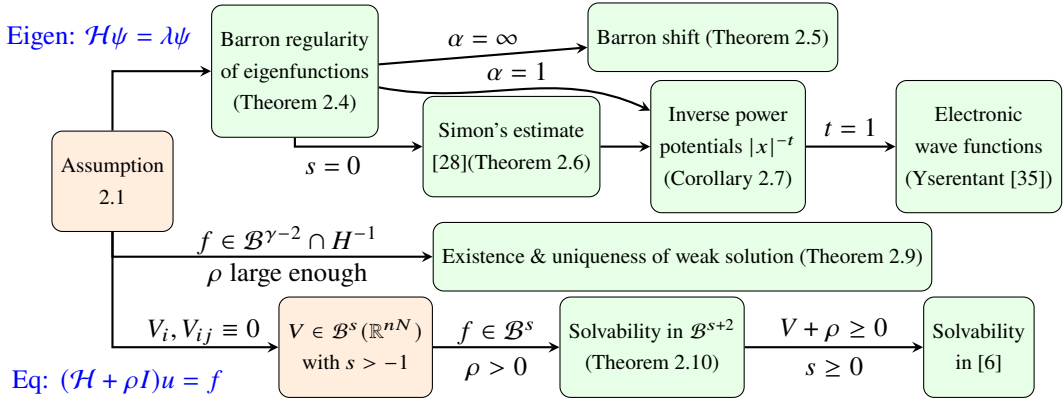
\begin{figure}
\centering 
\begin{tikzpicture}[node distance=2cm]
\node (Assump) [Assum] {\thead{Assumption\\ \ref{Assumption: V_i, v_ij in mcFL_s^1+mcFL_s^^(alpha^prime), 2+s-|s|-n/alpha>0}}};
\node (ThmEigen) [startstop, right of=Assump, yshift=1.46cm, xshift=0.4cm] {\thead{Barron regularity\\ of eigenfunctions\\ (Theorem \ref{Theorem: Barron regularity of the eigenfunctions})}};
\node (ThmEiBa) [startstop, right of=ThmEigen, yshift=0.45cm, xshift=3.5cm] {\thead{Barron shift (Theorem \ref{Theorem: Barron regularity lift of the eigenfunctions})}};
\node (ThmEiSi) [startstop, right of=ThmEigen, yshift=-1.00cm, xshift=0.85cm] {\thead{Simon’s estimate\\ \cite{Simon1974Pointwise}(Theorem \ref{Theorem: regularity of eigenfunction when Fourier of V_i, v_ij in L^1+L^(alpha^prime)})}};
\node (CoEiInv) [startstop, right of=ThmEiSi, yshift=-0.0cm, xshift=0.85cm] {\thead{Inverse power\\ potentials $|x|^{-t}$\\ (Corollary \ref{Corollary: regularity of eigenfunctions under inverse power potential})}};
\node (Coulomb) [startstop, right of=CoEiInv, yshift=-0.0cm, xshift=1.3cm] {\thead{Electronic \\ wave functions\\ (Yserentant~\cite{Yserentant2025regularity})}};

\node (ThmSoW) [startstop, right of=Assump, yshift=-1.0cm, xshift=5.5cm] {\thead{Existence \& uniqueness of weak solution (Theorem \ref{Theorem: solvability of (mcH+rho I)u = f under two assumptions})}};
\node (AssumC) [Assum, below of=Assump, yshift=-0.2cm, xshift=3.2cm] {\thead{$V\in\mc{B}^{s}(\mb{R}^{nN})$\\ with $s >-1$}};
\node (ThmSoC) [startstop, right of=AssumC, yshift=0cm, xshift=1.9cm] {\thead{Solvability in $\mc{B}^{s+2}$\\ (Theorem \ref{Theorem: solvability of (mcH+rho I)u = f when V in mcB^s(mbR^(nN))}) }};
\node (ThmSoLu) [startstop, right of=ThmSoC, yshift=0cm, xshift=2.3cm] {\thead{Solvability \\ in \cite{Chen2023regularity}}};

%Connection
\draw [arrow](Assump) |- (ThmEigen);
\draw [arrow](ThmEigen) -- node[above] {$\alpha=\infty$} (ThmEiBa);
\draw [arrow](ThmEigen) |- node[below right] {$s=0$} (ThmEiSi);
\draw [arrow](ThmEiSi) -- (CoEiInv);
\draw [arrow](CoEiInv) -- node[above] {$t=1$} (Coulomb);
\draw [arrow](ThmEigen) edge [in=155, out=-11] node[above] {$\alpha=1$} (CoEiInv);%(CoEiInv)+(-0.7,1.1);
%\path (ThmEigen) edge [in=-110, out=-70] (CoEiInv)
%(B) edge [in=110, out=-70,looseness=8] (B)

\draw [arrow](Assump) |- node[above, yshift=-0.1cm, xshift=2.3cm] {$f\in\mc{B}^{\gamma-2}\cap H^{-1}$} node[below, yshift=0.05cm, xshift=2.3cm] {$\rho$ large enough} (ThmSoW);
\draw [arrow](Assump) |- node[above, yshift=-0.1cm, xshift=1.2cm] {$V_{i},V_{ij}\equiv0$} (AssumC);
\draw [arrow](AssumC) -- node[above, yshift=-0.1cm, xshift=0.0cm] {$f\in\mc{B}^{s}$} node[below, yshift=0.0cm, xshift=0.0cm] {$\rho>0$} (ThmSoC);
\draw [arrow](ThmSoC) -- node[above, yshift=-0.05cm, xshift=0.0cm] {$V+\rho\ge0$} node[below, yshift=0.0cm, xshift=0.0cm] {$s\ge0$} (ThmSoLu);

\node at (-0.17,1.9) {\textcolor{blue}{Eigen: $\mc{H}\psi = \lambda\psi$}};
\node at (0.07,-2.65) {\textcolor{blue}{Eq: $(\mc{H}+\rho I)u = f$}};
\end{tikzpicture}
\caption{The structure diagram of main results in this paper} \label{Figure: structure of main results in this paper}
\end{figure}

\subsection{Notations and preliminary}
We denote the Gamma function as $\Gamma(x)$ for $x\in\mb{R}$.
\iffalse
\[
\Gamma(x) =\int_0^\infty t^{x-1}e^{-t}\mr{d}t, \quad \text{for } x>0.
\]
\fi
Denote by $\omega_{d} = 2 \pi^{d/ 2} \Gamma(d/2)^{-1}$ the surface area of the unit sphere $\mb{S}^{d-1}$. For a vector $x\in\mb{R}^{d}$, we denote $\abs{x}$ as its $\ell^{2}$-norm and the bracket $\left\langle x \right\rangle = (1 + |x|^{2})^{1/2}$.
%For a vector $x$, we denote $|x|_{p}$ as its $\ell^{p}$-norm and $|x|$ stands for $|x|_{2}$ for brevity. Denote $\left\langle x \right\rangle = (1 + |x|^{2})^{1/2}$.
%We denote $\mathbb{N}$ as the set of all natural numbers, and $\mathbb{N}_{+}$ as the set of all positive integers. For an index $\eta \in \mathbb{N}^{d}$, $$\partial^{\eta}u = \frac{\partial^{|\eta|_{1}}u}{\partial x_{1}^{\eta_{1}}\cdots\partial x_{d}^{\eta_{d}}} ,$$ and $w^{\eta} = w_{1}^{\eta_{1}}w_{2}^{\eta_{2}}\cdots w_{d}^{\eta_{d}}$.
The Fourier transform is defined on $L^{1}(\mb{R}^{d})$  by
\[
\mc{F} f(\xi)=\wh{f}(\xi)= \int_{\mb{R}^n} f(x) e^{-2\pi i\xi^{\top}x} \mr{d} x.
\]
It may be extended to the tempered distributions $\ms{S}'(\mb{R}^{d})$ by duality~\cite{Hormander1983analysis}. For $s\in \mb{R}$, the Fourier-Lebesgue space $\mc{F}L_{s}^{p}(\mb{R}^{d})$ is defined as%\footnote{Since the space considered in this paper is mainly on $\mb{R}^{d}$, we omit the dependency on $\mb{R}^{d}$ in the norm subscript for brevity.} %the Banach space which consists of all  $f \in \ms{S}^{\prime}(\mb{R}^{d})$  such that
\begin{equation*}
\mc{F}L_{s}^{p}(\mb{R}^{d}) := \{f\in\ms{S}'(\mb{R}^{d}) : \|f\|_{\mc{F}L_{s}^{p}(\mb{R}^{d})} = \|\la\cdot\ra^{s}\mc{F}(f)\|_{L^{p}(\mb{R}^{d})}<\infty\}.  
\end{equation*}
According to~\cite[Theorem 10.1.7]{Hormander1983analysis}, $\mc{F}L_{s}^{p}(\mb{R}^{d})$ is a Banach space. The case $p=2$ coincides with the Sobolev space $H^{s}(\mb{R}^{d})$ and $\|f\|_{H^{s}(\mb{R}^{d})}=\|f\|_{\mc{F}L_{s}^{2}(\mb{R}^{d})}.$ 
The case $p=1$ corresponds to the spectral Barron space $\mc{B}^{s}(\mb{R}^{d})$~\cite{Barron:1993,LiaoMing:2023}. 

Let $1\le \alpha\le \infty$ with $\alpha'$ its conjugate index. For $\beta>n/(2\alpha)$, we define $c_{\alpha\beta} = \pi^{n / 2} \Gamma(\alpha\beta - n/2)/\Gamma(\alpha\beta)$. When $\alpha=\infty$, we let $\beta\ge0$ and $c_{\alpha\beta}=1$. 

Define $\mc{F}L_{s}^{1}(\mb{R}^{d})+\mc{F}L_{s}^{\alpha^{\prime}}(\mb{R}^{d})$ as the space of all functions $f \in \ms{S}^{\prime}(\mb{R}^{d})$ with %can be expressed as the sum $f = f_{1} + f_{2}$, where $f_{1} \in \mc{F}L_{s}^{1}(\mb{R}^{d})$ and $f_{2} \in \mc{F}L_{s}^{\alpha^{\prime}}(\mb{R}^{d})$. We define on the space $\mc{F}L_{s}^{1}(\mb{R}^{d})+\mc{F}L_{s}^{\alpha^{\prime}}(\mb{R}^{d})$ 
the rescaled norm\footnote{The rescaled norms are equivalent with different $\beta$.} 
\begin{equation} \label{def: rescaled norms for FL_s^1+FL_s^(alpha')}
\|f\|_{s,\alpha;\beta} = \inf\{ \|f_{1}\|_{\mc{F}L_{s}^{1}(\mb{R}^{d})} + c_{\alpha\beta}^{1/\alpha} \|f_{2}\|_{\mc{F}L_{s}^{\alpha^{\prime}}(\mb{R}^{d})} : f = f_{1}+f_{2}\} .
\end{equation}
Then, $\mc{F}L_{s}^{1}(\mb{R}^{d})+\mc{F}L_{s}^{\alpha^{\prime}}(\mb{R}^{d})$ is a quotient of the Banach space $\mc{F}L_{s}^{1}(\mb{R}^{d})\times\mc{F}L_{s}^{\alpha^{\prime}}(\mb{R}^{d})$ with respect to the subspace $\mc{F}L_{s}^{1}(\mb{R}^{d})\cap\mc{F}L_{s}^{\alpha^{\prime}}(\mb{R}^{d})$ and hence itself a Banach space. Note that $\mc{F}L_{s}^{1}(\mb{R}^{d})+\mc{F}L_{s}^{1}(\mb{R}^{d}) = \mc{B}^{s}(\mb{R}^{d})$ and $\|f\|_{s,\infty;\beta} = \|f\|_{\mc{B}^{s}(\mb{R}^{d})}$ for any $\beta\ge0$.

The remainder of the paper is organized as follows. In \S \ref{Section: Main results}, we summarize regularity estimates for eigenfunctions and solvability results for Schr\"odinger equations. Then, we prove regularity estimates for eigenfunctions in \S \ref{Section: Barron regularity of eigenfunctions} and solvability results for Schr\"odinger equations in \S \ref{Section: Solvability of Schrodinger equations}, respectively. Finally, we conclude this paper and discuss future works in \S \ref{Section: Conclusion and future work}.

\section{Main results} \label{Section: Main results}
In this paper, we study the regularity of the eigenfunctions of $\mc{H}$ and the solvability of the equation $(\mc{H}+\rho I)u = f$ in spectral Barron spaces under the following assumption.
\begin{assumption} \label{Assumption: V_i, v_ij in mcFL_s^1+mcFL_s^^(alpha^prime), 2+s-|s|-n/alpha>0}
Let $1\le \alpha,\alpha^{\prime}\le \infty$ be conjugate exponents, $s\in\mb{R}$ and $2+s-|s|-n/\alpha>0$. 
$V_{i}$, $V_{ij}\in \mc{F}L_{s}^{1}(\mb{R}^{n})+\mc{F}L_{s}^{\alpha^{\prime}}(\mb{R}^{n})$ and $V_{\operatorname{a d}} \in \mc{F}L_{s}^{1}(\mb{R}^{Nn})$.
\end{assumption}
The above assumption may be reshaped into 
\begin{enumerate}
\item $s\ge 0$ and $\alpha>n/2$;

\item $-1<s<0$ and $\alpha>n/[2(1+s)]$.
\end{enumerate}
%When $s\ge 0$, the above assumption implies $\alpha>n/2$; When $s<0$, the above assumption leads to an extra constaint on $s$ as $-1<s<0$ and $\alpha>n/[2(1+s)]$.

\begin{lemma} \label{Lemma: index condition s<n/alpha-t for |x|^(-t) to satisfy Assumption}
For the $n$-dimensional inverse power potential $\abs{x}^{-t}$ with $0<t<n$, we have
\[
\abs{x}^{-t}\in \mc{F}L_{s}^{1}(\mb{R}^n)+\mc{F}L_{s}^{\alpha^{\prime}}(\mb{R}^n), \quad \text{with $s<n/\alpha-t$.}
\]
Particularly, the Coulomb potential 
$\abs{x}^{-1}$ in three dimension lies in $\mc{F}L_{s}^1(\mb{R}^3)+\mc{F}L_{s}^{\alpha^{\prime}}(\mb{R}^3)$ with $s<3/\alpha-1$. In three dimension, the Yukawa potential $\abs{x}^{-1}\mr{e}^{-\mu\abs{x}}\in\mc{F}L_{s}^1(\mb{R}^3)+\mc{F}L_{s}^{\alpha^{\prime}}(\mb{R}^3)$ with $s<3/\alpha-1$, where $\mu>0$ usually denotes the mass of the exchanged meson. 
\end{lemma}

We postpone the proof of Lemma \ref{Lemma: index condition s<n/alpha-t for |x|^(-t) to satisfy Assumption} to Appendix \ref{Appendix section: Estimate_Barron_smoothness}. 

To clarify the roles of $s$ and $\alpha$ in Assumption \ref{Assumption: V_i, v_ij in mcFL_s^1+mcFL_s^^(alpha^prime), 2+s-|s|-n/alpha>0}, we recall certain properties of the Fourier-Lebesgue spaces.
\begin{proposition} \label{Proposition: roles of s and alpha in the assumption and their relationship}
(1) $\mc{F}L_{s}^{p}(\mb{R}^{n})\hookrightarrow\mc{F}L_{t}^{p}(\mb{R}^{n})$ for $s\ge t$ and $1\le p\le\infty$;\\
(2) $\mc{F}L_{s}^{1}(\mb{R}^{n})+\mc{F}L_{s}^{p}(\mb{R}^{n})\hookrightarrow\mc{F}L_{s}^{1}(\mb{R}^{n})+\mc{F}L_{s}^{r}(\mb{R}^{n})$ for $1\le p\le r\le\infty$ and $s\in\mb{R}$;\\
(3) Let $1\le \alpha_{1}<\alpha_{2}\le \infty$ with $1/\alpha_{i}+1/\alpha_{i}^{\prime}=1$ and $i=1,2$. Then, for $s_{2}-n/\alpha_{2}<s_{1}-n/\alpha_{1}$, $$\mc{F}L_{s_{1}}^{1}(\mb{R}^{n})+\mc{F}L_{s_{1}}^{\alpha_{1}^{\prime}}(\mb{R}^{n})\hookrightarrow\mc{F}L_{s_{2}}^{1}(\mb{R}^{n})+\mc{F}L_{s_{2}}^{\alpha_{2}^{\prime}}(\mb{R}^{n}) ,$$
and for $s_{2}-n/\alpha_{2}\ge s_{1}-n/\alpha_{1}$,
\begin{equation} \label{eq: embedding relationship FL_(s_1)^1+FL_(s_1)^(alpha_1^') not in FL_(s_2)^1+FL_(s_2)^(alpha_2^') if alpha_1<alpha_2 and s_2-n/alpha_2 = s_1-n/alpha_1}
\mc{F}L_{s_{1}}^{1}(\mb{R}^{n})+\mc{F}L_{s_{1}}^{\alpha_{1}^{\prime}}(\mb{R}^{n})\not\hookrightarrow\mc{F}L_{s_{2}}^{1}(\mb{R}^{n})+\mc{F}L_{s_{2}}^{\alpha_{2}^{\prime}}(\mb{R}^{n}) ;
\end{equation}
(4) For any $s<t$, there exists $f\in \mc{B}^{s}(\mb{R}^{n})$ such that $f\notin \mc{F}L_{t}^{1}(\mb{R}^{n})+\mc{F}L_{t}^{\infty}(\mb{R}^{n})$;\\
(5) For $n\ge2$, $|x|^{-1}\in \mc{F}L_{0}^{1}(\mb{R}^{n})+\mc{F}L_{0}^{\alpha^{\prime}}(\mb{R}^{n})$ with $\alpha<n$ but $|x|^{-1}\notin\mc{B}^{-1}(\mb{R}^{n})$. For $n=1$, $x^{-1}\in \mc{F}L_{0}^{1}(\mb{R})+\mc{F}L_{0}^{\infty}(\mb{R})$ but $x^{-1}\notin\mc{B}^{-1}(\mb{R})$. 
\end{proposition}

We postpone the proof of Proposition \ref{Proposition: roles of s and alpha in the assumption and their relationship} to Section \ref{Section: Barron regularity of eigenfunctions}. The left image in Figure \ref{Figure: Visualize main results of this paper} illustrates embedding relations in Proposition \ref{Proposition: roles of s and alpha in the assumption and their relationship}, providing an intuitive grasp of how $s$ and $\alpha$ vary in Assumption \ref{Assumption: V_i, v_ij in mcFL_s^1+mcFL_s^^(alpha^prime), 2+s-|s|-n/alpha>0}. Each point $(\alpha,s)$ in the figure represents a space $\mc{F}L_{s}^{1}(\mb{R}^{n})+\mc{F}L_{s}^{\alpha^{\prime}}(\mb{R}^{n})$. Given a pair $(\alpha_{1},s_{1})$, the first two points of Proposition \ref{Proposition: roles of s and alpha in the assumption and their relationship} indicate that this space can be embedded in the orange solid line and the blue shaded area beneath it. The third point of Proposition \ref{Proposition: roles of s and alpha in the assumption and their relationship} states that it can be embedded in the green shaded area below the dashed orange line, but not in the orange dashed line (and the area above it). The fourth point of Proposition \ref{Proposition: roles of s and alpha in the assumption and their relationship} reveals that it cannot be embedded in any space with a higher smoothness index $s$. Roughly speaking, the space can be embedded below the orange line but not above it.

As to Assumption \ref{Assumption: V_i, v_ij in mcFL_s^1+mcFL_s^^(alpha^prime), 2+s-|s|-n/alpha>0}, Proposition \ref{Proposition: roles of s and alpha in the assumption and their relationship} (1)-(2) assert that Assumption \ref{Assumption: V_i, v_ij in mcFL_s^1+mcFL_s^^(alpha^prime), 2+s-|s|-n/alpha>0} weakens as $s$ or $\alpha$ decreases with the other remaining unchanged. Recall $\mc{B}^s(\mb{R}^{d}) \hookrightarrow C^s(\mb{R}^{d})$ for $s\ge0$ by~\cite[Corollary 2.13]{LiaoMing:2023}. When $s=0$ and $\alpha=\infty$, the potentials satisfy Assumption \ref{Assumption: V_i, v_ij in mcFL_s^1+mcFL_s^^(alpha^prime), 2+s-|s|-n/alpha>0} are continuous and bounded. While for $s<0$ or $\alpha<\infty$, Assumption \ref{Assumption: V_i, v_ij in mcFL_s^1+mcFL_s^^(alpha^prime), 2+s-|s|-n/alpha>0} accommodates various singular potentials, such as inverse power potentials (see Corollary \ref{Corollary: regularity of eigenfunctions under inverse power potential}). Point (3) studies the cases when increasing $\alpha$ and decreasing $s$. Points (4) and (5) offer examples highlighting that a single reduction of $\alpha$ (when $s=0$) or $s$ (when $\alpha=\infty$) does not encompass all instances of Assumption \ref{Assumption: V_i, v_ij in mcFL_s^1+mcFL_s^^(alpha^prime), 2+s-|s|-n/alpha>0}.  
\iffalse
\begin{assumption} \label{Assumption: V_i, v_ij in mcFL_s^1+mcFL_s^^(alpha^prime), 2+s-|s|-n/alpha>0}
Let $1\le \alpha,\alpha^{\prime}\le \infty$ be conjugate exponent $(1/\alpha+1/\alpha^{\prime}=1)$, $s\in\mb{R}$ and $2+s-|s|-n/\alpha>0$. 
$V_{i}$, $V_{ij}$ may be decomposed into two parts $V_{i} = V_{i,1} + V_{i,2}$, $V_{ij} = V_{ij,1} + V_{ij,2}$ such that $V_{i,1},V_{ij,1} \in \mc{F}L_{s}^{1}(\mb{R}^{n})$ and $V_{i,2},V_{ij,2}\in\mc{F}L_{s}^{\alpha^{\prime}}(\mb{R}^{n})$. Additionally, $V_{\operatorname{a d}} \in \mc{F}L_{s}^{1}(\mb{R}^{Nn})$.
\end{assumption}
\fi
%
%
\begin{figure}
\centering 
\begin{tikzpicture}
  \def \k{1.0};
  \def \b{1.3};
  \begin{axis}[
    width=.45\textwidth,
    %height=.8\textheight,
    axis y line = left,
    axis x line = bottom,
    xlabel = {\(\alpha\)},
    ylabel = {\(s\)},
    xtick       = {0},
    xticklabels = {$1$},
    ytick       = {},
    yticklabels = {},
    ylabel near ticks,
    %xlabel near ticks,
    samples     = 160,
    domain      = 0:4.2,
    xmin = -0.1, xmax = 4.9,
    ymin = -1.6, ymax = 3.5,
    legend style={nodes={scale=0.88, transform shape}},
  ]
  %\node [rotate=0] at (axis cs:  -0.0,  3.39) {$s$};
  %\node [rotate=0] at (axis cs:  4.4,  -1.4) {$\alpha$};

  %\addplot [name path=lowbd, domain=0:4.8, samples=100, color=black, line width=1.2pt,]{4.5/(2*(exp(\k*\x)-1+1.5))-1.5}; \addlegendentry{\(2+2s-n/\alpha=0\)} 
  \addplot [name path=invp, domain=\b:4.8, samples=100, color=orange, line width=1.2pt, dashed]{4.5/(exp(\k*\x)-1+1)}; \addlegendentry{\(s-n/\alpha = s_{1}-n/\alpha_{1}\)} 
  \addplot [name path=s1, domain=0:\b, samples=100, color=orange, line width=1pt]{4.5/(exp(\k*\b)-1+1)}; 

  \addplot[name path=line1, domain=0:\b, gray, no markers, line width=1pt] {-5};
  \addplot[name path=line2, domain=\b:4.8, gray, no markers, line width=1pt] {-5};
  \addplot[fill=blue,fill opacity=0.05] fill between[
  of = s1 and line1, 
  ];
  \addplot[fill=green,fill opacity=0.05] fill between[
  of = invp and line2, 
  ];
  
  \addplot[color=red, mark=*, mark size=1.6pt, fill=red, fill opacity=1,] coordinates  {({\b}, {4.5/(exp(\k*\b)-1+1)})} node[above right] {$(\alpha_{1},s_{1})$};
\end{axis}
\end{tikzpicture}
\begin{tikzpicture}
  \def \k{1.1};
  \def \a{ln(2.5)/1.1};
  \begin{axis}[
    width=.45\textwidth,
    %height=.8\textheight,
    axis y line = left,
    axis x line = bottom,
    xlabel = {\(\alpha\)},
    ylabel = {\(s\)},
    xtick       = {0,{\a}},
    xticklabels = {$\frac{n}{2}$,$n$},
    ytick       = {-1.5,0,3},
    yticklabels = {-1,0,$2$},
    ylabel near ticks,
    xlabel near ticks,
    samples     = 160,
    domain      = 0:4.2,
    xmin = -0.2, xmax = 4.9,
    ymin = -1.6, ymax = 3.5,
    legend style={nodes={scale=0.88, transform shape}},
  ]

  \addplot [name path=lowbd, domain=0:4.8, samples=100, color=black, line width=1.2pt,]{4.5/(2*(exp(\k*\x)-1+1.5))-1.5}; \addlegendentry{\(2+2s-n/\alpha=0\)} 
  \addplot [domain=0:4.8, samples=100, color=magenta, line width=1.1pt, dashed]{4.5/(exp(\k*\x)-1+1.5)}; \addlegendentry{\(s-n/\alpha = 0\)} 
  \addplot [name path=coulomb, domain=0:4.8, samples=100, color=blue!70, line width=1.1pt, dashed]{4.5/(exp(\k*\x)-1+1.5) - 1*1.5}; \addlegendentry{\(s-n/\alpha = -1\)} 
  \addplot [domain=0:\a, samples=100, color=red, line width=1.1pt, dashed]{4.5/(exp(\k*\x)-1+1.5) - 1.5*1.5}; \addlegendentry{\(s-n/\alpha = -1.5\)} 

  \addplot[domain=0:4.8, black, no markers, line width=0.07pt, opacity=0.3] {0};
  \addplot[name path=line, gray, no markers, line width=1pt] {5};
  \addplot[fill=orange!60!red,fill opacity=0.07] fill between[
  of = lowbd and coulomb, 
  ];
  \addplot[fill=green,fill opacity=0.05] fill between[
  of = coulomb and line, 
  ];

  \addplot[color=red, mark=*, mark size=1.6pt, fill=red, fill opacity=1,] coordinates  {({\a}, {4.5/(2*(exp(\k*\a)-1+1.5))-1.5})};
\end{axis}
\end{tikzpicture}
\caption{Left: Embedding relations in Proposition \ref{Proposition: roles of s and alpha in the assumption and their relationship}. Each point $(\alpha,s)$ represents a space $\mc{F}L_{s}^{1}(\mb{R}^{n})+\mc{F}L_{s}^{\alpha^{\prime}}(\mb{R}^{n})$. Right: Graphical interpretation for the regularity results of eigenfunctions. The shaded area corresponds to Assumption \ref{Assumption: V_i, v_ij in mcFL_s^1+mcFL_s^^(alpha^prime), 2+s-|s|-n/alpha>0}. Potentials on each dashed line leads to the same regularity of eigenfunctions\protect\footnotemark. Example \ref{Example: Sharpness of regularity of eigenfunctions} shows the sharpness of our regularity estimates in the orange area.} \label{Figure: Visualize main results of this paper}
\end{figure}
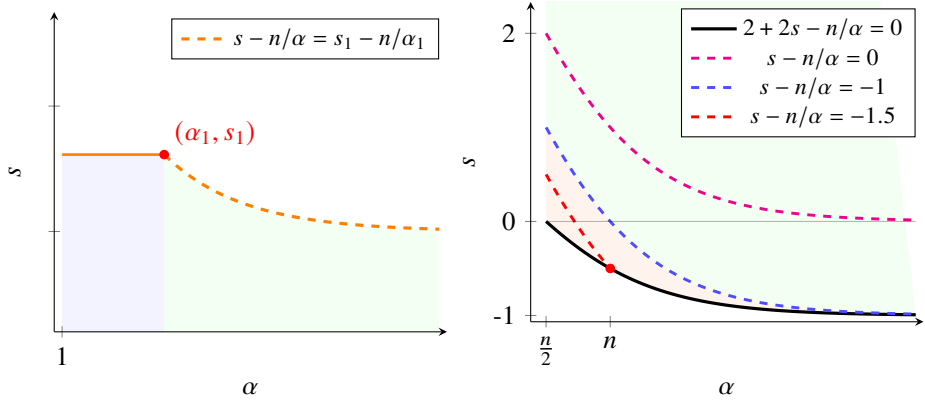 
\footnotetext{Here we draw the case where $n\ge2$, and it is similar for $n=1$.}

\subsection{Barron regularity of eigenfunctions}
The bilinear form $a:H^{1}(\mb{R}^{nN})\times H^{1}(\mb{R}^{nN})\to\mb{R}$ associated with $\mc{H}$ reads as
\begin{equation*}
a(u,v) = (\mc{H}u,v)_{L^{2}} = \sum_{i=1}^{N}\frac{1}{2\mu_{i}}\int_{\mb{R}^{nN}} \nabla_{i}u\nabla_{i}v\,\mr{d}x + \int_{\mb{R}^{nN}}Vuv\,\mr{d}x, 
\end{equation*}
where $(\cdot,\cdot)_{L^{2}}$ is the $L^{2}$-inner product.
Under Assumption \ref{Assumption: V_i, v_ij in mcFL_s^1+mcFL_s^^(alpha^prime), 2+s-|s|-n/alpha>0}, Corollary \ref{Corollary: form-boundedness of V} proves that $V$ is $-\Delta$ form-bounded with relative bound zero. Therefore, $a$ is a bounded bilinear form and eigenfunctions are characterized by the minimax principle of the Rayleigh quotient. 
We call a function $\psi \neq 0$ in $H^{1}(\mb{R}^{nN})$ an eigenfunction of $\mc{H}$  with the corresponding eigenvalue $\lambda$ if %the relation
\[
a(\psi,v) = \lambda (\psi,v)_{L^{2}},\qquad \text{for all\quad}v\in H^{1}(\mb{R}^{nN}).
\]

The following theorem establishes the Barron regularity (\ref{eq: regularity of eigenfunctions, summary}) of Schr\"odinger eigenfunctions under Assumption \ref{Assumption: V_i, v_ij in mcFL_s^1+mcFL_s^^(alpha^prime), 2+s-|s|-n/alpha>0}.
Denote $\tilde{\mu}_{\rho} = \max_{1\le i\le N}\left(\frac{\mu_{i}}{2\pi^{2}},\frac{1}{\rho}\right)$ for $\rho>0$. 
\iffalse
\begin{theorem} \label{Theorem: Barron regularity of the eigenfunctions}
Under Assumption \ref{Assumption: V_i, v_ij in mcFL_s^1+mcFL_s^^(alpha^prime), 2+s-|s|-n/alpha>0}, let $\psi \in H^{1}(\mb{R}^{nN})$ be the eigenfunction of $\mc{H}$ with the corresponding eigenvalue $\lambda$. Then, for any $\beta<1+(s-|s|)/2$ with $\beta>n/(2\alpha)$ if $\alpha<\infty$ or $\beta\ge0$ if $\alpha=\infty$, $\psi \in \mc{B}^{s+2-2\beta}(\mb{R}^{Nn})$ and 
\begin{equation*} %\label{ineq: norm estimate in Thm: regularity of eigenfunction when }
\begin{aligned}
& \left\|\psi\right\|_{\mc{B}^{s+2-2\beta}(\mb{R}^{Nn})} \le \tilde{\mu}_{1} \left[|\lambda+1| + \mc{C}(V)\right] \|\psi\|_{\mc{B}^{|s|}(\mb{R}^{nN})} ,\\
& \left\|\psi\right\|_{\mc{B}^{s+2-2\beta}(\mb{R}^{Nn})} 
\le   2^{|s|/2+nN/4}\sqrt{\frac{\omega_{nN}}{2|s|+nN}} \left[2\tilde{\mu}_{1} \left(|\lambda+1| + \mc{C}(V)\right)\right]^{1+\frac{|s|+nN/2}{2+s-|s|-2\beta}}  \|\psi\|_{L^{2}(\mb{R}^{nN})} ,
\end{aligned} 
\end{equation*}
where we denote the constant $\mc{C}(V) = \mc{C}(V;s,\alpha,\beta)$ as
\begin{equation}  \label{eq: operator norm of multiplying V}
\begin{aligned} 
&  2^{|s|/2} \sum_{i=1}^{N} \|V_{i}\|_{s,\alpha;\beta}  +  2^{|s|}\sum_{i<j} \|V_{ij}\|_{s,\alpha;\beta} + 2^{|s|/2}\left\|V_{\operatorname{ad}}\right\|_{\mc{F}L_{s}^{1}(\mb{R}^{nN})}  . 
\end{aligned}
\end{equation}
\end{theorem}
\fi
\begin{theorem} \label{Theorem: Barron regularity of the eigenfunctions}
Under Assumption \ref{Assumption: V_i, v_ij in mcFL_s^1+mcFL_s^^(alpha^prime), 2+s-|s|-n/alpha>0}, let $\psi \in H^{1}(\mb{R}^{nN})$ be the eigenfunction of $\mc{H}$ associated with the eigenvalue $\lambda$. For any $\gamma>\abs{s}$ with $\gamma<s-n/\alpha+2$ if $\alpha<\infty$ or $\gamma\le s+2$ if $\alpha=\infty$, $\psi \in \mc{B}^{\gamma}(\mb{R}^{Nn})$ and 
\begin{equation*} 
\begin{aligned}
& \|\psi\|_{\mc{B}^{\gamma}(\mb{R}^{Nn})} \le \tilde{\mu}_{1} \left[|\lambda+1| + \mc{C}(V)\right] \|\psi\|_{\mc{B}^{|s|}(\mb{R}^{nN})} ,\\
& \|\psi\|_{\mc{B}^{\gamma}(\mb{R}^{Nn})} 
\le   2^{\frac{2|s|+nN}{4}}\sqrt{\frac{\omega_{nN}}{2|s|+nN}} \left[2\tilde{\mu}_{1} \left(|\lambda+1| + \mc{C}(V)\right)\right]^{\frac{\gamma+nN/2}{\gamma-|s|}}\|\psi\|_{L^2(\mb{R}^{nN})},
\end{aligned} 
\end{equation*}
where we set $\beta = 1+(s-\gamma)/2$ and the constant $\mc{C}(V) = \mc{C}(V;s,\alpha,\beta)$ is denoted as
\begin{equation}  \label{eq: operator norm of multiplying V}
\begin{aligned} 
&  2^{|s|/2} \sum_{i=1}^{N} \|V_{i}\|_{s,\alpha;\beta}  +  2^{|s|}\sum_{i<j} \|V_{ij}\|_{s,\alpha;\beta} + 2^{|s|/2}\left\|V_{\operatorname{ad}}\right\|_{\mc{F}L_{s}^{1}(\mb{R}^{nN})}  . 
\end{aligned}
\end{equation}
\end{theorem}

The above theorem is quite general,  certain special cases are also interesting of their own rights.

Taking $\alpha=\infty$ in Theorem \ref{Theorem: Barron regularity of the eigenfunctions}, we obtain the following shift estimate in Barron spaces for eigenfunctions. 
\begin{theorem} \label{Theorem: Barron regularity lift of the eigenfunctions}
Assume $V_{i}, V_{ij} \in \mc{B}^{s}(\mb{R}^{n})$ and $V_{\operatorname{a d}} \in \mc{B}^{s}(\mb{R}^{Nn})$ for some $s>-1$. Let $\psi \in H^{1}(\mb{R}^{nN})$ be the eigenfunction of $\mc{H}$ with the corresponding eigenvalue $\lambda$. Then, $\psi \in \mc{B}^{s+2}(\mb{R}^{Nn})$ and 
\begin{equation*} %\label{ineq: norm estimate in Thm: regularity of eigenfunction when }
\begin{aligned}
& \left\|\psi\right\|_{\mc{B}^{s+2}(\mb{R}^{Nn})} 
\le  \tilde{\mu}_{1} \left[|\lambda+1| + \mc{C}\right]  \|\psi\|_{\mc{B}^{|s|}(\mb{R}^{nN})} ,\\
& \left\|\psi\right\|_{\mc{B}^{s+2}(\mb{R}^{Nn})} 
\le   2^{\frac{2|s|+nN}{4}}\sqrt{\frac{\omega_{nN}}{2|s|+nN}}  \left[2\tilde{\mu}_{1} \left(|\lambda+1| + \mc{C}\right)\right]^{\frac{s+2+nN/2}{s+2-|s|}}  \|\psi\|_{L^{2}(\mb{R}^{nN})} ,
\end{aligned} 
\end{equation*}
where 
\begin{equation*}  %\label{eq: operator norm of multiplying V}
\begin{aligned} 
\mc{C} & =   2^{|s|/2} \sum_{i=1}^{N} \|V_{i}\|_{\mc{B}^{s}(\mb{R}^{n})} +  2^{|s|}\sum_{i<j} \|V_{ij}\|_{\mc{B}^{s}(\mb{R}^{n})}  + 2^{|s|/2} \|V_{\operatorname{ad}}\|_{\mc{B}^{s}(\mb{R}^{nN})} .
\end{aligned}
\end{equation*}
\end{theorem}

Taking $s=0$ in Theorem \ref{Theorem: Barron regularity of the eigenfunctions}, we obtain
\begin{theorem} \label{Theorem: regularity of eigenfunction when Fourier of V_i, v_ij in L^1+L^(alpha^prime)}
Under Assumption \ref{Assumption: V_i, v_ij in mcFL_s^1+mcFL_s^^(alpha^prime), 2+s-|s|-n/alpha>0} with $s=0$, let $\psi \in H^{1}(\mb{R}^{nN})$ be the eigenfunction of $\mc{H}$ associated with the eigenvalue $\lambda$. Then, for any $0<\gamma<2-n/\alpha$, $\psi \in \mc{B}^{\gamma}(\mb{R}^{Nn})$ and 
\begin{equation*} \label{ineq: norm estimate in Thm: regularity of eigenfunction when Fourier of V_i, v_ij in L^1+L^(alpha^prime)}
\begin{aligned}
& \left\|\psi\right\|_{\mc{B}^{\gamma}(\mb{R}^{Nn})} 
\le \tilde{\mu}_{1} \left[|\lambda+1| + \mc{C}\right]\|\psi\|_{\mc{B}^{0}(\mb{R}^{Nn})}, \\
& \left\|\psi\right\|_{\mc{B}^{\gamma}(\mb{R}^{Nn})} \le 2^{nN/4}\sqrt{\frac{\omega_{nN}}{nN}} \left[2\tilde{\mu}_{1} \left(|\lambda+1| + \mc{C}\right)\right]^{1+\frac{nN}{2\gamma}}  \|\psi\|_{L^{2}(\mb{R}^{nN})} ,
\end{aligned} 
\end{equation*}
where $\mc{C} = \mc{C}(V;0,\alpha,1-\gamma/2)$ is defined in (\ref{eq: operator norm of multiplying V}). 
\end{theorem}

By~\cite[Corollary 2.13]{LiaoMing:2023}, $\mc{B}^s(\mb{R}^{d}) \hookrightarrow C^s(\mb{R}^{d})$, we recover Simon's estimate~\cite[Theorem $1^{\prime}$]{Simon1974Pointwise}.

As an application of Theorem \ref{Theorem: Barron regularity of the eigenfunctions}, the following corollary deals with a typical case where $V_{i}$, $V_{ij}$ are inverse power potentials. %(\ref{eq: constant before the Fourier transform of |x|^(-t)})
%The electronic Schrodinger equation corresponds to the power $t=1$.
%
\begin{corollary} \label{Corollary: regularity of eigenfunctions under inverse power potential}
Let $f(x)=|x|^{-t}$ be inverse power potential in $n$ dimension with $t\in(0,3/2+1_{n>1}/2)$, $V_{i} = \sum_{k = 1}^{M_{i}} b_{k} f(\cdot - a_{k})$, $V_{ij} = f$ and  $V_{\operatorname{a d}} \equiv 0$. Let $\psi$ be the eigenfunction of $\mc{H}$ associated with the eigenvalue $\lambda$. Denote the constants
$$\nu_{t, n} = \frac{2\pi^{t} |\Gamma((n-t) / 2)|}{\Gamma(t / 2)\Gamma(n / 2)}, \quad\mc{M} = \sum_{i = 1}^{N} \sum_{k = 1}^{M_{i}} |b_{k}| + N(N-1)/2 .$$
Then, $\psi \in \mc{B}^{\gamma}(\mb{R}^{Nn})$ for all $\gamma < 2-t$ and
\begin{equation*} \label{}
\begin{aligned}
%\left\|\psi\right\|_{\mc{B}^{\gamma}(\mb{R}^{Nn})} & \le  \tilde{\mu}_{1} \left[ \frac{2\pi^{t} \Gamma((n-t) / 2) \mc{M}}{\Gamma(t / 2)\Gamma(n / 2)} \left(\frac{1}{t} + \frac{2 + (\pi - 2)1_{\{n=1\}} }{2-t-\gamma} \right)  + |\lambda+1| \right]\|\psi\|_{\mc{B}^{0}(\mb{R}^{Nn})}, \\
\left\|\psi\right\|_{\mc{B}^{\gamma}(\mb{R}^{Nn})} & \le  \tilde{\mu}_{1} \left[ \nu_{t, n}  \mc{M} \left(\frac{1}{t} + \frac{2 + (\pi - 2)1_{\{n=1\}}}{2-t-\gamma} \right)  + |\lambda+1| \right]\|\psi\|_{\mc{B}^{0}(\mb{R}^{Nn})}, \quad \text{for } t<n, \\
\left\|\psi\right\|_{\mc{B}^{\gamma}(\mb{R}^{Nn})} & \le  \tilde{\mu}_{1} \left[  \frac{\pi \nu_{t, 1} \mc{M}}{2(2-t-\gamma)}  + |\lambda+1| \right]\|\psi\|_{\mc{B}^{t-1}(\mb{R}^{Nn})}, \quad \text{for } t>n,\\ 
\left\|\psi\right\|_{\mc{B}^{\gamma}(\mb{R}^{Nn})} 
& \le  \tilde{\mu}_{1} \left[ \mc{M} \left(3 + \frac{6}{1-\gamma} + \frac{8}{(1-\gamma)^{2}} \right)  + |\lambda+1| \right]\|\psi\|_{\mc{B}^{\frac{1-\gamma}{2}}(\mb{R}^{Nn})}, \quad \text{for } t=n, \\  
\end{aligned} 
\end{equation*}
where $\gamma>t-1$ for $t>n$ and $\gamma>1/3$ for $t=n=1$. 
As $\gamma$ tends to $2-t$, the prefactor diverges at a rate of $1/(2-t-\gamma)$, unless $t=n=1$, in which case the divergence rate is $1/(1-\gamma)^2$.

In particular,  for the $3$-dimensional Coulomb potential $f(x)=|x|^{-1}$,  the above estimate reduces to
\begin{equation*} \label{}
\begin{aligned}
& \left\|\psi\right\|_{\mc{B}^{\gamma}(\mb{R}^{Nn})} 
\le  \tilde{\mu}_{1} \left[ 4\mc{M} \left(1 + \frac{2}{1-\gamma} \right)   + |\lambda+1| \right]\|\psi\|_{\mc{B}^{0}(\mb{R}^{Nn})} .     
\end{aligned} 
\end{equation*}
and $\psi \in \mc{B}^{\gamma}(\mb{R}^{Nn})$ for all $\gamma < 1$. %with linear growth about $\lambda$.
We recover Yserentant's recent estimate \cite[Theorem 4.6]{Yserentant2025regularity}.
\end{corollary}

\iffalse
\begin{example}
The Yuwaka potential $f(x) = |x|^{-1}\mr{e}^{-\mu |x|}$ in dimension $n=3$ has Fourier transform
\begin{equation*}
\wh{f}(\xi) =  \int_{\mb{R}^{3}} |x|^{-1}\mr{e}^{-\mu |x|-2\pi i\xi\cdot x}\mr{d}x = \frac{4\pi}{\mu^{2}+4\pi^2|\xi|^{2}},
\end{equation*}
where $\mu>0$ usually denotes the mass of the exchanged meson. Similar as Coulomb potential, by the decomposition $\wh{f} = \wh{f}1_{\{|\xi|\le 1\}} + \wh{f}1_{\{|\xi|>1\}}$, $f \in \mc{F}L_{s}^{1}(\mb{R}^{3})+\mc{F}L_{s}^{\alpha^{\prime}}(\mb{R}^{3})$ with $\alpha^{\prime}>3/2$.
\end{example}
\fi%

In the following example, we show the sharpness of Theorem \ref{Theorem: Barron regularity of the eigenfunctions} near the borderline case.
\begin{example}[Sharpness of Theorem \ref{Theorem: Barron regularity of the eigenfunctions}] \label{Example: Sharpness of regularity of eigenfunctions}
Let $N=1$, $n\ge2$ and $\mu_{1}=1$. 
For $0<\delta<1$, let 
\[
V(x) = \delta^{2}|x|^{2\delta-2}/2 - \delta(n+\delta-2)|x|^{\delta-2}/2.
\]
For $\delta=1$, let $V(x) = (n-1)/(2\abs{x})$. $V$ is a linear combination of inverse power potentials, with the $|x|^{\delta-2}$ term exhibiting a stronger singularity. Lemma \ref{Lemma: index condition s<n/alpha-t for |x|^(-t) to satisfy Assumption} shows that $V$ satisfies Assumption~\ref{Assumption: V_i, v_ij in mcFL_s^1+mcFL_s^^(alpha^prime), 2+s-|s|-n/alpha>0} with $s-n/\alpha < \delta-2.$. Then, Theorem \ref{Theorem: Barron regularity of the eigenfunctions} ensures that all the eigenfunctions lie in $\mc{B}^{\gamma}(\mb{R}^{n})$ for $\gamma<\delta$. 

On the other hand, a direct calculation verifies that $\psi(x) = e^{-\abs{x}^{\delta}}$ is an eigenfunction of $\mc{H}$ associated with eigenvalue $0$ when $\delta<1$ or $-1/2$ when $\delta=1$. Lemma \ref{Lemma: decay rate of Fourier transform of e^(|x|^delta)} reveals that $\wh{\psi}$ decays exactly at a rate of $\la\xi\ra^{-\delta-n}$, leading to $\psi\in\mc{B}^{\gamma}$ for all $\gamma<\delta$ but $\psi\notin\mc{B}^{\gamma}(\mb{R}^{n})$ for all $\gamma\ge\delta$. Moreover, as $\gamma$ approaches $\delta$, $\|\psi\|_{\mc{B}^{\gamma}(\mb{R}^{n})}$ blows up at a rate of $1/(\delta-\gamma)$, resembling the estimate in Corollary \ref{Corollary: regularity of eigenfunctions under inverse power potential}. 

\iffalse
When  $V$ satisfies Assumption \ref{Assumption: V_i, v_ij in mcFL_s^1+mcFL_s^^(alpha^prime), 2+s-|s|-n/alpha>0} with $s-n/\alpha < -1.$ Theorem \ref{Theorem: Barron regularity of the eigenfunctions} ensures that all the eigenfunctions lie in $\mc{B}^{\gamma}(\mb{R}^{n})$ for $\gamma < 1$. On the other hand, a direct calculation gives that $\psi(x) = e^{-|x|}$ is an eigenfunction of $\mc{H}$ associated with eigenvalue $-1/2$. 
By~\cite[Theorem 1.14 with $\al=1/(2\pi)$]{SteinWeiss:1971},  we have the expilicit formula for the Fourier transform of $\psi$ as
\[
\wh{\psi}(\xi) = 2^{n}\pi^{(n-1)/2} \delta((n+1)/2) \left(1+4 \pi^{2}|\xi|^{2}\right)^{-(n+1)/2}.
\]
It is straightforward to verify that $\psi\in\mc{B}^t(\mb{R}^{n})$ for all $t< 1$, but
$\psi\notin\mc{B}^t(\mb{R}^{n})$ for all $t\ge1$.\fi 
\end{example}

Now, with the aid of Figure \ref{Figure: Visualize main results of this paper}, we  summarize and further discuss the obtained regularity results. 
Combined with the understanding of the left image in Figure \ref{Figure: Visualize main results of this paper}, the right image shows the results of this paper intuitively. Firstly, the shaded area above the black solid line corresponds to the range of potentials addressed in this paper, as Assumption \ref{Assumption: V_i, v_ij in mcFL_s^1+mcFL_s^^(alpha^prime), 2+s-|s|-n/alpha>0}. Under Assumption \ref{Assumption: V_i, v_ij in mcFL_s^1+mcFL_s^^(alpha^prime), 2+s-|s|-n/alpha>0}, Theorem \ref{Theorem: Barron regularity of the eigenfunctions} shows $\psi \in \mc{B}^{\gamma}(\mb{R}^{Nn})$ with $\gamma<s-n/\alpha+2$. When $\alpha<\infty$, the regularity of eigenfunctions is consistent along each dotted line. The eigenfunction's Barron smoothness is at least $2$ (resp. $1$) if the potential lies above the dashed magenta (resp. blue) line, and only in $(0,1)$ if the potential lies in the orange area below the blue line, e.g., the  Coulomb  potential. The orange region, as shown in Example \ref{Example: Sharpness of regularity of eigenfunctions}, confirms the sharpness of our estimated regularity index $s-n/\alpha+2$.

We claim that for $\alpha>n/2$, the condition $2+2s-n/\alpha>0$  in Assumption \ref{Assumption: V_i, v_ij in mcFL_s^1+mcFL_s^^(alpha^prime), 2+s-|s|-n/alpha>0} is necessary. This is due to the convolution $\wh{V}*\wh{\psi}$ in the momentum representation. Suppose that $V$ satisfies Assumption \ref{Assumption: V_i, v_ij in mcFL_s^1+mcFL_s^^(alpha^prime), 2+s-|s|-n/alpha>0} with $(\alpha,s)$ below the black line and $\psi\in\mc{B}^{\gamma}$. By \cite[Proposition 2.7]{Toft2015Sharp}, the convolution $\wh{V}*\wh{\psi}$ is well defined only if $s+\gamma\ge0$. However, by Example \ref{Example: Sharpness of regularity of eigenfunctions}, generally, the Barron smoothness $\gamma$ is at most $$[n/(2\alpha) - 1] - n/\alpha + 2 = 1 - n/(2\alpha),$$ leading to $s+\gamma = s+1-n/(2\alpha)<0$. This implies that proper convolution estimates are not feasible below the black line for Schr\"odinger eigenfunctions. 

Furthermore, in the orange area, it is necessary to work in  Fourier-Lebesgue spaces with small $\alpha$. The points on the red dashed line exemplify this need. If we merely establish shift estimates in Barron spaces, to achieve Barron regularity of eigenfunctions, we need to embed the potentials within the points on the red dashed line along the dashed line into the Barron space with a smoothness index $<s-n/\alpha=-1.5$ by Proposition \ref{Proposition: roles of s and alpha in the assumption and their relationship} (3). However, the red dashed line will definitely  intersect the black line and descend below it, precluding the desired regularity results. Therefore, it is imperative to characterize these potentials in Fourier-Lebesgue spaces with $\alpha<\infty$ and directly establish regularity estimates on them.
\subsection{Solvability of Schr\"odinger equations}
Now, we show the solvability for the weak solution to $(\mc{H}+\rho I)u = f$. For given $\rho\in\mb{R}$ and $f\in H^{-1}(\mb{R}^{nN})$, a function $u\in H^{1}(\mb{R}^{nN})$ is a weak solution if %of the operator equation $(\mc{H}+\rho I)u = f$ if
\begin{equation*}
a(u,v) + (\rho u,v)_{L^{2}} = (f,v)_{L^{2}}, \qquad \text{for all \quad $v\in H^{1}(\mb{R}^{nN})$.}
\end{equation*}
\begin{theorem} \label{Theorem: solvability of (mcH+rho I)u = f under two assumptions}
Under Assumption \ref{Assumption: V_i, v_ij in mcFL_s^1+mcFL_s^^(alpha^prime), 2+s-|s|-n/alpha>0}, let $\gamma>\abs{s}$ with $\gamma<s-n/\alpha+2$ if $\alpha<\infty$ or $\gamma\le s+2$ if $\alpha=\infty$. Let $\rho$ satisfy (\ref{ineq: lower bd for rho to ensure coerciveness of a_rho under Assumption: V_i, v_ij, V_ad in mathcalB^s}). Then, for any $f\in \mc{B}^{\gamma-2}(\mb{R}^{nN})\cap H^{-1}(\mb{R}^{nN})$, there exists a unique weak solution $u^{*}\in\mc{B}^{\gamma}(\mb{R}^{nN})\cap H^{1}(\mb{R}^{nN})$. Furthermore, $\|u^{*}\|_{H^{1}(\mb{R}^{nN})}\lesssim \|f\|_{H^{-1}(\mb{R}^{nN})}$ and 
\begin{equation*}
\begin{aligned}
\|u^{*}\|_{\mc{B}^{\gamma}(\mb{R}^{nN})}
& \lesssim \|f\|_{\mc{B}^{\gamma-2}(\mb{R}^{nN})} + \|f\|_{H^{-1}(\mb{R}^{nN})},
\end{aligned}
\end{equation*}
where the implied constants depend only on $\gamma$, $s$, $\alpha$, $n$, $N$, $\mu_{i}$, $\rho$, $V$ explicitly.
\end{theorem}
Theorem \ref{Theorem: solvability of (mcH+rho I)u = f under two assumptions} only requires the weakest Assumption \ref{Assumption: V_i, v_ij in mcFL_s^1+mcFL_s^^(alpha^prime), 2+s-|s|-n/alpha>0} in this paper for $V$, which greatly expands the scope of $V$ compared to \cite{Chen2023regularity}. To prove Theorem \ref{Theorem: solvability of (mcH+rho I)u = f under two assumptions}, since $f\in H^{-1}$, by the boundedness and coercivenes we proved for $a_{\rho}$, Lax-Milgram theorem ensures a unique weak solution $u^{*}\in H^{1}(\mb{R}^{nN})$. Next, we use the already derived multiplier estimates for $V$ to construct a similar contraction operator representation as in the proof of Theorem \ref{Theorem: Barron regularity of the eigenfunctions}, which in turn gives the Barron regularity of $u^{*}$.

If we relax the assumptions on $f$ and $\rho$, and understand the operator equation $(\mc{H}+\rho I)u=f$ in the distributional sense, then solvability can be obtained when $V\in\mc{B}^{s}(\mb{R}^{nN})$.
\begin{theorem} \label{Theorem: solvability of (mcH+rho I)u = f when V in mcB^s(mbR^(nN))}
%Consider the operator equation $(\mc{H}+\rho I)u=f$. 
Let $V\in\mc{B}^{s}(\mb{R}^{nN})$ with $s >-1$ and $\rho>0$. Then exactly one of the following is true: \\
(1) The homogeneous problem $(\mc{H}+\rho I)u=0$ has a nonzero solution $u \in \mc{B}^{s+2}(\mb{R}^{nN})$. \\
(2) For any $f \in \mc{B}^{s}(\mathbb{R}^{nN})$, there exists a unique solution $u^{*} \in \mc{B}^{s+2}(\mb{R}^{nN})$. Additionally, $\|u^{*}\|_{\mc{B}^{s+2}(\mb{R}^{nN})} \le C(\mc{H},\rho) \|f\|_{\mc{B}^{s}(\mb{R}^{nN})}.$ 
\end{theorem}
The above solvability result is purely in Barron spaces. We prove Theorem \ref{Theorem: solvability of (mcH+rho I)u = f when V in mcB^s(mbR^(nN))} by applying the Fredholm alternative theorem to the equivalent equation $u + \mc{R}u = (\mc{H}_{0}+\rho I)^{-1}f$, where $\mc{H}_{0} = \mc{H}-V$ is the $N$-body free quantum Hamiltonian. To this end, we prove in Proposition \ref{Proposition: mcR is compact on mcB^(s+2)(mbR^(Nn))} that $\mc{R}=(\mc{H}_{0}+\rho I)^{-1}V$ is a compact operator on $\mc{B}^{s+2}(\mb{R}^{Nn})$.
The proof of Theorem \ref{Theorem: solvability of (mcH+rho I)u = f when V in mcB^s(mbR^(nN))} adopts an idea similar as \cite{Chen2023regularity}, while we deal with the case of Barron spaces with negative indices. Since $\mc{B}^{s}(\mb{R}^{Nn})$ with $s<0$ contains unbounded functions with singularities, there is no reason to assume that $V$ has an upper or lower bound in our theorem. 

\begin{remark}[Generalizability towards slightly more general situations]
Let $\tilde{\alpha} = N\alpha$ and $1/\tilde{\alpha}+1/\tilde{\alpha}^{\prime}=1$. It can be seen from the proof that we can extend the assumption of $V_{\operatorname{ad}}$ to $\mc{F}L_{s}^{1}(\mb{R}^{Nn})+\mc{F}L_{s}^{\tilde{\alpha}^{\prime}}(\mb{R}^{Nn})$, with all regularity results in Theorem \ref{Theorem: Barron regularity of the eigenfunctions} and Theorem \ref{Theorem: solvability of (mcH+rho I)u = f under two assumptions} remaining unchanged. Only the norm of the $\mc{F}L_{s}^{\tilde{\alpha}^{\prime}}(\mb{R}^{Nn})$ part needs to be added to the constant $\mc{C}$ in a similar way. For the sake of brevity, we leave this to interested readers.

As in \cite{Simon1974Pointwise}, one can generalize the form of terms $V_{i}$, $V_{ij}$ to $V_{\kappa}(P_{\kappa}x)$ with $\kappa$ in a finite index set and each $P_{\kappa}$ a projection to an $n$-dimensional subspace by applying appropriate linear transformations.
\end{remark}

\section{Barron regularity of eigenfunctions}  \label{Section: Barron regularity of eigenfunctions}
\subsection{Preliminaries for Fourier-Lebesgue spaces}
In this part, we clarify the roles of $s$ and $\alpha$ in Assumption \ref{Assumption: V_i, v_ij in mcFL_s^1+mcFL_s^^(alpha^prime), 2+s-|s|-n/alpha>0} by proving Proposition \ref{Proposition: roles of s and alpha in the assumption and their relationship}. 
\begin{proof}[Proof of Proposition \ref{Proposition: roles of s and alpha in the assumption and their relationship}]
(1) The embedding directly follows from the definition. 

(2) Let $f = f_{1} + f_{2}$ with $f_{1}\in\mc{F}L_{s}^{1}(\mb{R}^{n})$ and $f_{2}\in\mc{F}L_{s}^{p}(\mb{R}^{n})$. Denote $g = \la\cdot\ra^{s}\wh{f}_{2} \in L^{p}(\mb{R}^{n})$, $\kappa = \|g\|_{L^{p}(\mb{R}^{n})}$ and $S = \{\xi\in\mb{R}^{n}:|g(\xi)|>\kappa\}$. Then, it suffices to prove $g1_{S}\in L^{1}(\mb{R}^{n})$ and $g1_{S^{c}} = g - g1_{S}\in L^{r}(\mb{R}^{n}).$ %$S^{c} = \{\xi\in\mb{R}^{n}:|g(\xi)|\le\kappa\}$ 
By direct calculations,
\begin{equation*}
\begin{aligned}
\|g1_{S}\|_{L^{1}(\mb{R}^{n})} & %= \int_{S} |g(\xi)|\mr{d}\xi 
\le \kappa^{1-p}\int_{S} |g(\xi)|^{p}\mr{d}\xi
\le \kappa^{1-p}\int_{\mb{R}^{n}} |g(\xi)|^{p}\mr{d}\xi = \kappa.  
\end{aligned}
\end{equation*}
If $r=\infty$, $\|g1_{S^{c}}\|_{L^{r}(\mb{R}^{n})}\le\kappa$. If $r<\infty$, $\|g1_{S^{c}}\|_{L^{r}(\mb{R}^{n})} \le \|g\|_{L^{p}(\mb{R}^{n})}$ follows from
\begin{equation*}
\begin{aligned}
\|g1_{S^{c}}\|_{L^{r}(\mb{R}^{n})}^{r} & %= \int_{S^{c}} |g(\xi)|^{r}\mr{d}\xi
\le \kappa^{r-p}\int_{S^{c}} |g(\xi)|^{p}\mr{d}\xi
\le \kappa^{r-p}\int_{\mb{R}^{n}} |g(\xi)|^{p}\mr{d}\xi = \kappa^{r}. 
\end{aligned}
\end{equation*}

(3) When $s_{2}-n/\alpha_{2}<s_{1}-n/\alpha_{1}$, let $f = f_{1} + f_{2}$ with $f_{1}\in\mc{F}L_{s_{1}}^{1}(\mb{R}^{n})$ and $f_{2}\in\mc{F}L_{s_{1}}^{\alpha_{1}^{\prime}}(\mb{R}^{n})$. Since $\alpha_{1}\le\alpha_{2}$, we have $s_{2}<s_{1}$ and $f_{1}\in\mc{F}L_{s_{2}}^{1}(\mb{R}^{n}).$ Denote $g = \la\cdot\ra^{s_{1}}\wh{f}_{2}$ and it suffices to prove $\la\cdot\ra^{s_{2}-s_{1}}g\in L^{\alpha_{2}^{\prime}}(\mb{R}^{n}).$ For $\tau:= 1/\alpha_{2}^{\prime}-1/\alpha_{1}^{\prime}<(s_{1}-s_{2})/n$, H\"older's inequality yields 
\begin{equation*}
\begin{aligned}
\|\la\cdot\ra^{s_{2}-s_{1}}g\|_{L^{\alpha_{2}^{\prime}}(\mb{R}^{n})} & \le \|\la\cdot\ra^{s_{2}-s_{1}}\|_{L^{1/\tau}(\mb{R}^{n})}\|g\|_{L^{\alpha_{1}^{\prime}}(\mb{R}^{n})} . 
\end{aligned}
\end{equation*}
\iffalse
When $s_{2}-n/\alpha_{2}>s_{1}-n/\alpha_{1}$, we consider $f$ such that $\wh{f} = \la\cdot\ra^{t}$ with $-s_{2}-n/\alpha_{2}^{\prime}<t<-s_{1}-n/\alpha_{1}^{\prime}$. A direct calculation yields $f\in\mc{F}L_{s_{1}}^{\alpha_{1}^{\prime}}(\mb{R}^{n})$. Suppose $f\in\mc{F}L_{s_{2}}^{1}(\mb{R}^{n})+\mc{F}L_{s_{2}}^{\alpha_{2}^{\prime}}(\mb{R}^{n})$. We choose $t_{2}$ such that $-t-n<t_{2}<s_{2}-n/\alpha_{2}$. From the first part of (3), $f\in\mc{F}L_{t_{2}}^{1}(\mb{R}^{n})$. However, since $t + t_{2}>-n$, we have $\|\la\cdot\ra^{t_{2}}\wh{f}\|_{L^{1}(\mb{R}^{n})} = \infty$, which leads to a contradiction. Hence, $f\notin\mc{F}L_{s_{2}}^{1}(\mb{R}^{n})+\mc{F}L_{s_{2}}^{\alpha_{2}^{\prime}}(\mb{R}^{n})$.
\fi

When $s_{2}-n/\alpha_{2}\ge s_{1}-n/\alpha_{1}$, by (1) proved above, $\mc{F}L_{s_{2}}^{1}(\mb{R}^{n})+\mc{F}L_{s_{2}}^{\alpha_{2}^{\prime}}(\mb{R}^{n})$ monotonically decreases as $s_{2}$ increases. It suffice to prove (\ref{eq: embedding relationship FL_(s_1)^1+FL_(s_1)^(alpha_1^') not in FL_(s_2)^1+FL_(s_2)^(alpha_2^') if alpha_1<alpha_2 and s_2-n/alpha_2 = s_1-n/alpha_1}) for $s_{2}-n/\alpha_{2}=s_{1}-n/\alpha_{1}$. 
Since the rescaled norms defined in (\ref{def: rescaled norms for FL_s^1+FL_s^(alpha')}) are equivalent for different $c>0$, we only need to consider the norm
\begin{equation*} 
\|f\|_{s,\alpha} := \inf\{ \|g\|_{\mc{F}L_{s}^{1}(\mb{R}^{n})} + \|f-g\|_{\mc{F}L_{s}^{\alpha^{\prime}}(\mb{R}^{n})} : g\in \mc{F}L_{s}^{1}(\mb{R}^{n})\} .
\end{equation*}
We construct $\{f_{k}\}_{k=1}^{\infty}$ so that $\wh{f}_{k}(\xi)=\varepsilon_{k}^{-1/\alpha_{1}^{\prime}}\la\xi\ra^{-s_{1}-n / \alpha_{1}^{\prime}} 1_{\{|\xi| \le k\}}$ with $\varepsilon_{k} = \int_{|\xi| \le k}\la\xi\ra^{-n}\mr{d}\xi.$ 
We shall prove that $\|f_{k}\|_{s_{1},\alpha_{1}} \le 1$ and $\|f_{k}\|_{s_{2},\alpha_{2}} \to \infty$ as $k\to\infty$. 
It follows from $\|\la\cdot\ra^{s_{1}}\wh{f}_{k}\|_{L^{\alpha_{1}^{\prime}}(\mb{R}^{n})}\le1$ that $\|f_{k}\|_{s_{1},\alpha_{1}} \le 1$.
To lower bound $\|f_{k}\|_{s_{2},\alpha_{2}}$, we observe that
\begin{equation*} 
\begin{aligned}
\|\la\cdot\ra^{s_{2}}g\|_{L^{1}(\mb{R}^{n})} + \|\la\cdot\ra^{s_{2}}(\wh{f}_{k}-g)\|_{L^{\alpha_{2}^{\prime}}(\mb{R}^{n})}
\end{aligned}
\end{equation*}
does not increase if we replace $g$ with $\operatorname{Re}(g) = (g+\bar{g})/2$, or replace a real-valued function $g$ with $h = g1_{\{0\le g\le\wh{f}_{k}\}}+\wh{f}_{k}1_{\{g>\wh{f}_{k}\}}$ because $|h(\xi)|\le|g(\xi)|$ and $|\wh{f}_{k}(\xi)-h(\xi)|\le |\wh{f}_{k}(\xi)-g(\xi)|$ for all $\xi$. Hence, 
\begin{equation*} 
\begin{aligned}
\|f_{k}\|_{s_{2},\alpha_{2}} %& = \inf_{g} \left\{ \|g\|_{\mc{F}L_{s_{2}}^{1}(\mb{R}^{n})} + \|f-g\|_{\mc{F}L_{s_{2}}^{\alpha_{2}^{\prime}}(\mb{R}^{n})} \right\} \\
%& = \inf_{g} \left\{ \|\la\cdot\ra^{s_{2}}g\|_{L^{1}(\mb{R}^{n})} + \|\la\cdot\ra^{s_{2}}(\wh{f}_{k}-g)\|_{L^{\alpha_{2}^{\prime}}(\mb{R}^{n})} \right\} \\
& = \inf \left\{ \|\la\cdot\ra^{s_{2}}g\|_{L^{1}(\mb{R}^{n})} + \|\la\cdot\ra^{s_{2}}(\wh{f}_{k}-g)\|_{L^{\alpha_{2}^{\prime}}(\mb{R}^{n})}: g \text{ is real-valued}, 0\le g\le \wh{f}_{k} \right\}.
\end{aligned}
\end{equation*}
We consider $k$ large enough so that $\varepsilon_{k}\ge1$. %Then, $0\le\wh{f}_{k}\le1.$
For any real-valued $g$ with $0\le g\le \wh{f}_{k}$, we have $$0\le\la\xi\ra^{s_{2}}g \le \varepsilon_{k}^{-1/\alpha_{1}^{\prime}}\la\xi\ra^{-n / \alpha_{2}^{\prime}} 1_{\{|\xi| \le k\}} \le 1.$$
Note that $1\le\alpha_{2}^{\prime}<\infty.$ If $\|\la\cdot\ra^{s_{2}}g\|_{L^{\alpha_{2}^{\prime}}(\mb{R}^{n})}>1$,
\begin{equation*} 
\begin{aligned}
\|\la\cdot\ra^{s_{2}}g\|_{L^{1}(\mb{R}^{n})} & 
\ge \|\la\cdot\ra^{s_{2}}g\|_{L^{\alpha_{2}^{\prime}}(\mb{R}^{n})}^{\alpha_{2}^{\prime}} \ge \|\la\cdot\ra^{s_{2}}g\|_{L^{\alpha_{2}^{\prime}}(\mb{R}^{n})},
\end{aligned}
\end{equation*}
which gives 
\begin{equation*} 
\begin{aligned}
\|\la\cdot\ra^{s_{2}}g\|_{L^{1}(\mb{R}^{n})} + \|\la\cdot\ra^{s_{2}}(\wh{f}_{k}-g)\|_{L^{\alpha_{2}^{\prime}}(\mb{R}^{n})} & \ge \|\la\cdot\ra^{s_{2}}g\|_{L^{\alpha_{2}^{\prime}}(\mb{R}^{n})} + \|\la\cdot\ra^{s_{2}}(\wh{f}_{k}-g)\|_{L^{\alpha_{2}^{\prime}}(\mb{R}^{n})} \\
& \ge \|\la\cdot\ra^{s_{2}}\wh{f}_{k}\|_{L^{\alpha_{2}^{\prime}}(\mb{R}^{n})}.
\end{aligned}
\end{equation*}
If $\|\la\cdot\ra^{s_{2}}g\|_{L^{\alpha_{2}^{\prime}}(\mb{R}^{n})}\le1$, we have 
\begin{equation*} 
\begin{aligned}
\|\la\cdot\ra^{s_{2}}g\|_{L^{1}(\mb{R}^{n})} + \|\la\cdot\ra^{s_{2}}(\wh{f}_{k}-g)\|_{L^{\alpha_{2}^{\prime}}(\mb{R}^{n})} & %\ge \|\la\cdot\ra^{s_{2}}(\wh{f}_{k}-g)\|_{L^{\alpha_{2}^{\prime}}(\mb{R}^{n})} 
\ge \|\la\cdot\ra^{s_{2}}\wh{f}_{k}\|_{L^{\alpha_{2}^{\prime}}(\mb{R}^{n})} - 1.
\end{aligned}
\end{equation*}
A combination of the above two estimates yields
\begin{equation*} 
\begin{aligned}
\|f_{k}\|_{s_{2},\alpha_{2}} & \ge \|\la\cdot\ra^{s_{2}}\wh{f}_{k}\|_{L^{\alpha_{2}^{\prime}}(\mb{R}^{n})} - 1 = \varepsilon_{k}^{1/\alpha_{2}^{\prime}-1/\alpha_{1}^{\prime}}-1.
\end{aligned}
\end{equation*}
Since $1/\alpha_{2}^{\prime}-1/\alpha_{1}^{\prime}>0$ and $\varepsilon_{k}\to\infty$ as $k
\to\infty$, we complete the proof of (3).

(4) For $k\in\mb{N}_{+},$ take $\delta_{k}>0$ such that $(k+\delta_{k})^{n}-k^{n} = k^{-1}.$ We define $f$ by
\begin{equation*}
\begin{aligned}
\wh{f}(\xi) = \sum_{k=1}^{\infty} \la\xi\ra^{-(s+t)/2} 1_{\{k\le|\xi|<k+\delta_{k}\}}(\xi). 
\end{aligned}
\end{equation*}
Since $s<t$, we check $f\in\mc{B}^{s}(\mb{R}^{n})$ by
\begin{equation*}
\begin{aligned}
\|f\|_{\mc{B}^{s}(\mb{R}^{n})} \le \sum_{k=1}^{\infty}\int_{k\le|\xi|<k+\delta_{k}}\la\xi\ra^{(s-t)/2} \mr{d}\xi \le \frac{\omega_{n}}{n}\sum_{k=1}^{\infty}k^{(s-t)/2-1} <\infty. 
\end{aligned}
\end{equation*}
Suppose $f\in \mc{F}L_{t}^{1}(\mb{R}^{n})+\mc{F}L_{t}^{\infty}(\mb{R}^{n})$, $f = f_{1} + f_{2}$ with $g_{1} = \la\cdot\ra^{t}\wh{f}_{1} \in L^{1}(\mb{R}^{n})$ and $g_{2} = \la\cdot\ra^{t}\wh{f}_{2} \in L^{\infty}(\mb{R}^{n})$. Let $g = g_{1} + g_{2}.$ There exists $M\in\mb{N}$ such that $\|g_{2}\|_{L^{\infty}(\mb{R}^{n})}<\la M\ra^{(t-s)/2}-1$. Then, it follows from $|g_{1}(\xi)| \ge |g(\xi)|-|g_{2}(\xi)|$ that 
\begin{equation*}
\begin{aligned}
\|g_{1}\|_{L^{1}(\mb{R}^{n})} & \ge \sum_{k=M}^{\infty}\int_{k\le|\xi|<k+\delta_{k}}|g_{1}(\xi)|\mr{d}\xi \\
& \ge \sum_{k=M}^{\infty}\int_{k\le|\xi|<k+\delta_{k}} \left(\la M\ra^{(t-s)/2}-\|g_{2}\|_{L^{\infty}(\mb{R}^{n})}\right) \mr{d}\xi \\
& \ge \sum_{k=M}^{\infty}k^{-1} = \infty ,
\end{aligned}
\end{equation*}
which contradicts $g_{1} \in L^{1}(\mb{R}^{n})$. Hence, $f\notin \mc{F}L_{t}^{1}(\mb{R}^{n})+\mc{F}L_{t}^{\infty}(\mb{R}^{n}).$ 

(5) For $n\ge2$, we have checked $f(x) = |x|^{-1}\in \mc{F}L_{0}^{1}(\mb{R}^{n})+\mc{F}L_{0}^{\frac{n}{n-1}}(\mb{R}^{n})$ in the proof of Corollary \ref{Corollary: regularity of eigenfunctions under inverse power potential}. Since $\wh{f}(\xi) = c_{1,n} |\xi|^{1-n}$ with $c_{1,n} = \pi^{(1-n)/2} \Gamma((n-1) / 2)$, we calculate the norm
\begin{equation*}
\begin{aligned}
\|f\|_{\mc{B}^{-1}(\mb{R}^{n})} & = c_{1,n} \int_{\mb{R}^{n}}\la\xi\ra^{-1}|\xi|^{1-n}\mr{d}\xi \ge c_{1,n} \int_{\mb{R}^{n}}\la\xi\ra^{-n}\mr{d}\xi =\infty, \\
\end{aligned}
\end{equation*}
which implies $|x|^{-1}\notin\mc{B}^{-1}(\mb{R}^{n})$. Similarly, for $n=1$, let $f(x) = x^{-1}$ and then $\wh{f}(\xi) = -i\pi \operatorname{sgn}(\xi).$ Direct calculations yield $x^{-1}\in \mc{F}L_{0}^{\infty}(\mb{R})$ and $x^{-1}\notin\mc{B}^{-1}(\mb{R})$, which completes the proof. 
\end{proof}

\subsection{Form-boundedness of \texorpdfstring{$V$}{V}}
In this part, we firstly prove the form-boundedness of $V$ with respect to $-\Delta.$ Under Assumption \ref{Assumption: V_i, v_ij in mcFL_s^1+mcFL_s^^(alpha^prime), 2+s-|s|-n/alpha>0}, we always decompose $V_{i}$, $V_{ij}$ into two parts $V_{i} = V_{i,1} + V_{i,2}$, $V_{ij} = V_{ij,1} + V_{ij,2}$ such that $V_{i,1},V_{ij,1} \in \mc{F}L_{s}^{1}(\mb{R}^{n})$ and $V_{i,2},V_{ij,2}\in\mc{F}L_{s}^{\alpha^{\prime}}(\mb{R}^{n})$. We shall use the followimg basic fact later on: For $\gamma>n/2$,
\begin{equation}\label{eq: L^(2gamma) norm of <k>^{-1}}
\left\|\la \cdot \ra^{-1}\right\|_{L^{2\gamma}(\mb{R}^{n})}^{2\gamma} = \pi^{n / 2}  \Gamma(\gamma - n/2)/\Gamma(\gamma).
\end{equation}

We shall use the following elementary inequality, and we refer to~\cite[Lemma 6.10]{Folland:1995} for a proof.
\begin{lemma}
For any $x,y\in\mb{R}^d$ with $d\in\mathbb{N}$ and $s\in\mb{R}$, there holds
\begin{equation} \label{ineq: <xi> le 2^(|s|/2) <eta>^s <xi-eta>^|s|}
\la x\ra^{s} \le 2^{|s|/2}\la y\ra^{s} \la x-y\ra^{|s|}.
\end{equation}  
\end{lemma}

\iffalse 
\begin{proof}
For $s\le 0$, it follows from $\la y\ra^{2} \le 1+2\abs{x}^{2}+2\abs{x-y}^2\le 2\la x\ra^{2}\la x-y\ra^{2}$ that
\[
\la x\ra^{s} \le 2^{-s/2} \la y\ra^{s}\la x-y\ra^{-s}.
\]

For $s\ge 0$, it follows from $\la x\ra^{2} \le 1+2\abs{x-y}^2+2\abs{y}^2\le 2\la x-y\ra^{2}\la y\ra^{2}$ that
\[
\la x\ra^{s} \le 2^{s/2}\la y\ra^{s} \la x-y\ra^{s}.
\]
A combination of the above two inequalities gives
~\eqref{ineq: <xi> le 2^(|s|/2) <eta>^s <xi-eta>^|s|}.
\end{proof}\fi
%
\begin{lemma} \label{Lemma: boundedness of the quadratic form, part of V}
Under Assumption \ref{Assumption: V_i, v_ij in mcFL_s^1+mcFL_s^^(alpha^prime), 2+s-|s|-n/alpha>0}, let $2t+s-|s|-n/\alpha>0$ if $\alpha<\infty$ and $2t+s-|s|\ge0$ if $\alpha=\infty$. Then, for $u,v \in H^{1}(\mb{R}^{nN})$,
\begin{equation*}
\begin{aligned}
\left|\int_{\mb{R}^{nN}}Vuv\,\mr{d}x\right| & \le \mk{C}(V)\|u\|_{H^{t}(\mb{R}^{nN})} \|v\|_{H^{t}(\mb{R}^{nN})},  
\end{aligned}
\end{equation*}
where $\mk{C}(V) = \mc{C}(V;(s-|s|)/2,\alpha,t+(s-|s|)/2)$ as defined in (\ref{eq: operator norm of multiplying V}). 
\iffalse
\begin{equation*}
\begin{aligned}
\mk{C}(V) & = 2^{(|s|-s)/4}\sum_{i=1}^{N} \left(\|V_{i,1}\|_{\mc{F}L_{s}^{1}(\mb{R}^{n})} + c_{(2t+s-|s|)\alpha/2}^{1/\alpha} \|V_{i,2}\|_{\mc{F}L_{s}^{\alpha^{\prime}}(\mb{R}^{n})}\right) +  2^{(|s|-s)/4}\sum_{i<j} \left(\|V_{ij,1}\|_{\mc{F}L_{s}^{1}(\mb{R}^{n})} \right. \\
& \quad \left. + c_{(2t+s-|s|)\alpha/2}^{1/\alpha} \|V_{ij,2}\|_{\mc{F}L_{s}^{\alpha^{\prime}}(\mb{R}^{n})}\right) + 2^{(|s|-s)/4} \|V_{\operatorname{a d}}\|_{\mc{F}L_{s}^{1}(\mb{R}^{n})} .
\end{aligned}
\end{equation*}
\fi
\end{lemma}

\begin{proof}
For any fixed $1\le i\le N$, we rearrange $\xi=(\xi_{i},\xi^{\prime})^{\top}$. By Parseval’s identity and Fubini's theorem,
\begin{equation} \label{eq: Parseval’s relation for the part containing V_(i,2) in quadratic form}
\begin{aligned}
\int_{\mb{R}^{nN}}V_{i,2}uv\mr{d}x & =\int_{\mb{R}^{n}} \overline{\wh{V}_{i,2}(\theta)} \int_{\mb{R}^{n N}} \wh{u}(\theta-\xi_{i},-\xi^{\prime}) \wh{v}(\xi_{i},\xi^{\prime}) \mr{d}\xi \mr{d}\theta \\
& =\int_{\mb{R}^{n N-n}} \left(\int_{\mb{R}^{n}} \overline{\wh{V}_{i,2}(\theta)} \int_{\mb{R}^{n}} \wh{u}(\theta-\xi_{i},-\xi^{\prime}) \wh{v}(\xi_{i},\xi^{\prime}) \mr{d}\xi_{i} \mr{d}\theta \right)\mr{d}\xi^{\prime}. \\
%& \le \int_{\mb{R}^{n N-n}} \|V_{i,2}\|_{\mc{F}L_{s}^{\alpha^{\prime}}(\mb{R}^{n})} \int_{\mb{R}^{n}} \left|\int_{\mb{R}^{n}} \la\theta\ra^{-s}\wh{u}(\theta-\xi_{i},-\xi^{\prime}) \wh{v}(\xi_{i},\xi^{\prime}) \mr{d}\xi_{i}\right| \mr{d}\theta \mr{d}\xi^{\prime}
\end{aligned}
\end{equation}
Let $s_{1}=(|s|-s)/2$. By H\"older's inequality, (\ref{ineq: <xi> le 2^(|s|/2) <eta>^s <xi-eta>^|s|}) and Young's inequality, the term in the parenthesis in~\eqref{eq: Parseval’s relation for the part containing V_(i,2) in quadratic form} may be controlled by
\begin{equation*}
\begin{aligned}
%& \int_{\mb{R}^{n}} \overline{\wh{V}_{i,2}(\theta)} \int_{\mb{R}^{n}} \wh{u}(\theta-\xi_{i},-\xi^{\prime}) \wh{v}(\xi_{i},\xi^{\prime}) \mr{d}\xi_{i} \mr{d}\theta \\
& \quad \|V_{i,2}\|_{\mc{F}L_{-s_{1}}^{\alpha^{\prime}}(\mb{R}^{n})} \left(\int_{\mb{R}^{n}} \left|\int_{\mb{R}^{n}} \la\theta\ra^{s_{1}}\wh{u}(\theta-\xi_{i},-\xi^{\prime}) \wh{v}(\xi_{i},\xi^{\prime}) \mr{d}\xi_{i}\right|^{\alpha} \mr{d}\theta \right)^{1/\alpha} \\
%& \le 2^{|s|/2}\|V_{i,2}\|_{\mc{F}L_{s}^{\alpha^{\prime}}(\mb{R}^{n})} \left(\int_{\mb{R}^{n}} \la\theta\ra^{s_{1}r}|\wh{u}(\theta,-\xi^{\prime})|^{r} \mr{d}\theta \right)^{1/r}  \left(\int_{\mb{R}^{n}} \la\xi_{i}\ra^{s_{1}r}|\wh{v}(\xi_{i},\xi^{\prime})|^{r} \mr{d}\xi_{i} \right)^{1/r} \\
&\le 2^{s_{1}/2}\|V_{i,2}\|_{\mc{F}L_{-s_{1}}^{\alpha^{\prime}}(\mb{R}^{n})} \left(\int_{\mb{R}^{n}} \left|\int_{\mb{R}^{n}} \la\theta-\xi_i\ra^{s_{1}}\wh{u}(\theta-\xi_{i},-\xi^{\prime})\la\xi_i\ra^{s_{1}} \wh{v}(\xi_{i},\xi^{\prime}) \mr{d}\xi_{i}\right|^{\alpha} \mr{d}\theta \right)^{1/\alpha} \\
& \le 2^{s_{1}/2}\|V_{i,2}\|_{\mc{F}L_{-s_{1}}^{\alpha^{\prime}}(\mb{R}^{n})} \|\la\cdot\ra^{s_{1}}\wh{u}(\cdot,-\xi^{\prime})\|_{L^{r}(\mb{R}^{n})} \|\la\cdot\ra^{s_{1}}\wh{v}(\cdot,\xi^{\prime})\|_{L^{r}(\mb{R}^{n})} ,
\end{aligned}
\end{equation*}
where $2/r=1+1/\alpha$. Applying H\"older's inequality, we obtain 
\begin{equation*}
\|\la\cdot\ra^{s_{1}}\wh{u}(\cdot,-\xi^{\prime})\|_{L^{r}(\mb{R}^{n})} \le \|\la\cdot\ra^{s_{1}-t}\|_{L^{2\alpha}(\mb{R}^{n})} \|\la\cdot\ra^{t}\wh{u}(\cdot,-\xi^{\prime})\|_{L^{2}(\mb{R}^{n})},
\end{equation*}
and the same bound holds for $v$. 

Substituting the above two estimate into (\ref{eq: Parseval’s relation for the part containing V_(i,2) in quadratic form}) yields 
\begin{align}
\left|\int_{\mb{R}^{nN}}V_{i,2}uv\,\mr{d}x\right| & \le 2^{s_{1}/2} \|\la\cdot\ra^{s_{1}-t}\|_{L^{2\alpha}\!(\mb{R}^{n})}^{2} \|V_{i,2}\|_{\mc{F}L_{-s_{1}\!}^{\alpha^{\prime}}\!(\mb{R}^{n})}\! \int_{\mb{R}^{n N-n}} \! \|\la\cdot\ra^{t}\wh{u}(\cdot,-\xi^{\prime})\|_{L^{2}(\mb{R}^{n})} \|\la\cdot\ra^{t}\wh{v}(\cdot,\xi^{\prime})\|_{L^{2}(\mb{R}^{n})}  \mr{d}\xi^{\prime} \notag\\
& \le 2^{s_{1}/2}c_{\alpha(t-s_{1})}^{1/\alpha} \|V_{i,2}\|_{\mc{F}L_{-s_{1}}^{\alpha^{\prime}}(\mb{R}^{n})} \|u\|_{H^{t}(\mb{R}^{nN})} \|v\|_{H^{t}(\mb{R}^{nN})}, \label{ineq: bound for the part containing V_(i,2) in quadratic form}
\end{align}
where we have used $\la\xi_{i}\ra^{t}\le \la\xi\ra^{t}$ and (\ref{eq: L^(2gamma) norm of <k>^{-1}}) in the last line. 

Similarly, the terms containing $V_{i,1}$ are adapted to the case $\alpha=\infty$ in the above bound 
\begin{equation} \label{ineq: bound for the part containing V_(i,1) in quadratic form}
\begin{aligned}
\left|\int_{\mb{R}^{nN}}V_{i,1}uv\mr{d}x\right| & \le 2^{s_{1}/2} \|V_{i,1}\|_{\mc{F}L_{-s_{1}}^{1}(\mb{R}^{n})} \|u\|_{H^{t}(\mb{R}^{nN})} \|v\|_{H^{t}(\mb{R}^{nN})}. 
\end{aligned}
\end{equation}

For any fixed $1\le i,j\le N$, we rearrange $\xi=(\xi_{i},\xi_{j},\xi^{\prime})^{\top}$. Let $L$ be a linear transform and we change the variable 
\begin{equation} \label{eq: variable substitution to deal with convolutions containing V_ij}
\tilde{\xi}= L \xi=\left((\xi_{i}-\xi_{j})/2, \xi_{i}+\xi_{j},  \xi^{\prime}\right).
\end{equation}
Then, $\operatorname{det} L= 1$ and $\xi = L^{-1} \tilde{\xi}=\left(\tilde{\xi}_{i}+\tilde{\xi}_{j}/2, -\tilde{\xi}_{i}+\tilde{\xi}_{j}/2, \tilde{\xi}^{\prime}\right)^{\top}$. 
Denote $\wh{w}(\tilde{\xi})=\wh{u}(\xi)$ and $\wh{z}(\tilde{\xi})=\wh{v}(\xi)$. 
\iffalse
\begin{equation*} 
\begin{aligned}
& \wh{w}(\tilde{\xi}) = \wh{w}(\tilde{\xi}_{i},\tilde{\xi}_{j},\tilde{\xi}^{\prime}) = \wh{u}(\tilde{\xi}_{i}+\tilde{\xi}_{j}/2,-\tilde{\xi}_{i}+\tilde{\xi}_{j}/2,\tilde{\xi}^{\prime}) = \wh{u}(\xi), \\
& \wh{z}(\tilde{\xi}) = \wh{z}(\tilde{\xi}_{i},\tilde{\xi}_{j},\tilde{\xi}^{\prime}) = \wh{v}(\tilde{\xi}_{i}+\tilde{\xi}_{j}/2,-\tilde{\xi}_{i}+\tilde{\xi}_{j}/2,\tilde{\xi}^{\prime}) = \wh{v}(\xi). 
\end{aligned}
\end{equation*}
\fi
Invoking Parseval’s identity and Fubini's theorem,
again, we obtain
\begin{equation*} %\label{eq: Parseval’s relation for the part containing V_(ij,2) in quadratic form}
\begin{aligned}
\int_{\mb{R}^{nN}}V_{ij,2}uv\mr{d}x & =\int_{\mb{R}^{n}} \overline{\wh{V}_{ij,2}(\theta)} \int_{\mb{R}^{n N}} \wh{u}(\theta-\xi_{i},-\theta-\xi_{j},-\xi^{\prime}) \wh{v}(\xi_{i},\xi_{j},\xi^{\prime}) \mr{d}\xi \mr{d}\theta \\
%& =\int_{\mb{R}^{n}} \overline{\wh{V}_{ij,2}(\theta)} \int_{\mb{R}^{n N}} \wh{u}(\theta-\tilde{\xi}_{i}-\tilde{\xi}_{j}/2,-\theta+\tilde{\xi}_{i}-\tilde{\xi}_{j}/2,-\tilde{\xi}^{\prime}) \wh{v}(\tilde{\xi}_{i}+\tilde{\xi}_{j}/2,-\tilde{\xi}_{i}+\tilde{\xi}_{j}/2,\tilde{\xi}^{\prime}) \mr{d}\tilde{\xi} \mr{d}\theta \\
%& =\int_{\mb{R}^{nN-n}}\left(\int_{\mb{R}^{n}} \overline{\wh{V}_{ij,2}(\theta)} \int_{\mb{R}^{n}} \wh{u}(\theta-\tilde{\xi}_{i}-\tilde{\xi}_{j}/2,-\theta+\tilde{\xi}_{i}-\tilde{\xi}_{j}/2,-\tilde{\xi}^{\prime}) \right.\\
%& \qquad \left. \wh{v}(\tilde{\xi}_{i}+\tilde{\xi}_{j}/2,-\tilde{\xi}_{i}+\tilde{\xi}_{j}/2,\tilde{\xi}^{\prime}) \mr{d}\tilde{\xi}_{i} \mr{d}\theta\right) \mr{d}\tilde{\xi}_{j} \mr{d}\tilde{\xi}^{\prime} ,\\
& =\int_{\mb{R}^{nN-n}}\left(\int_{\mb{R}^{n}} \overline{\wh{V}_{ij,2}(\theta)} \int_{\mb{R}^{n}} \wh{w}(\theta-\tilde{\xi}_{i},-\tilde{\xi}_{j},-\tilde{\xi}^{\prime})  \wh{z}(\tilde{\xi}_{i},\tilde{\xi}_{j},\tilde{\xi}^{\prime}) \mr{d}\tilde{\xi}_{i} \mr{d}\theta\right) \mr{d}\tilde{\xi}_{j} \mr{d}\tilde{\xi}^{\prime} .
\end{aligned}
\end{equation*}

Proceeding along the same line that leads to~\eqref{ineq: bound for the part containing V_(i,2) in quadratic form}, and noting $\la\tilde{\xi}_{i}\ra^{t}\le \la\xi\ra^{t}$, we obtain
\begin{equation} \label{ineq: bounds for the part containing V_ij in quadratic form}
\begin{aligned}
\left|\int_{\mb{R}^{nN}}V_{ij,1}uv\mr{d}x\right| & \le 2^{s_{1}/2}\|V_{ij,1}\|_{\mc{F}L_{-s_{1}}^{1}(\mb{R}^{n})} \|u\|_{H^{t}(\mb{R}^{nN})} \|v\|_{H^{t}(\mb{R}^{nN})}, \\
\left|\int_{\mb{R}^{nN}}V_{ij,2}uv\mr{d}x\right| %& \le \|\la\cdot\ra^{s_{1}-t}\|_{L^{2\alpha}(\mb{R}^{n})}^{2} \|V_{ij,2}\|_{\mc{F}L_{s}^{\alpha^{\prime}}(\mb{R}^{n})} \int_{\mb{R}^{n N-n}}  \|\la\cdot\ra^{t}\wh{w}(\cdot,-\tilde{\xi}_{j},-\tilde{\xi}^{\prime})\|_{L^{2}(\mb{R}^{n})} \|\la\cdot\ra^{t}\wh{z}(\cdot,\tilde{\xi}_{j},\tilde{\xi}^{\prime})\|_{L^{2}(\mb{R}^{n})} \mr{d}\tilde{\xi}_{j} \mr{d}\tilde{\xi}^{\prime} \\
%& \le \|\la\cdot\ra^{s_{1}-t}\|_{L^{2\alpha}(\mb{R}^{n})}^{2} \|V_{ij,2}\|_{\mc{F}L_{s}^{\alpha^{\prime}}(\mb{R}^{n})} \left(\int_{\mb{R}^{n N-n}} \int_{\mb{R}^{n N-n}}\la\tilde{\xi}_{i}\ra^{2t}|\wh{w}(-\tilde{\xi}_{i},-\tilde{\xi}_{j},-\tilde{\xi}^{\prime})|^{2} \mr{d}\tilde{\xi}_{i}\mr{d}\tilde{\xi}_{j} \mr{d}\tilde{\xi}^{\prime}\right)^{1/2} \\
& \le 2^{s_{1}/2} c_{\alpha(t-s_{1})}^{1/\alpha} \|V_{ij,2}\|_{\mc{F}L_{-s_{1}}^{\alpha^{\prime}}(\mb{R}^{n})} \|u\|_{H^{t}(\mb{R}^{nN})} \|v\|_{H^{t}(\mb{R}^{nN})}.
\end{aligned}
\end{equation}
Similarly, it follows from H\"older's and Young's inequality that
\begin{equation*} 
\begin{aligned}
\left|\int_{\mb{R}^{nN}}V_{\operatorname{ad}}uv\mr{d}x\right| & \le 2^{s_{1}/2} \|V_{\operatorname{ad}}\|_{\mc{F}L_{-s_{1}}^{1}(\mb{R}^{nN})} \|u\|_{H^{t}(\mb{R}^{nN})} \|v\|_{H^{t}(\mb{R}^{nN})}. 
\end{aligned}
\end{equation*}

Finally, we conclude the desired estimate from the triangle inequality, (\ref{ineq: bound for the part containing V_(i,2) in quadratic form}), (\ref{ineq: bound for the part containing V_(i,1) in quadratic form}), (\ref{ineq: bounds for the part containing V_ij in quadratic form}) and the above bound.
\end{proof}

\begin{corollary} \label{Corollary: form-boundedness of V}
Under Assumption \ref{Assumption: V_i, v_ij in mcFL_s^1+mcFL_s^^(alpha^prime), 2+s-|s|-n/alpha>0}, $V$ is $-\Delta$ form-bounded with relative bound zero. 
\end{corollary}

\begin{proof}
Under Assumption \ref{Assumption: V_i, v_ij in mcFL_s^1+mcFL_s^^(alpha^prime), 2+s-|s|-n/alpha>0}, there exists $t<1$ satisfying the assumptions of Lemma \ref{Lemma: boundedness of the quadratic form, part of V}. It follows from Young's inequality that for any $\varepsilon>0$,
\begin{equation*}
\begin{aligned}
\la\xi\ra^{2t} & \le t\varepsilon^{1-t}\la\xi\ra^{2} + (1-t)\varepsilon^{-t}.  
\end{aligned}
\end{equation*}
Then, by Lemma \ref{Lemma: boundedness of the quadratic form, part of V}, Parseval’s identity and the above bound, we obtain
\begin{equation*}
\begin{aligned}
\left|\int_{\mb{R}^{nN}}Vu^{2}\mr{d}x\right| & \le %\mk{C}(V)\|u\|_{H^{t}(\mb{R}^{nN})}^{2} = 
\mk{C}(V)\int_{\mb{R}^{nN}}\la\xi\ra^{2t}|\wh{u}(\xi)|^{2}\mr{d}\xi  \\
%& \le \mk{C}(V) \int_{\mb{R}^{nN}}(t\varepsilon^{1-t}|\xi|^{2} + t\varepsilon^{1-t} + (1-t)\varepsilon^{-t})|\wh{u}(\xi)|^{2}\mr{d}\xi \\
%& \le \mk{C}(V) \left[t\varepsilon^{1-t}\int_{\mb{R}^{nN}}|\nabla u|^{2}\mr{d}x +  (t\varepsilon^{1-t} + (1-t)\varepsilon^{-t}) \|u\|_{L^{2}(\mb{R}^{nN})}^{2}\right] \\ 
& \le \mk{C}(V) \left[\varepsilon^{1-t}\int_{\mb{R}^{nN}}|\nabla u|^{2}\mr{d}x + (\varepsilon^{1-t} + \varepsilon^{-t}) \|u\|_{L^{2}(\mb{R}^{nN})}^{2}\right] .  
\end{aligned}
\end{equation*}
The proof completes by letting $\varepsilon\to0$.
\end{proof}
\subsection{Regularity estimates for eigenfunctions}
In this part, we prove the shift estimate for eigenfunctions stated in Theorem \ref{Theorem: Barron regularity of the eigenfunctions}. 
As a preparation, we provide multiplication estimates for potential $V$ under Assumption \ref{Assumption: V_i, v_ij in mcFL_s^1+mcFL_s^^(alpha^prime), 2+s-|s|-n/alpha>0} in Fourier Lebesgue spaces. 
The following lemma proves a convolution estimate for $L^{1}(\mb{R}^{n})+L^{\alpha^{\prime}}(\mb{R}^{n})$-functions, which is a core tool for the multiplication estimates of $V$. 
\begin{lemma} \label{Lemma: core tools for bootstrap, Young & Hölder's ineq, merged version}
Let $1\le \alpha,\alpha^{\prime}\le \infty$ be conjugate exponent $(1/\alpha+1/\alpha^{\prime}=1)$ and $\alpha>n/2$. 
Assume that $f = f_{1}+f_{2}$ with $f_{1}\in L^{1}(\mb{R}^{n})$, $f_{2}\in L^{\alpha^{\prime}}(\mb{R}^{n})$ and $\varphi \in L^{p}(\mb{R}^{n})$ with $p\in [1,2]$. 
Let $\beta>n/(2 \alpha)$ if $\alpha<\infty$ and $\beta\ge0$ if $\alpha=\infty$. Denote $\sigma = (1 - \alpha/p)_{+}$\footnote{We denote $x_{+} = \max(x,0)$ for $x\in\mb{R}$.}. Then,  $\la \cdot \ra^{-2(1-\sigma)\beta}[f * (\la \cdot \ra^{-2\sigma\beta}\varphi)] \in L^{p}(\mb{R}^{n})$ and 
\begin{equation*}
\begin{aligned}
\left\|\la \cdot \ra^{-2(1-\sigma)\beta}[f * (\la \cdot \ra^{-2\sigma\beta}\varphi)]\right\|_{L^{p}(\mb{R}^{n})} \le \left(\left\|f_{1}\right\|_{L^{1}(\mb{R}^{n})} + c_{\alpha\beta}^{1/\alpha} \left\|f_{2}\right\|_{L^{\alpha^{\prime}}(\mb{R}^{n})} \right) \|\varphi\|_{L^{p}(\mb{R}^{n})}  ,    
\end{aligned}
\end{equation*}
where the constant $c_{\alpha\beta} = \pi^{n / 2}  \Gamma(\alpha\beta - n/2)\Gamma(\alpha\beta)^{-1}$ if $\alpha<\infty$ and $c_{\alpha\beta}=1$ if $\alpha=\infty$.
\end{lemma}

\begin{proof}
Firstly, we use Hölder's inequality to get $\left\|\la \cdot \ra^{-2\sigma\beta}\varphi\right\|_{L^{p}(\mb{R}^{n})} \le \|\varphi\|_{L^{p}(\mb{R}^{n})}$ and
\begin{equation} \label{ineq: lifting of the integrability of varphi}
\begin{aligned}
\left\|\la \cdot \ra^{-2\sigma\beta}\varphi\right\|_{L^{\min(p,\alpha)}(\mb{R}^{n})} & \le \left\|\la \cdot \ra^{-2\sigma\beta}\right\|_{L^{\alpha/\sigma}(\mb{R}^{n})} \|\varphi\|_{L^{p}(\mb{R}^{n})}  ,   
\end{aligned}
\end{equation}
where $\alpha/\sigma=\infty$ when $\sigma=0.$

Secondly, we denote $\tilde{\varphi} = \la \cdot \ra^{-2\sigma\beta}\varphi$ and take $t\in[1,\infty]$ such that $1/t = (1/p-1/\alpha)_{+}$. Using Hölder's inequality and Young's inequality, we obtain 
\begin{equation} \label{ineq: estimate of convolutions f_1 * tildevarphi and f_2 * tildevarphi}
\begin{aligned}
\left\|\la \cdot \ra^{-2(1-\sigma)\beta} (f_{1} * \tilde{\varphi})\right\|_{L^{p}(\mb{R}^{n})} & \le \left\|\la \cdot \ra^{-2(1-\sigma)\beta}\right\|_{L^{\infty}(\mb{R}^{n})} \left\|f_{1} * \tilde{\varphi}\right\|_{L^{p}(\mb{R}^{n})} \\
& \le \left\|f_{1}\right\|_{L^{1}(\mb{R}^{n})}\|\tilde{\varphi}\|_{L^{p}(\mb{R}^{n})} , \\
\left\|\la \cdot \ra^{-2(1-\sigma)\beta} (f_{2} * \tilde{\varphi})\right\|_{L^{p}(\mb{R}^{n})} & \le \left\|\la \cdot \ra^{-2(1-\sigma)\beta}\right\|_{L^{\max(p,\alpha)}(\mb{R}^{n})}  \left\|f_{2} * \tilde{\varphi}\right\|_{L^{t}(\mb{R}^{n})} \\
& \le \left\|\la \cdot \ra^{-2(1-\sigma)\beta}\right\|_{L^{\max(p,\alpha)}(\mb{R}^{n})} \left\|f_{2}\right\|_{L^{\alpha^{\prime}}(\mb{R}^{n})} \|\tilde{\varphi}\|_{L^{\min(p,\alpha)}(\mb{R}^{n})} .   
\end{aligned}
\end{equation}

Using (\ref{eq: L^(2gamma) norm of <k>^{-1}}), we get
\[
\|\la\cdot\ra^{-2(1-\sigma)\beta}\|_{L^{\max(p,\alpha)}(\mb{R}^{n})} \|\la\cdot\ra^{-2\sigma\beta}\|_{L^{\alpha/\sigma}(\mb{R}^{n})} = c_{\alpha\beta}^{1/\alpha}.
\]
By the linearity of convolution and the triangle inequality, we conclude from (\ref{ineq: lifting of the integrability of varphi}), (\ref{ineq: estimate of convolutions f_1 * tildevarphi and f_2 * tildevarphi}) and the above estimates the desired estimate. 
\iffalse
\begin{equation*}
\begin{aligned}
\left\|\la \cdot \ra^{-2(1-\sigma)\beta} (f * \tilde{\varphi})\right\|_{L^{p}(\mb{R}^{n})} & \le \left\|\la \cdot \ra^{-2(1-\sigma)\beta} (f_{1} * \tilde{\varphi})\right\|_{L^{p}(\mb{R}^{n})} + \left\|\la \cdot \ra^{-2(1-\sigma)\beta} (f_{2} * \tilde{\varphi})\right\|_{L^{p}(\mb{R}^{n})} \\
& \le \left(\left\|f_{1}\right\|_{L^{1}(\mb{R}^{n})} + c_{\alpha\beta}^{1/\alpha} \left\|f_{2}\right\|_{L^{\alpha^{\prime}}(\mb{R}^{n})} \right) \|\varphi\|_{L^{p}(\mb{R}^{n})} ,
\end{aligned}
\end{equation*}
\fi
\end{proof}

The following lemma shows that under Assumption \ref{Assumption: V_i, v_ij in mcFL_s^1+mcFL_s^^(alpha^prime), 2+s-|s|-n/alpha>0}, multiplication of $V$ is a bounded linear operator from $\mc{F}L_{|s|+2\sigma\beta}^{p}(\mb{R}^{nN})$ to $\mc{F}L_{s-2(1-\sigma)\beta}^{p}(\mb{R}^{nN})$ when $1\le p\le 2$.
\begin{lemma} \label{Lemma: multiply V, bounded from mcFL_(|s|+2sigma beta)^p to mcFL_(s-2(1-sigma)beta)^p}
Under Assumption \ref{Assumption: V_i, v_ij in mcFL_s^1+mcFL_s^^(alpha^prime), 2+s-|s|-n/alpha>0}, let $\beta$ and $\sigma$ be chosen as in Lemma \ref{Lemma: core tools for bootstrap, Young & Hölder's ineq, merged version}.
%Denote $c_{\gamma} = \pi^{n / 2}  \Gamma(\gamma - n/2)\Gamma(\gamma)^{-1}$. 
Then, for any $1\le p\le 2$ and $\varphi\in\mc{F}L_{|s|+2\sigma\beta}^{p}(\mb{R}^{nN})$,
\begin{equation*}
\begin{aligned}
\|V\varphi\|_{\mc{F}L_{s-2(1-\sigma)\beta}^{p}(\mb{R}^{nN})} & \le \mc{C}(V)\|\varphi\|_{\mc{F}L_{|s|+2\sigma\beta}^{p}(\mb{R}^{nN})},  %\quad \text{for } \varphi \in \mc{F}L_{|s|+2\sigma\beta}^{p}(\mb{R}^{nN}),
\end{aligned}
\end{equation*}
where $\mc{C}(V)$ defined in (\ref{eq: operator norm of multiplying V}) is an upper bound for the operator norm.  
\end{lemma}

\begin{proof}
For any fixed $1\le i,j\le N$, we rearrange $\xi=(\xi_{i},\xi_{j},\xi^{\prime})^{\top}$. By direct calculations,
\begin{equation}
\begin{aligned} \label{eq: three terms in the Fourier transform of Vvarphi}
& \mc{F}(V_{i}(x_{i})\varphi(x))(\xi)=\int_{\mb{R}^{n}} \wh{V}_{i}(\theta) \wh{\varphi}(\xi_{i}-\theta,\xi_{j},\xi^{\prime})  \mr{d} \theta, \\
& \mc{F}(V_{ij}(x_{i}-x_{j})\varphi(x))(\xi)=\int_{\mb{R}^{n}} \wh{V}_{ij}(\theta) \wh{\varphi}(\xi_{i}-\theta,\xi_{j}+\theta,\xi^{\prime})  \mr{d} \theta, \\
& \mc{F}(V_{\operatorname{ad}}\varphi)(\xi)=\int_{\mb{R}^{Nn}} \wh{V}_{\operatorname{ad}}(\theta) \wh{\varphi}(\xi-\theta)  \mr{d} \theta . 
\end{aligned}
\end{equation}
%$(2 \pi)^{-N n / 2}\sum_{i<j}\left(\wh{V}_{i j} * \wh{\psi}\right)(\xi)$
%for any $\varphi$ such that $\wh{\varphi} \in L^{p}(\mb{R}^{N n})$ with some $p\le \alpha$.
 
Taking $d=n, x=\xi$ and $y=\theta e_1$ in (\ref{ineq: <xi> le 2^(|s|/2) <eta>^s <xi-eta>^|s|}), where $e_1$ is the cannonical basis in $\mb{R}^n$, we have 
\begin{equation} \label{ineq: bound for <xi>, adapted to convolution with V11}
\la \xi \ra^{s}\le 2^{|s|/2}\la\theta\ra^{s} \la(\xi_{i}-\theta,\xi_{j},\xi^{\prime})\ra^{|s|}.
\end{equation} 

Taking $d=n, x=\xi$ and $y=\theta(e_1+e_2)$ in (\ref{ineq: <xi> le 2^(|s|/2) <eta>^s <xi-eta>^|s|}), we get
\begin{equation}\label{ineq: bound for <xi>, adapted to convolution with V12}
\begin{aligned}
\la \xi \ra^{s} & \le 2^{|s|/2}\la y\ra^{s} \la(\xi_{i}-\theta,\xi_{j}+\theta,\xi^{\prime})\ra^{|s|} \\
& \le 2^{|s|}\la\theta\ra^{s} \la(\xi_{i}-\theta,\xi_{j}+\theta,\xi^{\prime})\ra^{|s|}. \\
\end{aligned}
\end{equation}

Denote $\wh{z}(\xi)= \la \xi \ra^{|s|} \wh{\varphi}(\xi)$.  It follows from (\ref{eq: three terms in the Fourier transform of Vvarphi})$_1$, (\ref{ineq: bound for <xi>, adapted to convolution with V11}) and the fact $\la\xi_{i}\ra\le \la\xi\ra$ that
\begin{equation*}
\begin{aligned} 
\|V_{i} \varphi\|_{\mc{F}L_{s-2(1-\sigma)\beta}^{p}(\mb{R}^{nN})}^{p} %& = \|\la\cdot\ra^{s-2(1-\sigma)\beta}\mc{F}(V_{i} \varphi)\|_{L^{p}(\mb{R}^{nN})}^{p}  \\ 
& = \int_{\mb{R}^{nN}} \left|\la\xi\ra^{s-2(1-\sigma)\beta} \int_{\mb{R}^{n}}\wh{V}_{i}(\theta) \wh{\varphi}(\xi_{i}-\theta,\xi_{j},\xi^{\prime})  \mr{d} \theta\right|^{p}\mr{d}\xi \\
%& \le 2^{p|s|/2}\int_{\mb{R}^{nN}} \left|\la\xi\ra^{-2(1-\sigma)\beta} \int_{\mb{R}^{n}} \la\theta\ra^{s} \la(\xi_{i}-\theta,\xi_{j},\xi^{\prime})\ra^{|s|} |\wh{V}_{i}(\theta) \wh{\varphi}(\xi_{i}-\theta,\xi_{j},\xi^{\prime})|  \mr{d} \theta\right|^{p}\mr{d}\xi \\
& \le 2^{p|s|/2}\int_{\mb{R}^{nN}} \left|\la\xi\ra^{-2(1-\sigma)\beta} \int_{\mb{R}^{n}} \la\theta\ra^{s} |\wh{V}_{i}(\theta) \wh{z}(\xi_{i}-\theta,\xi_{j},\xi^{\prime})|  \mr{d}\theta\right|^{p}\mr{d}\xi \\
& \le 2^{p|s|/2}\int_{\mb{R}^{nN}} \left|\la\xi_{i}\ra^{-2(1-\sigma)\beta} \int_{\mb{R}^{n}} \la\theta\ra^{s} |\wh{V}_{i}(\theta) \wh{z}(\xi_{i}-\theta,\xi_{j},\xi^{\prime})|  \mr{d}\theta\right|^{p}\mr{d}\xi\\
&\le 2^{p|s|/2} \left(\left\|V_{i,1}\right\|_{\mc{F}L_{s}^{1}(\mb{R}^{n})} + c_{\alpha\beta}^{1/\alpha} \left\|V_{i,2}\right\|_{\mc{F}L_{s}^{\alpha^{\prime}}(\mb{R}^{n})} \right)^{p} \int_{\mb{R}^{nN}} \la\xi_{i}\ra^{2\sigma\beta p} |\wh{z}(\xi_{i},\xi_{j},\xi^{\prime})|^{p} \mr{d}\xi.
\end{aligned}
\end{equation*}
where we have used Lemma \ref{Lemma: core tools for bootstrap, Young & Hölder's ineq, merged version} in the last step.
Therefore, 
\begin{equation} \label{ineq: norm estimate for V_i varphi}
\begin{aligned} 
\|V_{i} \varphi\|_{\mc{F}L_{s-2(1-\sigma)\beta}^{p}(\mb{R}^{nN})} & \le 2^{|s|/2} \left(\left\|V_{i,1}\right\|_{\mc{F}L_{s}^{1}(\mb{R}^{n})} + c_{\alpha\beta}^{1/\alpha} \left\|V_{i,2}\right\|_{\mc{F}L_{s}^{\alpha^{\prime}}(\mb{R}^{n})} \right) \|\varphi\|_{\mc{F}L_{|s|+2\sigma\beta}^{p}(\mb{R}^{nN})} .
\end{aligned}
\end{equation}

Proceeding along the same line that leads to~\eqref{ineq: norm estimate for V_i varphi}, and combining with the changing of vatiables in~\eqref {eq: variable substitution to deal with convolutions containing V_ij}, we get 
\begin{equation} \label{ineq: norm estimate for V_ij varphi}
\begin{aligned} 
\|V_{ij} \varphi\|_{\mc{F}L_{s-2(1-\sigma)\beta}^{p}(\mb{R}^{nN})} & \le 2^{|s|} \left(\left\|V_{ij,1}\right\|_{\mc{F}L_{s}^{1}(\mb{R}^{n})} + c_{\alpha\beta}^{1/\alpha} \left\|V_{ij,2}\right\|_{\mc{F}L_{s}^{\alpha^{\prime}}(\mb{R}^{n})} \right) \|\varphi\|_{\mc{F}L_{|s|+2\sigma\beta}^{p}(\mb{R}^{nN})} .
\end{aligned}
\end{equation}

As to the third term, it follows from (\ref{eq: three terms in the Fourier transform of Vvarphi}), (\ref{ineq: <xi> le 2^(|s|/2) <eta>^s <xi-eta>^|s|}) and Young's inequality that 
\begin{equation*} \label{ineq: norm estimate for V_ad varphi}
\begin{aligned} 
\|V_{\operatorname{ad}} \varphi\|_{\mc{F}L_{s-2(1-\sigma)\beta}^{p}(\mb{R}^{nN})} & = \left(\int_{\mb{R}^{nN}} \left|\la\xi\ra^{s-2(1-\sigma)\beta} \int_{\mb{R}^{nN}}\wh{V}_{\operatorname{ad}}(\theta) \wh{\varphi}(\xi-\theta)  \mr{d} \theta\right|^{p}\mr{d}\xi \right)^{1/p}  \\
& \le 2^{|s|/2}\left(\int_{\mb{R}^{nN}} \left| \int_{\mb{R}^{nN}} \la\theta\ra^{s} |\wh{V}_{\operatorname{ad}}(\theta) \wh{z}(\xi-\theta)|  \mr{d}\theta\right|^{p}\mr{d}\xi\right)^{1/p}  \\
& \le 2^{|s|/2} \left\|V_{\operatorname{ad}}\right\|_{\mc{F}L_{s}^{1}(\mb{R}^{nN})}  \|\varphi\|_{\mc{F}L_{|s|+2\sigma\beta}^{p}(\mb{R}^{nN})} .
\end{aligned}
\end{equation*}
The desired estimate follows from triangle inequality, (\ref{ineq: norm estimate for V_i varphi}), (\ref{ineq: norm estimate for V_ij varphi}) and the above estimate.
\iffalse
\begin{equation*}
\begin{aligned} 
\|V \varphi\|_{\mc{F}L_{s}^{p}(\mb{R}^{nN})} & \le \sum_{i=1}^{N} \|V_{i} \varphi\|_{\mc{F}L_{s}^{p}(\mb{R}^{nN})} +\sum_{i<j} \|V_{ij} \varphi\|_{\mc{F}L_{s}^{p}(\mb{R}^{nN})} + \|V_{\operatorname{a d}} \varphi\|_{\mc{F}L_{s}^{p}(\mb{R}^{nN})} \\
%& \le 2^{|s|/2} (2 \pi)^{-n / 2}  \left(\sum_{i=1}^{N} \|V_{i}\|_{\mc{B}^{s}(\mb{R}^{n})} +\sum_{i<j} 2^{|s|/2}\|V_{ij}\|_{\mc{B}^{s}(\mb{R}^{n})} \right)  \|z\|_{L^{p}(\mb{R}^{nN})} \\
%& \quad + 2^{|s|/2} (2\pi)^{-nN/2}\|V_{\operatorname{ad}}\|_{\mc{B}^{s}(\mb{R}^{nN})} \|z\|_{L^{p}(\mb{R}^{nN})} . 
\end{aligned}
\end{equation*}
\fi
\iffalse
\begin{equation*}
\begin{aligned} 
\|\la\cdot\ra^{s}\mc{F}(V \varphi)\|_{L^{p}(\mb{R}^{nN})} & \le \sum_{i=1}^{N} \|\la\cdot\ra^{s}\mc{F}(V_{i} \varphi)\|_{L^{p}(\mb{R}^{nN})} +\sum_{i<j} \|\la\cdot\ra^{s}\mc{F}(V_{ij} \varphi)\|_{L^{p}(\mb{R}^{nN})} + \|\la\cdot\ra^{s}\mc{F}(V_{\operatorname{a d}} \varphi)\|_{L^{p}(\mb{R}^{nN})} \\
%& \le 2^{|s|/2} (2 \pi)^{-n / 2}  \left(\sum_{i=1}^{N} \|V_{i}\|_{\mc{B}^{s}(\mb{R}^{n})} +\sum_{i<j} 2^{|s|/2}\|V_{ij}\|_{\mc{B}^{s}(\mb{R}^{n})} \right)  \|z\|_{L^{p}(\mb{R}^{nN})} \\
%& \quad + 2^{|s|/2} (2\pi)^{-nN/2}\|V_{\operatorname{ad}}\|_{\mc{B}^{s}(\mb{R}^{nN})} \|z\|_{L^{p}(\mb{R}^{nN})} . 
\end{aligned}
\end{equation*}
\fi
\end{proof}

Denote the $N$-body free quantum Hamiltonian $\mc{H}_{0} = \mc{H} - V$ and the operator
\begin{equation*}
\mc{T}_{\lambda} = (\lambda+1) \left(\mc{H}_{0}+I\right)^{-1} - \left(\mc{H}_{0}+I\right)^{-1}V .
\end{equation*} 
Substituting $\mc{H}\psi = \lambda\psi$ into the identity
\begin{equation*} %\label{eq: identity u = (H_0+I)^(-1)(H+I-V)u}
\psi = \left(\mc{H}_{0}+I\right)^{-1}(\mc{H}+I-V) \psi ,
\end{equation*}
we obtain $\psi = \mc{T}_{\lambda}\psi$, which is the key relation in our regularity estimate.
The following lemma shows that $\mc{T}_{\lambda}$ is a bounded linear operator from $\mc{F}L_{|s|+2\sigma\beta}^{p}(\mb{R}^{nN})$ to $\mc{F}L_{s-2(1-\sigma)\beta+2}^{p}(\mb{R}^{nN})$ when $1\le p\le 2$. %from $\mc{B}^{|s|}(\mb{R}^{nN})$ to $\mc{B}^{s+2}(\mb{R}^{nN})$ and from $H^{|s|}(\mb{R}^{nN})$ to $H^{s+2}(\mb{R}^{nN})$.
\begin{lemma} \label{Lemma: mcT_lambda bounded from mcFL_(|s|+2sigma beta)^p to mcFL_(s-2(1-sigma)beta+2)^p}
Under Assumption \ref{Assumption: V_i, v_ij in mcFL_s^1+mcFL_s^^(alpha^prime), 2+s-|s|-n/alpha>0}, let $\beta$ and $\sigma$ be chosen as in Lemma \ref{Lemma: core tools for bootstrap, Young & Hölder's ineq, merged version}. Then, for any $1\le p \le 2$,
\begin{equation*}
\begin{aligned}
\|\mc{T}_{\lambda}\|_{\mc{F}L_{|s|+2\sigma\beta}^{p}(\mb{R}^{nN}) \to \mc{F}L_{s-2(1-\sigma)\beta+2}^{p}(\mb{R}^{nN})} & \le \tilde{\mu}_{1} \left[|\lambda+1| + \mc{C}(V)\right] , 
\end{aligned}
\end{equation*}
where $\mc{C}(V)$ is defined in (\ref{eq: operator norm of multiplying V}). 
\iffalse
\begin{equation*}
\begin{aligned}
\|\mc{T}_{\lambda}\varphi\|_{\mc{B}^{s+2}(\mb{R}^{nN})} & \le \max_{1\le i\le N}\left(\frac{\mu_{i}}{2\pi^{2}},1\right) \left[|\lambda+1| + C(V)\right]  \|\varphi\|_{\mc{B}^{|s|}(\mb{R}^{nN})},  \quad \text{for } \varphi \in \mc{B}^{|s|}(\mb{R}^{nN}),  \\
\|\mc{T}_{\lambda}\varphi\|_{H^{s+2}(\mb{R}^{nN})} & \le \max_{1\le i\le N}\left(\frac{\mu_{i}}{2\pi^{2}},1\right) \left[|\lambda+1| + C(V)\right] \|\varphi\|_{H^{|s|}(\mb{R}^{nN})},  \quad \text{for } \varphi \in H^{|s|}(\mb{R}^{nN}),
\end{aligned}
\end{equation*}
\fi
\end{lemma}
\begin{proof}
We define the function
\begin{equation} \label{eq: Fourier transform of (H_0+I)}
h(\xi) = 2\pi^{2}\sum_{i=1}^{N} \frac{|\xi_{i}|^{2}}{\mu_{i}}+1.
\end{equation}
It follows from the fact $h(\xi)^{-1} %\le \max_{1\le i\le N}\left(\frac{\mu_{i}}{2\pi^{2}},1\right)\la\xi\ra^{-2} 
\le \tilde{\mu}_{1} \la\xi\ra^{-2}$ 
that for any $t\in\mb{R}$ and $\phi \in \mc{F}L_{t}^{p}(\mb{R}^{nN})$, 
\begin{equation*}
\left\|(\mc{H}_{0}+I)^{-1}\phi\right\|_{\mc{F}L_{t+2}^{p}(\mb{R}^{nN})} = \|\la\cdot\ra^{t+2}h^{-1}\wh{\phi}\|_{L^{p}(\mb{R}^{nN})}
\le \tilde{\mu}_{1} \left\|\phi\right\|_{\mc{F}L_{t}^{p}(\mb{R}^{nN})}.
\end{equation*}
Let $\varphi \in \mc{F}L_{|s|+2\sigma\beta}^{p}(\mb{R}^{nN})$. Using the above estimate and Lemma \ref{Lemma: multiply V, bounded from mcFL_(|s|+2sigma beta)^p to mcFL_(s-2(1-sigma)beta)^p}, we obtain
\begin{equation*}
\begin{aligned} 
\|\mc{T}_{\lambda} \varphi\|_{\mc{F}L_{s-2(1-\sigma)\beta+2}^{p}(\mb{R}^{nN})} & \le \tilde{\mu}_{1} \left(|\lambda+1| \|\varphi\|_{\mc{F}L_{s}^{p}(\mb{R}^{nN})} + \|V\varphi\|_{\mc{F}L_{s-2(1-\sigma)\beta}^{p}(\mb{R}^{nN})} \right)  \\
& \le \tilde{\mu}_{1} \left[|\lambda+1| + \mc{C}(V)\right] \|\varphi\|_{\mc{F}L_{|s|+2\sigma\beta}^{p}(\mb{R}^{nN})},
\end{aligned}
\end{equation*}
which completes the proof.
\end{proof}

Under the assumption $\gamma>\abs{s}$ in Theorem \ref{Theorem: Barron regularity of the eigenfunctions}, we observe that $\beta<1+(s-|s|)/2$ and $$[s-2(1-\sigma)\beta+2] - (|s|+2\sigma\beta) > 0.$$ 
To utilize the regularity lifting of $\mc{T}_{\lambda}$, we define the projection operator $\mc{P}_{K}:\varphi \to \mc{F}^{-1}\left(\wh{\varphi}1_{\{|\xi|>K\}}\right)$ for $K\ge 0,$ which extracts the high-frequency part of a function. The subsequent lemma provides a norm estimate for the operator $\mc{P}_{K}\mc{T}_{\lambda}$, suggesting that for sufficiently large $K$, $\mc{P}_{K}\mc{T}_{\lambda}$ is contractive. 
\begin{lemma} \label{Lemma: mcP_KmcT_lambda bounded on mcB^|s| and H^((|s|-s)/2+2sigma beta)}
Under Assumption \ref{Assumption: V_i, v_ij in mcFL_s^1+mcFL_s^^(alpha^prime), 2+s-|s|-n/alpha>0}, let $n/(2\alpha)<\beta<1+(s-|s|)/2$ and $\sigma=(1-\alpha/2)_{+}$. Then, 
\begin{equation*}
\begin{aligned}
& \|\mc{P}_{K}\mc{T}_{\lambda}\|_{\mc{B}^{|s|}(\mb{R}^{nN}) \to \mc{B}^{|s|}(\mb{R}^{nN})} \le \tilde{\mu}_{1} \left[|\lambda+1| + \mc{C}(V)\right]  \la K\ra^{|s|-s-2+2\beta},     \\
& \|\mc{P}_{K}\mc{T}_{\lambda}\|_{H^{s_{1}+2\sigma\beta}(\mb{R}^{nN}) \to H^{s_{1}+2\sigma\beta}(\mb{R}^{nN})} \le \tilde{\mu}_{1} \left[|\lambda+1| + \mc{C}(V)\right] \la K\ra^{|s|-s-2+2\beta} ,   
\end{aligned}
\end{equation*}
where $s_{1} = (|s|-s)/2$ and $\mc{C}(V)$ is defined in (\ref{eq: operator norm of multiplying V}). 
\end{lemma}

\begin{proof}
First, we examine the contractility of $\mc{P}_{K}$. For any $r<t$ and $p\ge1$, if $\|g\|_{\mc{F}L_{t}^{p}(\mb{R}^{nN})}$ is finite,
\begin{equation*}
\begin{aligned} 
\|\mc{P}_{K} g\|_{\mc{F}L_{r}^{p}(\mb{R}^{nN})}^{p} %& = \int_{\mb{R}^{n N}}\la\xi\ra^{pr} |g(\xi)|^{p} 1_{\{|\xi|>K\}} \mr{d} \xi \\
& = \int_{|\xi|>K}\la\xi\ra^{pr} |g(\xi)|^{p}  \mr{d} \xi  
\le \la K\ra^{p(r-t)} \int_{|\xi|>K}\la\xi\ra^{pt} |g(\xi)|^{p}  \mr{d} \xi, \\
%& \le \la K\ra^{p(r-t)} \|\la\cdot\ra^{t}\wh{g}\|_{L^{p}(\mb{R}^{nN})}^{p} \\
\end{aligned}
\end{equation*}
which implies 
\begin{equation} \label{ineq: norm estimate for projection operator mcP_K}
\|\mc{P}_{K}\|_{\mc{F}L_{t}^{p}(\mb{R}^{nN})\to\mc{F}L_{r}^{p}(\mb{R}^{nN})} \le \la K\ra^{-(t-r)}.
\end{equation}

For $p=1$, $\sigma$ is always 0. Substituting the bound in Lemma \ref{Lemma: mcT_lambda bounded from mcFL_(|s|+2sigma beta)^p to mcFL_(s-2(1-sigma)beta+2)^p} and (\ref{ineq: norm estimate for projection operator mcP_K}) into 
\begin{equation*}
\begin{aligned}
& \|\mc{P}_{K}\mc{T}_{\lambda}\|_{\mc{B}^{|s|}(\mb{R}^{nN}) \to \mc{B}^{|s|}(\mb{R}^{nN})} \le \|\mc{T}_{\lambda}\|_{\mc{B}^{|s|}(\mb{R}^{nN}) \to \mc{B}^{s-2\beta+2}(\mb{R}^{nN})} \|\mc{P}_{K}\|_{\mc{B}^{s-2\beta+2}(\mb{R}^{nN}) \to \mc{B}^{|s|}(\mb{R}^{nN})}  \\  
\end{aligned}
\end{equation*}
yields the first estimate.

For $p=2$, since $-s_{1}\le s$ and $2-s_{1}+|s_{1}|-n/\alpha>0$, $V$ satisfies Assumption \ref{Assumption: V_i, v_ij in mcFL_s^1+mcFL_s^^(alpha^prime), 2+s-|s|-n/alpha>0} with $s$ replaced by $-s_{1}$. Then, the second estimate follows from
\begin{equation*}
\begin{aligned}
\|\mc{P}_{K}\mc{T}_{\lambda}\|_{H^{s_{1}+2\sigma\beta}(\mb{R}^{nN}) \to H^{s_{1}+2\sigma\beta}(\mb{R}^{nN})} & \le \|\mc{T}_{\lambda}\|_{H^{s_{1}+2\sigma\beta}(\mb{R}^{nN}) \to H^{2-s_{1}-2(1-\sigma)\beta}(\mb{R}^{nN})} \cdot \\
& \qquad\|\mc{P}_{K}\|_{H^{2-s_{1}-2(1-\sigma)\beta}(\mb{R}^{nN}) \to H^{s_{1}+2\sigma\beta}(\mb{R}^{nN})} , \\
\end{aligned}
\end{equation*}
Lemma \ref{Lemma: mcT_lambda bounded from mcFL_(|s|+2sigma beta)^p to mcFL_(s-2(1-sigma)beta+2)^p} and (\ref{ineq: norm estimate for projection operator mcP_K}).
\end{proof}

Yserentant employs an effective fixed-point argument in~\cite{Yserentant2025regularity}, but the multiplier estimate therein for $V$ \cite[Lemma 2.3]{Yserentant2025regularity} essentially depends on the specific inverse power form of the Fourier transform of the Coulomb potential. To prove Theorem \ref{Theorem: Barron regularity of the eigenfunctions}, we adopt a similar fixed-point argument, while we work on Fourier-Lebesgue spaces and use H\"older-Young type inequalities to establish refined multiplier estimates for $V$, thereby greatly expanding the range of potentials that can be handled. % and the cases where Barron regularity can be obtained.

\begin{proof}[Proof of Theorem \ref{Theorem: Barron regularity of the eigenfunctions}]
First, we choose $s_{1} = (|s|-s)/2$, $\sigma=(1-\alpha/2)_{+}$ and $K$ so that %large enough
$$\tilde{\mu}_{1} \left[|\lambda+1| + \mc{C}(V)\right]  \la K\ra^{|s|-s-2+2\beta} = 1/2.$$
Then, by Lemma \ref{Lemma: mcP_KmcT_lambda bounded on mcB^|s| and H^((|s|-s)/2+2sigma beta)}, $\mc{P}_{K}\mc{T}_{\lambda}$ is contractive on $H^{s_{1}+2\sigma\beta}(\mb{R}^{nN})$ and $\mc{B}^{|s|}(\mb{R}^{nN})$. Since $\alpha\ge1$, we have $1/2\le\sigma\le1$ and $(|s|-s)/2+2\sigma\beta < 1.$ %$(|s|-s)/2+2\sigma\beta \le (|s|-s)/2+\beta < 1.$
Therefore, $\psi\in H^{s_{1}+2\sigma\beta}(\mb{R}^{nN})$.
Let $v = \mc{P}_{K}\psi$. We shall represent the high-frequency component $v$ of $\psi$ in terms of its low-frequency part, $\psi - \mc{P}_{K}\psi$. Since $\psi = \mc{T}_{\lambda}\psi$, direct calculations yield 
\[
v - \mc{P}_{K}\mc{T}_{\lambda} v  = \mc{P}_{K}\mc{T}_{\lambda}\psi - \mc{P}_{K}\mc{T}_{\lambda} \mc{P}_{K}\psi= \mc{P}_{K}\mc{T}_{\lambda}(\psi - \mc{P}_{K}\psi ).  
\]
We regard the above equation as the equation of $v$ on $H^{s_{1}+2\sigma\beta}(\mb{R}^{nN})$ where $\mc{P}_{K}\mc{T}_{\lambda}$ is contractive. Hence, $v$ has the representation 
\begin{equation} \label{eq: representation of high frequency part of psi by its low frequency part}
\begin{aligned}
v = \sum_{k=1}^{\infty}(\mc{P}_{K}\mc{T}_{\lambda})^{k}(\psi - \mc{P}_{K}\psi) .   
\end{aligned}
\end{equation}
Additionally, $\psi \in H^{1}(\mathbb{R}^{nN})$ implies that $\psi - \mathcal{P}_{K}\psi$ is bandlimited and %$\psi - \mathcal{P}_{K}\psi \in \mathcal{B}^{|s|}(\mathbb{R}^{nN})$
\begin{equation*}
\begin{aligned}
\|\psi - \mathcal{P}_{K}\psi\|_{\mathcal{B}^{|s|}(\mathbb{R}^{nN})}  =  \|\langle\cdot\rangle^{|s|}\widehat{\psi}1_{\{|\cdot|\le K\}}\|_{L^{1}(\mathbb{R}^{nN})} \le \|\langle\cdot\rangle^{|s|}1_{\{|\cdot|\le K\}}\|_{L^{2}(\mathbb{R}^{nN})} \|\widehat{\psi}\|_{L^{2}(\mathbb{R}^{nN})} .
\end{aligned}
\end{equation*}
Then, it follows from (\ref{eq: representation of high frequency part of psi by its low frequency part}) and $\|\mc{P}_{K}\mc{T}_{\lambda}\|_{\mc{B}^{|s|}(\mb{R}^{nN}) \to \mc{B}^{|s|}(\mb{R}^{nN})} \le 1/2$ that $v \in \mc{B}^{|s|}(\mb{R}^{nN})$ with 
\[
\|v\|_{\mc{B}^{|s|}(\mb{R}^{nN})} \le \|\psi - \mc{P}_{K}\psi\|_{\mc{B}^{|s|}(\mb{R}^{nN})}.
\]
Hence, $\psi = (\psi - \mc{P}_{K}\psi) + v \in \mc{B}^{|s|}(\mb{R}^{nN})$ and $$\|\psi\|_{\mc{B}^{|s|}(\mb{R}^{nN})} \le 2\|\la\cdot\ra^{|s|}1_{\{|\cdot|\le K\}}\|_{L^{2}(\mb{R}^{nN})} \|\wh{\psi}\|_{L^{2}(\mb{R}^{nN})} .$$
By the equation $\psi = \mc{T}_{\lambda}\psi$ and Lemma \ref{Lemma: mcT_lambda bounded from mcFL_(|s|+2sigma beta)^p to mcFL_(s-2(1-sigma)beta+2)^p} with $p=1$, we obtain $\psi \in \mc{B}^{s+2-2\beta}(\mb{R}^{nN})$ with 
\begin{equation*} 
\begin{aligned}
\left\|\psi\right\|_{\mc{B}^{s+2-2\beta}(\mb{R}^{Nn})} 
& \le  \tilde{\mu}_{1} \left[|\lambda+1| + \mc{C}(V)\right] \|\psi\|_{\mc{B}^{|s|}(\mb{R}^{nN})}. \\
%& \le  2\tilde{\mu}_{1} \left[|\lambda+1| + \mc{C}(V)\right]  \|\la\cdot\ra^{|s|}1_{\{|\cdot|\le K\}}\|_{L^{2}(\mb{R}^{nN})} \|\wh{\psi}\|_{L^{2}(\mb{R}^{nN})} ,
\end{aligned} 
\end{equation*}
Finally, a combination of the above two estimates and 
\begin{equation} \label{ineq: L2 norm of <xi>^|s| when xi<K} 
\begin{aligned}
\|\la\cdot\ra^{|s|}1_{\{|\cdot|\le K\}}\|_{L^{2}(\mb{R}^{nN})}
& =  \left(\omega_{nN}\int_{0}^{K} (1+r^{2})^{|s|}r^{nN-1} \mr{d}r \right)^{1/2} \\
& \le \left(\omega_{nN}\int_{0}^{K} (1+r)^{2|s|+nN-1} \mr{d}r \right)^{1/2} \\
& \le 2^{|s|/2+nN/4}\sqrt{\frac{\omega_{nN}}{2|s|+nN}} \la K\ra^{|s|+nN/2}
\end{aligned} 
\end{equation}
yields the desired estimates. Recall $\gamma = s+2-2\beta$ and then the proof completes.
\iffalse
\begin{equation*} 
\begin{aligned}
\left\|\psi\right\|_{\mc{B}^{s+2}(\mb{R}^{Nn})} 
& \le  \max_{1\le i\le N}\left(\frac{\mu_{i}}{2\pi^{2}},1\right) \left[|\lambda+1| + C(V)\right]  \|\psi\|_{\mc{B}^{|s|}(\mb{R}^{nN})}. \\
& \le  2\max_{1\le i\le N}\left(\frac{\mu_{i}}{2\pi^{2}},1\right) \left[|\lambda+1| + C(V)\right]  \|\la\cdot\ra^{|s|}1_{\{|\cdot|\le K\}}\|_{L^{2}(\mb{R}^{nN})} \|\wh{\psi}\|_{L^{2}(\mb{R}^{nN})} \\
& \le  2^{|s|/2+nN/4}\sqrt{\frac{\omega_{nN}}{2|s|+nN}}2\max_{1\le i\le N}\left(\frac{\mu_{i}}{2\pi^{2}},1\right) \left[|\lambda+1| + C(V)\right]   \la K\ra^{|s|+nN/2} \|\wh{\psi}\|_{L^{2}(\mb{R}^{nN})} \\
& \le  2^{|s|/2+nN/4}\sqrt{\frac{\omega_{nN}}{2|s|+nN}}\left\{2\max_{1\le i\le N}\left(\frac{\mu_{i}}{2\pi^{2}},1\right) \left[|\lambda+1| + C(V)\right]  \right\}^{1+\frac{|s|+nN/2}{\min(1+s,2)}} \|\wh{\psi}\|_{L^{2}(\mb{R}^{nN})} ,
\end{aligned} 
\end{equation*}
$$2\max_{1\le i\le N}\left(\frac{\mu_{i}}{2\pi^{2}},1\right) \left[|\lambda+1| + C(V)\right]  = \la K\ra^{\min(1+s,2)}.$$
\fi
\end{proof}

Simon proved the H\"older continuity result by a bootstrap argument while we employ a fixed-point argument. By Lemma \ref{Lemma: mcT_lambda bounded from mcFL_(|s|+2sigma beta)^p to mcFL_(s-2(1-sigma)beta+2)^p}, we can also perform a similar bootstrap argument. The difference between the two approaches is that Simon leverages the regularity lifting of the operator $\mc{T}_{\lambda}$ to improve the integrability of $\psi$ in the frequency space, whereas we use it to build the contractivity of the operator $\mc{P}_{K}\mc{T}_{\lambda}$. Our proof seems more adaptable and succinct. Direct calculations yield flexible norm estimates, whereas Simon's proof necessitates careful iteration counting for norm estimates. We control $\left\|\psi\right\|_{\mc{B}^{\gamma}}$ through $\|\psi\|_{L^{2}}$. Simon's proof provides control only by $\left\|\psi\right\|_{H^{2\sigma\beta}}$ with $\sigma=(1-\alpha/2)_{+}$, which is weaker.

Now, we turn to prove Corollary \ref{Corollary: regularity of eigenfunctions under inverse power potential}. To this end, we need the following lemma.
%as an application of Theorem \ref{Theorem: regularity of eigenfunction when Fourier of V_i, v_ij in L^1+L^(alpha^prime)}.
%
\begin{lemma} \label{Lemma: monotonicity of Gamma(x+a)/Gamma(x)}
Let $h(x) = \Gamma(x+a)/\Gamma(x)$ be defined on $(\max(-a,0), +\infty)$. Then, if $a\ge 0$, $h$ is monotonically nondecreasing, and if $a\le 0$, $h$ is monotonically nonincreasing. 
\end{lemma}

\begin{proof}
Denote by $\eta(x)=(\ln \Gamma(x))^{\prime}= \Gamma^{\prime}(x)/\Gamma(x)$ the digamma function.  
Since
$$\Gamma^{(k)}(x)=\int_{0}^{\infty} t^{x-1} e^{-t}(\ln t)^{k}\mr{d}t, $$
by Cauchy-Schwarts inequality, $(\Gamma^{\prime}(x))^2 \le \Gamma^{\prime\prime}(x)\Gamma(x),$ which implies $\eta^{\prime}(x) \ge 0$ and that $\eta$  is monotonically nondecreasing. 
Note that $\ln h$ and $h$ have the same monotonicity, and 
$$\left(\ln h(x)\right)^{\prime} = \eta(x+a) - \eta(x),$$
which completes the proof. 
\end{proof}

\begin{proof}[Proof of Corollary \ref{Corollary: regularity of eigenfunctions under inverse power potential}]
For $0<t<n$, a direct calculation yields $\wh{f}(\xi) = c_{t,n} |\xi|^{t-n}$ with 
\begin{equation} \label{eq: constant before the Fourier transform of |x|^(-t)}
c_{t,n} = \pi^{-n / 2+t} \Gamma((n-t) / 2) \Gamma(t / 2)^{-1}.
\end{equation}
When $n=1$, $\wh{f}(\xi) = c_{t,n} |\xi|^{t-n}$ also for $t\in(1,2)$, and $\wh{f}(\xi) = -2 (\ln |\xi| + \epsilon)$ for $t=1$, where $\epsilon$ is the Euler-Mascheroni constant. 
We decompose $\wh{f}$ into $\wh{f}_{1}(\xi) = \wh{f}(\xi) 1_{\{|\xi| \le 1\}}$ and $\wh{f}_{2}(\xi) = \wh{f}(\xi) 1_{\{|\xi| > 1\}}$, except in the case $t>n$ and $n=1$. To apply Theorem \ref{Theorem: Barron regularity of the eigenfunctions}, we choose appropriate $s$ and $\alpha$ so that $V$ meets Assumption \ref{Assumption: V_i, v_ij in mcFL_s^1+mcFL_s^^(alpha^prime), 2+s-|s|-n/alpha>0}. We always set $\beta = 1-(s-\gamma)/2$ and $\delta = 2-t-\gamma$.

For $n\ge 2$, we have $\wh{f}_{1}\in L^{1}(\mb{R}^{n})$ and $\wh{f}_{2}\in L^{\alpha^{\prime}}(\mb{R}^{n})$ for $\alpha^{\prime}>n/(n-t)$, because
\begin{equation} \label{ineq: norm of two parts of inverse power potential}
\begin{aligned}
& \left\|\wh{f}_{1}\right\|_{L^{1}(\mb{R}^{n})} = c_{t, n} \int_{|\xi| \le 1}|\xi|^{t-n} \mr{d}\xi 
%& = c_{t, n}\omega_{n} \int_{0}^{1} r^{t-1} \mr{d} r \\
= \frac{c_{t, n}\omega_{n}}{t}  , \\
%& =  \frac{2^{1-t} \Gamma((n-t) / 2)}{\Gamma(t / 2)\Gamma(n / 2)t}    \\ 
& \left\|\wh{f}_{2}\right\|_{L^{\alpha^{\prime}}(\mb{R}^{n})} = c_{t, n}\left(\int_{|\xi|>1}  |\xi|^{(t-n)\alpha^{\prime}} \mr{d} \xi\right)^{1 / \alpha^{\prime}}  
= c_{t, n}\left(\frac{\omega_{n}}{(n-t) \alpha^{\prime}-n}\right)^{1/\alpha^{\prime}} . \\
\end{aligned} 
\end{equation} 
Thus, $V$ satisfies Assumption \ref{Assumption: V_i, v_ij in mcFL_s^1+mcFL_s^^(alpha^prime), 2+s-|s|-n/alpha>0} with $s=0$ and $\alpha < n/t$.     
Let $\alpha = 2n/(2t+\delta)$ be the harmonic average of $n/(2\beta)$ and $n/t$, naturally satisfying $n/(2\beta)<\alpha<n/t$. 
To calculate $\|f\|_{s,\alpha;\beta}$, we write
\begin{equation*} 
\begin{aligned}
\alpha\beta %= \frac{2n}{2t+\delta}\cdot\frac{t+\delta}{2} 
= \frac{n}{2}\left(1+ \frac{\delta}{2t+\delta}\right) 
= \frac{n}{2}+ \frac{\alpha\delta}{4},
\end{aligned} 
\end{equation*}
$\alpha^{\prime} = 2n/(2n - 2t -\delta)$ and
$$\frac{1}{(n-t) \alpha^{\prime}-n} = \frac{2n - 2t -\delta}{n\delta} = \frac{2 - 2/\alpha}{\delta}.$$
Substituting the above quantities into (\ref{ineq: norm of two parts of inverse power potential}), we obtain 
\begin{align}\label{eq: L^1+L^(alpha^prime) norm of inverse power potential}
\left\|\wh{f}_{1}\right\|_{L^{1}(\mb{R}^{n})} + c^{1/\alpha}_{\alpha\beta} \left\|\wh{f}_{2}\right\|_{L^{\alpha^{\prime}}(\mb{R}^{n})} & = \frac{c_{t, n}\omega_{n}}{t} + c_{t, n} c^{1/\alpha}_{\alpha\beta} \left(\frac{\omega_{n}}{(n-t) \alpha^{\prime}-n}\right)^{1/\alpha^{\prime}} \nonumber\\
& = \frac{c_{t, n}\omega_{n}}{t} + c_{t, n}\omega_{n}  \left[\frac{\Gamma(n/2)\Gamma(\alpha\beta - n/2)}{2\Gamma(\alpha\beta)} \right]^{1/\alpha} \left(\frac{2 - 2/\alpha}{\delta}\right)^{1-1/\alpha} \nonumber\\
& = \frac{c_{t, n}\omega_{n}}{t} + \frac{2 c_{t, n}\omega_{n} }{\delta} \left[\frac{\Gamma(n/2) \Gamma(\alpha\beta - n/2+1)}{\alpha\Gamma(\alpha\beta)}\right]^{1/\alpha} \left(1 - 1/\alpha\right)^{1-1/\alpha}  ,  
\end{align}
where we used $\Gamma(x+1) = x\Gamma(x)$ in the last equality.  
\iffalse
\begin{equation*} 
\begin{aligned}
\left\|\wh{f}_{1}\right\|_{L^{1}(\mb{R}^{n})} + c^{1/\alpha}_{\alpha\beta} \left\|\wh{f}_{2}\right\|_{L^{\alpha^{\prime}}(\mb{R}^{n})} & = \frac{c_{t, n}\omega_{n}}{t} + c_{t, n} c^{1/\alpha}_{\alpha\beta} \left(\frac{\omega_{n}}{(n-t) \alpha^{\prime}-n}\right)^{1/\alpha^{\prime}}, \\
& = \frac{c_{t, n}\omega_{n}}{t} + c_{t, n}\omega_{n}  \left[\frac{\Gamma(n/2)\Gamma(\alpha\beta - n/2)}{2\Gamma(\alpha\beta)} \right]^{1/\alpha} \left(\frac{2 - 2/\alpha}{\delta}\right)^{1-1/\alpha}, \\
& = \frac{c_{t, n}\omega_{n}}{t} + c_{t, n} \frac{\omega_{n}}{\delta} \left(\frac{2n - 2t -\delta}{n}\right)^{1-1/\alpha} \\
& \quad\cdot \left[ \frac{\Gamma(n/2)}{2}\Gamma\left(\frac{n\delta}{2(2t+\delta)}\right)  \Gamma\left(\frac{n}{2}\left(1+ \frac{\delta}{2t+\delta}\right)\right)^{-1} \delta \right]^{1/\alpha} \\
& = \frac{c_{t, n}\omega_{n}}{t} + c_{t, n} \frac{\omega_{n}}{\delta} \left(2 - 2/\alpha\right)^{1-1/\alpha} (4/\alpha)^{1/\alpha} \\
& \quad\cdot \left[ \frac{\Gamma(n/2)}{2}\Gamma\left(\frac{n\delta}{2(2t+\delta)}+1\right)  \Gamma\left(\frac{n}{2}\left(1+ \frac{\delta}{2t+\delta}\right)\right)^{-1}  \right]^{1/\alpha}
\end{aligned} 
\end{equation*}
\fi
%Now, we make some simplifications.

Note that $\alpha\beta>n/2$ and $-n/2+1 \le 0$, by Lemma \ref{Lemma: monotonicity of Gamma(x+a)/Gamma(x)}, we get 
\[
\frac{\Gamma(\alpha\beta - n/2+1)}{\Gamma(\alpha\beta)} \le \frac{\Gamma(1)}{\Gamma(n/2)}.
\]
Substituting the above estimate into (\ref{eq: L^1+L^(alpha^prime) norm of inverse power potential}) and using $\alpha > 1$, we obtain
\begin{equation} \label{ineq: norm bound for ||f||_(s,alpha;beta), n ge 2}
\begin{aligned}
%\left\|\wh{f}_{1}\right\|_{L^{1}(\mb{R}^{n})} + c^{1/\alpha}_{\alpha\beta} \left\|\wh{f}_{2}\right\|_{L^{\alpha^{\prime}}(\mb{R}^{n})} 
\|f\|_{0,\alpha;\beta} & \le c_{t, n}\omega_{n}\left(\frac{1}{t} + \frac{2}{\delta} \right) 
\le \frac{2^{1-t} \Gamma((n-t) / 2)}{\Gamma(t / 2)\Gamma(n / 2)} \left(\frac{1}{t} + \frac{2}{\delta} \right)  .  
\end{aligned} 
\end{equation}

For $n=1$ and $t\in(0,1)$, $V$ satisfies Assumption \ref{Assumption: V_i, v_ij in mcFL_s^1+mcFL_s^^(alpha^prime), 2+s-|s|-n/alpha>0} with $s=0$ and $\alpha < n/t$. Let $\alpha = \max(2/(2t+\delta), 1)$. 
When $\gamma\ge t$, $\alpha = 2/(2t+\delta)$. Since $\delta/(2t+\delta)\in (0,1)$, we have $\alpha\beta\in (n/2,n)$. By Lemma \ref{Lemma: monotonicity of Gamma(x+a)/Gamma(x)} and $-n/2+1 > 0$, we get $$\frac{\Gamma(\alpha\beta - n/2+1)}{\Gamma(\alpha\beta)} \le \frac{\Gamma(3/2)}{\Gamma(1)} = \frac{\sqrt{\pi}}{2}.$$
Similarly, using the above estimate and (\ref{eq: L^1+L^(alpha^prime) norm of inverse power potential}), we obtain 
\begin{equation}  \label{ineq: bound for L^1+L^(alpha^prime) norm of f, s ge t}
\begin{aligned}
%\left\|\wh{f}_{1}\right\|_{L^{1}(\mb{R}^{n})} + c^{1/\alpha}_{\alpha\beta} \left\|\wh{f}_{2}\right\|_{L^{\alpha^{\prime}}(\mb{R}^{n})} 
\|f\|_{0,\alpha;\beta}
%& \le c_{t, n}\omega_{n}\left(\frac{1}{t} + \frac{\pi}{\delta} \right)  \\
& \le \frac{2\pi^{t} \Gamma((n-t) / 2)}{\Gamma(t / 2)\Gamma(n / 2)} \left(\frac{1}{t} + \frac{\pi}{\delta} \right)  .  
\end{aligned} 
\end{equation}
When $\gamma\le t$, $\alpha = 1$.  It follows from Lemma \ref{Lemma: monotonicity of Gamma(x+a)/Gamma(x)} and $\gamma\ge 0$ that $$\frac{\Gamma((3-\gamma)/2)}{\Gamma(1-\gamma/2)} \le \frac{\Gamma(3/2)}{\Gamma(1)} = \frac{\sqrt{\pi}}{2} . $$
Then, by the above estimate, a similar calculation as in (\ref{eq: L^1+L^(alpha^prime) norm of inverse power potential}) yields
\begin{equation*} 
\begin{aligned}
%\left\|\wh{f}_{1}\right\|_{L^{1}(\mb{R}^{n})} + c^{1/\alpha}_{\alpha\beta} \left\|\wh{f}_{2}\right\|_{L^{\alpha^{\prime}}(\mb{R}^{n})} 
\|f\|_{0,\alpha;\beta}
& = \frac{c_{t, n}\omega_{n}}{t} + c_{t, n}\omega_{n}  \frac{\sqrt{\pi} \Gamma((3-\gamma)/2)}{(1-\gamma)\Gamma(1-\gamma/2)}  \\
& \le \frac{c_{t, n}\omega_{n}}{t} + c_{t, n}\omega_{n}  \frac{\pi }{2(1-\gamma)}   .   
\end{aligned} 
\end{equation*}
Since $\gamma\le t$ implies $\delta \le 2(1-\gamma)$, we obtain the same bound as (\ref{ineq: bound for L^1+L^(alpha^prime) norm of f, s ge t}) when $\gamma\le t$. 

For $n=1$ and $t\in(1,3/2)$, $V$ satisfies Assumption \ref{Assumption: V_i, v_ij in mcFL_s^1+mcFL_s^^(alpha^prime), 2+s-|s|-n/alpha>0} with $s=1-t$ and $\alpha=1$ because
\begin{equation*} 
\begin{aligned}
& \left\|f\right\|_{\mc{F}L_{1-t}^{\infty}(\mb{R}^{n})} = |c_{t, n}| \max_{|\xi|\ge0} (|\xi|/\la\xi\ra)^{t-1} 
\le  |c_{t, n}| . \\
\end{aligned} 
\end{equation*}
Since $\alpha\beta = 1+(s-\gamma)/2 = (1+\delta)/2$ and $0<\delta<1$, by Lemma \ref{Lemma: monotonicity of Gamma(x+a)/Gamma(x)}, we have
\iffalse
\begin{equation*} %\label{ineq: simplification of constant c_(alphabeta), n=1}
\begin{aligned}
c^{1/\alpha}_{\alpha\beta}
& = \frac{2\sqrt{\pi}\Gamma(\delta/2+1)}{\delta\Gamma(\delta/2+1/2)}  \le \frac{\pi}{\delta}, \\
\end{aligned} 
\end{equation*}
which gives
\fi
\begin{equation} \label{ineq: norm bound for ||f||_(s,alpha;beta), n=1, t>1}
\begin{aligned}
& \|f\|_{s,\alpha;\beta} %\le c^{1/\alpha}_{\alpha\beta} \left\|f\right\|_{\mc{F}L_{s}^{\alpha^{\prime}}(\mb{R}^{n})} 
\le \frac{2\sqrt{\pi}\Gamma(\delta/2+1)}{\delta\Gamma(\delta/2+1/2)} \left\|f\right\|_{\mc{F}L_{s}^{\alpha^{\prime}}(\mb{R}^{n})} 
\le \frac{\pi|c_{t, n}|\omega_{n}}{2\delta} . \\
\end{aligned} 
\end{equation}

For $n=t=1$, $V$ satisfies Assumption \ref{Assumption: V_i, v_ij in mcFL_s^1+mcFL_s^^(alpha^prime), 2+s-|s|-n/alpha>0} with $s=(\gamma-1)/2$ and $\alpha=1$. By $\la\xi\ra^{s}\le|\xi|^{s}\le1$,
\begin{equation*} 
\begin{aligned}
& \left\|f_{1}\right\|_{\mc{F}L_{s}^{1}(\mb{R}^{n})} \le 4\int_{0}^{1} \left|\ln \xi +\epsilon\right| \mr{d}\xi 
=4 (2e^{-\epsilon}+\epsilon-1), \\
& \left\|f_{2}\right\|_{\mc{F}L_{s}^{\alpha^{\prime}}(\mb{R}^{n})} \le 2 \left(\max_{|\xi|\ge1} \la\xi\ra^{s}\ln |\xi| + \epsilon\right)
\le  2\left(\frac{1}{|s|e} + \epsilon\right). \\
\end{aligned} 
\end{equation*}
Since $\alpha\beta = 1+(s-\gamma)/2 = 1/2+\delta/4$ and $0<\delta<2/3$, by Lemma \ref{Lemma: monotonicity of Gamma(x+a)/Gamma(x)}, we have
\begin{equation} \label{ineq: norm bound for ||f||_(s,alpha;beta), n=1, t=1}  
\begin{aligned}
\|f\|_{s,\alpha;\beta}
%\left\|f_{1}\right\|_{\mc{F}L_{s}^{1}(\mb{R}^{n})} + c^{1/\alpha}_{\alpha\beta} \left\|f_{2}\right\|_{\mc{F}L_{s}^{\alpha^{\prime}}(\mb{R}^{n})}  
& \le 3 + \frac{8\sqrt{\pi}\Gamma(\delta/4+1)}{\delta\Gamma(\delta/4+1/2)}  \left(\frac{2}{\delta e} + \epsilon\right)
\le 3 + \frac{7.15}{\delta^{2}}  + \frac{5.61}{\delta} .\\
%\le 3 + \frac{8}{\delta^{2}}  + \frac{6}{\delta} .\\
\end{aligned} 
\end{equation}

Finally, the proof completes by applying Theorem \ref{Theorem: Barron regularity of the eigenfunctions} and substituting norm estimates (\ref{ineq: norm bound for ||f||_(s,alpha;beta), n ge 2})-(\ref{ineq: norm bound for ||f||_(s,alpha;beta), n=1, t=1}) into the bound $\mc{C}(V) \le \mc{M} \|f\|_{s,\alpha;\beta}$ given by the triangle inequality.
\iffalse
\begin{equation*} 
\begin{aligned}
\mc{C}(V) & \le  \left(\sum_{i = 1}^{N} \sum_{\nu = 1}^{M_{i}} |b_{\nu}| + \frac{N(N-1)}{2}\right) \|f\|_{s,\alpha;\beta}.
%\left(\left\|\wh{f}_{1}\right\|_{L^{1}(\mb{R}^{n})} + c^{1/\alpha}_{\alpha\beta} \left\|\wh{f}_{2}\right\|_{L^{\alpha^{\prime}}(\mb{R}^{n})}\right)  . 
\end{aligned}
\end{equation*}

\begin{equation*} 
\begin{aligned}
C %& \le    \sum_{i = 1}^{N} \max\left(\frac{\mu_{i}}{2\pi^{2}},1\right) \sum_{\nu = 1}^{M_{i}} |b_{\nu}| \left(\left\|\wh{f}_{1}\right\|_{L^{1}(\mb{R}^{n})} + c^{1/\alpha}_{\alpha\beta} \left\|\wh{f}_{2}\right\|_{L^{\alpha^{\prime}}(\mb{R}^{n})}\right)  \\
%& \quad + \sum_{i<j}\max\left(\frac{\mu_{i}}{4\pi^{2}},\frac{\mu_{j}}{4\pi^{2}},1\right) \left(\left\|\wh{f}_{1}\right\|_{L^{1}(\mb{R}^{n})} + c^{1/\alpha}_{\alpha\beta} \left\|\wh{f}_{2}\right\|_{L^{\alpha^{\prime}}(\mb{R}^{n})}\right)  \\
& \le  \left(\sum_{i = 1}^{N} \max\left(\frac{\mu_{i}}{2\pi^{2}},1\right) \sum_{\nu = 1}^{M_{i}} |b_{\nu}| + \sum_{i<j}\max\left(\frac{\mu_{i}}{4\pi^{2}},\frac{\mu_{j}}{4\pi^{2}},1\right)\right) \left(\left\|\wh{f}_{1}\right\|_{L^{1}(\mb{R}^{n})} + c^{1/\alpha}_{\alpha\beta} \left\|\wh{f}_{2}\right\|_{L^{\alpha^{\prime}}(\mb{R}^{n})}\right)  . 
\end{aligned}
\end{equation*}
\fi
\end{proof}

\section{Solvability of Schr\"odinger equations} \label{Section: Solvability of Schrodinger equations}
\subsection{Solvability results under Assumption \ref{Assumption: V_i, v_ij in mcFL_s^1+mcFL_s^^(alpha^prime), 2+s-|s|-n/alpha>0}}
To study the solvability of the Schr\"odinger equation $(\mc{H}+\rho I)u = f$, we firstly use Lax-Milgram theorem to ensure the existence and uniqueness of the weak solution. 
Denote the bilinear form $a_{\rho}(u,v) = a(u,v) + (\rho u,v)_{L^{2}}$. %and the seminorm $$|u|_{H^{1}(\mb{R}^{d})} = \int_{\mb{R}^{d}} |\xi|^{2}|\wh{u}(\xi)|^{2}\mr{d}\xi = \frac{1}{4\pi^{2}}\int_{\mb{R}^{d}} |\nabla u|^{2}\mr{d}x.$$
As a preparation, we prove
\begin{proposition} \label{Proposition: boundedness and coerciveness of a_rho}
Under Assumption \ref{Assumption: V_i, v_ij in mcFL_s^1+mcFL_s^^(alpha^prime), 2+s-|s|-n/alpha>0}, $\gamma$ is chosen as in Theorem \ref{Theorem: solvability of (mcH+rho I)u = f under two assumptions}. Let $t=(|s|-\gamma)/2+1$ and $A= \min_{1\le i\le N}2\pi^{2}/\mu_{i}$. Under Assumption \ref{Assumption: V_i, v_ij in mcFL_s^1+mcFL_s^^(alpha^prime), 2+s-|s|-n/alpha>0}, $a_{\rho}$ is a bounded coercive bilinear form on $H^{1}(\mb{R}^{nN})$ if
\begin{equation} \label{ineq: lower bd for rho to ensure coerciveness of a_rho under Assumption: V_i, v_ij, V_ad in mathcalB^s}
\rho>\left\{\begin{aligned} 
& \mk{C}(V) & \text{for } A>t\mk{C}(V),\\
& A+\left(\frac{1}{t}-1\right)A\left(\frac{t\mk{C}(V)}{A}\right)^{1/(1-t)}& \text{for } A\le t\mk{C}(V),
\end{aligned}\right.
\end{equation}
where $\mk{C}(V) = \mc{C}(V;(s-|s|)/2,\alpha,1+(s-\gamma)/2)$.
\end{proposition}

\begin{proof}
Since $t<1$, it follows from Lemma \ref{Lemma: boundedness of the quadratic form, part of V} that for any $u,v\in H^{1}(\mb{R}^{nN})$,
\begin{equation*}
\begin{aligned} 
\left|a_{\rho}(u,v)\right| & \le \sum_{i=1}^{N}\frac{1}{2\mu_{i}}\left| \int_{\mb{R}^{nN}} \nabla_{i}u\nabla_{i}v\mr{d}x\right| + \left|\int_{\mb{R}^{nN}}Vuv \mr{d}x\right|+ \rho\left|\int_{\mb{R}^{nN}}uv \mr{d}x\right| \\
%& \le \max_{1\le i\le N}\frac{2\pi^{2}}{\mu_{i}}|u|_{H^{1}(\mb{R}^{d})}|v|_{H^{1}(\mb{R}^{d})} +  C(V) \|u\|_{H^{|t|}(\mb{R}^{nN})} \|v\|_{H^{-t}(\mb{R}^{nN})} + \rho\|u\|_{L^{2}(\mb{R}^{nN})} \|v\|_{L^{2}(\mb{R}^{nN})} \\
& \le \left(\max_{1\le i\le N}\frac{2\pi^{2}}{\mu_{i}} + \mk{C}(V) + \rho\right) \|u\|_{H^{1}(\mb{R}^{d})}\|v\|_{H^{1}(\mb{R}^{d})}  .  
\end{aligned}
\end{equation*}
This proves that $a_{\rho}$ is bounded.

By Lemma \ref{Lemma: boundedness of the quadratic form, part of V}, 
\begin{equation*}
\begin{aligned} 
a_{\rho}(u,u) & \ge \min_{1\le i\le N}\frac{1}{2\mu_{i}}\int_{\mb{R}^{nN}} |\nabla u|^{2}\mr{d}x - \left|\int_{\mb{R}^{nN}}Vu^{2} \mr{d}x\right|+ \rho\|u\|_{L^{2}(\mb{R}^{nN})} \\
%& \ge  \min_{1\le i\le N}\frac{2\pi^{2}}{\mu_{i}} |u|_{H^{1}(\mb{R}^{d})}^{2} -  \mk{C}(V) \|u\|_{H^{|t|}(\mb{R}^{nN})}^{2} + \rho\|u\|_{L^{2}(\mb{R}^{nN})}^{2}  \\
& \ge \int_{\mb{R}^{nN}}\left(A|\xi|^{2} - \mk{C}(V)\la\xi\ra^{2|t|} + \rho\right) |\wh{u}(\xi)|^{2} \mr{d}\xi. 
\end{aligned}
\end{equation*}
For any $\varepsilon>0$, to ensure $a_{\rho}(u,u)\ge \varepsilon\|u\|_{H^{1}(\mb{R}^{nN})}^{2}$, we only need 
\begin{equation*}
\begin{aligned} 
(A-\varepsilon)z^{2} - \mk{C}(V)\la z\ra^{2|t|} + \rho-\varepsilon\ge 0,\quad \text{for all } z\in\mb{R}.
\end{aligned}
\end{equation*}
By finding the minimum value and noting the arbitrariness of $\varepsilon$, we obtain that (\ref{ineq: lower bd for rho to ensure coerciveness of a_rho under Assumption: V_i, v_ij, V_ad in mathcalB^s}) ensures the coerciveness of $a_{\rho}$. 
\end{proof}

We define the operator
\begin{equation*}
\mc{R} = \left(\mc{H}_{0}+\rho\right)^{-1} V .
\end{equation*}
A similar argument as in Lemma \ref{Lemma: mcT_lambda bounded from mcFL_(|s|+2sigma beta)^p to mcFL_(s-2(1-sigma)beta+2)^p} yields the following norm estimate for $(\mc{H}_{0}+\rho I)^{-1}$ and $\mc{R}.$ %that $\mc{R}$ is bounded from $\mc{F}L_{|s|}^{p}(\mb{R}^{nN})$ to $\mc{F}L_{s+2}^{p}(\mb{R}^{nN})$. %from $\mc{B}^{|s|}(\mb{R}^{nN})$ to $\mc{B}^{s+2}(\mb{R}^{nN})$ and from $H^{|s|}(\mb{R}^{nN})$ to $H^{s+2}(\mb{R}^{nN})$.
\begin{lemma} \label{Lemma: operator mcR is bounded from mcFL_(|s|+2sigma beta)^p to mcFL_(s-2(1-sigma)beta)^p}
Under Assumption \ref{Assumption: V_i, v_ij in mcFL_s^1+mcFL_s^^(alpha^prime), 2+s-|s|-n/alpha>0}, let $\beta$ and $\sigma$ be chosen as in Lemma \ref{Lemma: core tools for bootstrap, Young & Hölder's ineq, merged version}. Then, for any $1\le p \le 2$,
\begin{equation*}
\begin{aligned}
& \|(\mc{H}_{0}+\rho I)^{-1}\|_{\mc{F}L_{s}^{p}(\mb{R}^{nN}) \to \mc{F}L_{s+2}^{p}(\mb{R}^{nN})} \le \tilde{\mu}_{\rho} ,\\
& \|\mc{R}\|_{\mc{F}L_{|s|+2\sigma\beta}^{p}(\mb{R}^{nN}) \to \mc{F}L_{s-2(1-\sigma)\beta+2}^{p}(\mb{R}^{nN})} %& \le \max_{1\le i\le N}\left(\frac{\mu_{i}}{2\pi^{2}},\rho^{-1}\right) C(V) . 
\le \tilde{\mu}_{\rho} \mc{C}(V) ,
\end{aligned}
\end{equation*}
where $\mc{C}(V)$ is defined in (\ref{eq: operator norm of multiplying V}). 
\end{lemma}
\begin{proof}
Let $\varphi \in \mc{F}L_{|s|+2\sigma\beta}^{p}(\mb{R}^{nN})$. Recall the function $h$ defined in (\ref{eq: Fourier transform of (H_0+I)}). It follows from $$(h(\xi)-1+\rho)^{-1} \le \tilde{\mu}_{\rho}\la\xi\ra^{-2}$$ and Lemma \ref{Lemma: multiply V, bounded from mcFL_(|s|+2sigma beta)^p to mcFL_(s-2(1-sigma)beta)^p} that
\begin{equation*}
\begin{aligned} 
\|\mc{R} \varphi\|_{\mc{F}L_{s-2(1-\sigma)\beta+2}^{p}(\mb{R}^{nN})} & \le \|\la\cdot\ra^{s-2(1-\sigma)\beta+2}(h(\cdot)-1+\rho)^{-1}\mc{F}(V\varphi)\|_{L^{p}(\mb{R}^{nN})}  \\
& \le \tilde{\mu}_{\rho} \|V\varphi\|_{\mc{F}L_{s-2(1-\sigma)\beta}^{p}(\mb{R}^{nN})}   \\
& \le \tilde{\mu}_{\rho} \mc{C}(V) \|\varphi\|_{\mc{F}L_{|s|+2\sigma\beta}^{p}(\mb{R}^{nN})}, 
\end{aligned}
\end{equation*}
which completes the proof.  
\end{proof}

A similar argument as in Lemma \ref{Lemma: mcP_KmcT_lambda bounded on mcB^|s| and H^((|s|-s)/2+2sigma beta)} yields the following norm estimate for $\mc{P}_{K}\mc{R}$.
\begin{lemma} \label{Lemma: mathcalP_KmathcaR bounded on mathcalB^|s| and H^1}
Under Assumption \ref{Assumption: V_i, v_ij in mcFL_s^1+mcFL_s^^(alpha^prime), 2+s-|s|-n/alpha>0}, let $n/(2\alpha)<\beta<1+(s-|s|)/2$ and $\sigma=(1-\alpha/2)_{+}$. Then,
\begin{equation*}
\begin{aligned}
& \|\mc{P}_{K}\mc{R}\|_{\mc{B}^{|s|}(\mb{R}^{nN}) \to \mc{B}^{|s|}(\mb{R}^{nN})} \le \tilde{\mu}_{\rho} \mc{C}(V) \la K\ra^{|s|-s-2+2\beta},     \\
& \|\mc{P}_{K}\mc{R}\|_{H^{s_{1}+2\sigma\beta}(\mb{R}^{nN}) \to H^{s_{1}+2\sigma\beta}(\mb{R}^{nN})} \le \tilde{\mu}_{\rho} \mc{C}(V) \la K\ra^{|s|-s-2+2\beta} ,   
\end{aligned}
\end{equation*}
where $s_{1} = (|s|-s)/2$ and $\mc{C}(V)$ is defined in (\ref{eq: operator norm of multiplying V}). 
\end{lemma}

\begin{proof}[Proof of Theorem \ref{Theorem: solvability of (mcH+rho I)u = f under two assumptions}]
By Lax-Milgram theorem and Proposition \ref{Proposition: boundedness and coerciveness of a_rho}, for any $f\in H^{-1}(\mb{R}^{nN})$, there exists a unique weak solution $u^{*}\in H^{1}(\mb{R}^{nN})$ such that $(\mc{H}+\rho I)u^{*} = f$ and $\|u^{*}\|_{H^{1}(\mb{R}^{nN})} \le \varepsilon(\rho,V)\|f\|_{H^{-1}(\mb{R}^{nN})}$. Applying $(\mc{H}_{0}+\rho I)^{-1}$ on both sides of $(\mc{H}+\rho I)u = f$, we obtain an equivalent equation 
\begin{equation} \label{eq: equivalent eq for u*}
u + \mc{R}u = (\mc{H}_{0}+\rho I)^{-1}f.
\end{equation}
We choose $\beta = 1+(s-\gamma)/2$, $s_{1} = (|s|-s)/2$, $\sigma=(1-\alpha/2)_{+}$ and $K$ so that %large enough
$$\tilde{\mu}_{\rho} \mc{C}(V) \la K\ra^{|s|-s-2+2\beta} = 1/2.$$
Then, by Lemma \ref{Lemma: mathcalP_KmathcaR bounded on mathcalB^|s| and H^1}, $\mc{P}_{K}\mc{R}$ is contractive on $H^{s_{1}+2\sigma\beta}(\mb{R}^{nN})$ and $\mc{B}^{|s|}(\mb{R}^{nN})$.
Let $v = \mc{P}_{K}u^{*}$. As in the proof of Theorem \ref{Theorem: Barron regularity lift of the eigenfunctions}, we represent the high-frequency component $v$ of $u^{*}$ in terms of its low-frequency part $u^{*} - \mc{P}_{K}u^{*}$ as
\begin{equation*} 
\begin{aligned}
v + \mc{P}_{K}\mc{R}v = \mc{P}_{K}(\mc{H}_{0}+\rho I)^{-1}f - \mc{P}_{K}\mc{R}\left(u^{*}-\mc{P}_{K}u^{*}\right).
\end{aligned}
\end{equation*}
Note $s_{1}+2\sigma\beta\le (|s|-s)/2+\beta<1$. Since $\mc{P}_{K}(\mc{H}_{0}+\rho I)^{-1}f$, $u^{*}-\mc{P}_{K}u^{*}\in H^{s_{1}+2\sigma\beta}(\mb{R}^{nN})$ and $\|\mc{P}_{K}\mc{R}\|_{H^{s_{1}+2\sigma\beta}(\mb{R}^{nN}) \to H^{s_{1}+2\sigma\beta}(\mb{R}^{nN})} \le 1/2$, we have 
\begin{equation} \label{eq: representation of high frequency part of u* by its low frequency part}
\begin{aligned}
v = \sum_{k=0}^{\infty}(-\mc{P}_{K}\mc{R})^{k}\left[\mc{P}_{K}(\mc{H}_{0}+\rho I)^{-1}f - \mc{P}_{K}\mc{R}\left(u^{*}-\mc{P}_{K}u^{*}\right)\right] .   
\end{aligned}
\end{equation}
By Lemma \ref{Lemma: operator mcR is bounded from mcFL_(|s|+2sigma beta)^p to mcFL_(s-2(1-sigma)beta)^p} and $|s|< s-2\beta+2$, we have $(\mc{H}_{0}+\rho I)^{-1}f\in \mc{B}^{|s|}(\mb{R}^{nN})$. 
Additionally, $u^{*} \in H^{1}(\mb{R}^{nN})$ implies that $u^{*} - \mc{P}_{K}u^{*}$ is bandlimited and $u^{*} - \mc{P}_{K}u^{*} \in \mc{B}^{|s|}(\mb{R}^{nN})$. Hence, it follows from (\ref{eq: representation of high frequency part of u* by its low frequency part}) and $\|\mc{P}_{K}\mc{R}\|_{\mc{B}^{|s|}(\mb{R}^{nN}) \to \mc{B}^{|s|}(\mb{R}^{nN})} \le 1/2$ that $v \in \mc{B}^{|s|}(\mb{R}^{nN})$ and 
\begin{equation*}
\begin{aligned}
\|v\|_{\mc{B}^{|s|}(\mb{R}^{nN})} & \le 2\|\mc{P}_{K}(\mc{H}_{0}+\rho I)^{-1}f\|_{\mc{B}^{|s|}(\mb{R}^{nN})} + \|u^{*}-\mc{P}_{K}u^{*}\|_{\mc{B}^{|s|}(\mb{R}^{nN})}  \\
& \le 2\tilde{\mu}_{\rho} \la K\ra^{|s|-s-2+2\beta} \|f\|_{\mc{B}^{s-2\beta}(\mb{R}^{nN})} + \|u^{*}-\mc{P}_{K}u^{*}\|_{\mc{B}^{|s|}(\mb{R}^{nN})} ,
\end{aligned}
\end{equation*}
where we have used (\ref{ineq: norm estimate for projection operator mcP_K}) in the second line. 
Therefore,  $u^{*} = (u^{*} - \mc{P}_{K}u^{*}) + v \in \mc{B}^{|s|}(\mb{R}^{nN})$. By (\ref{eq: equivalent eq for u*}), Lemma \ref{Lemma: operator mcR is bounded from mcFL_(|s|+2sigma beta)^p to mcFL_(s-2(1-sigma)beta)^p} and the above estimate, we obtain
\begin{equation*}
\begin{aligned}
\|u^{*}\|_{\mc{B}^{s+2-2\beta}(\mb{R}^{nN})} %& \le \|\mc{R}u^{*}\|_{\mc{B}^{s+2-2\beta}(\mb{R}^{nN})} + \|(\mc{H}_{0}+\rho I)^{-1}f\|_{\mc{B}^{s+2}(\mb{R}^{nN})} \\
& \le \tilde{\mu}_{\rho} \mc{C}(V)\|u^{*}\|_{\mc{B}^{|s|}(\mb{R}^{nN})} + \tilde{\mu}_{\rho} \|f\|_{\mc{B}^{s-2\beta}(\mb{R}^{nN})} \\ 
& \le  2\tilde{\mu}_{\rho}\|f\|_{\mc{B}^{s-2\beta}(\mb{R}^{nN})} + \tilde{\mu}_{\rho} \mc{C}(V)\|u^{*}-\mc{P}_{K}u^{*}\|_{\mc{B}^{|s|}(\mb{R}^{nN})}  .
\end{aligned}
\end{equation*}
Moreover, using (\ref{ineq: L2 norm of <xi>^|s| when xi<K}), we control the norm of the low-frequency part as
\begin{equation*}
\begin{aligned}
\|u^{*} - \mathcal{P}_{K}u^{*}\|_{\mathcal{B}^{|s|}(\mathbb{R}^{nN})}  %& =  \|\langle\cdot\rangle^{|s|}\widehat{u^{*}}1_{\{|\cdot|\le K\}}\|_{L^{1}(\mathbb{R}^{nN})} 
& \le \|\langle\cdot\rangle^{|s|}1_{\{|\cdot|\le K\}}\|_{L^{2}(\mathbb{R}^{nN})} \|u^{*}\|_{H^{1}(\mathbb{R}^{nN})} \\
& \le 2^{|s|/2+nN/4}\sqrt{\frac{\omega_{nN}}{2|s|+nN}} \la K\ra^{|s|+nN/2}  \varepsilon(\rho,V)\|f\|_{H^{-1}(\mb{R}^{nN})} .
\end{aligned}
\end{equation*}
A combination of the above two estimates completes the proof.
\end{proof}
\subsection{Solvability results when \texorpdfstring{$V \in \mc{B}^{s}(\mb{R}^{Nn})$}{V in Bs}}
\iffalse
Let $W=V+\rho$ be a potential function with $\rho>0$ and $V$ of the form (\ref{eq: interaction potential V}). Now, we consider the Schrödinger operator
\begin{equation*}
\widetilde{\mc{H}}_{\rho} = -\sum_{i=1}^{N}\frac{1}{2\mu_{i}}\Delta_{i} + W,
\end{equation*}
where $\mu_{i}>0$. 
\fi
In this part, we prove Theorem \ref{Theorem: solvability of (mcH+rho I)u = f when V in mcB^s(mbR^(nN))} using the Fredholm alternative theorem. To this end, we prove the compactness of $\mc{R}$ on $\mc{B}^{s+2}(\mb{R}^{Nn})$. 
\begin{proposition} \label{Proposition: mcR is compact on mcB^(s+2)(mbR^(Nn))}
If $V \in \mc{B}^{s}(\mb{R}^{Nn})$ for some $s>-1$ and $\rho>0$, $\mc{R}$ is a compact operator on $\mc{B}^{s+2}(\mb{R}^{Nn})$.
\end{proposition}

The authors in~\cite{Chen2023regularity} showed the compactness of $\mc{R}$ on $\mc{B}^{s}(\mb{R}^{Nn})$ for nonnegative $s$. %Differently, according to the equivalent equation (\ref{eq: equivalent eq for u*}), 
We shall prove the compactness of $\mc{R}$ on $\mc{B}^{s+2}(\mb{R}^{Nn})$ and handle the case of negative $s$.

%Since $\mc{F}L_{s}^{p}(\mb{R}^{d})$ is a Banach space according to \cite[Theorem 10.1.7]{Hormander1983analysis}, to show the compactness, equivalently total boundedness, of the image of $\mc{R}$ in the Fourier Lebesgue space, we introduce the following well-known result of Kolmogorov-Riesz theorem.

We firstly recall the classical Kolmogorov-Riesz theorem.
\begin{proposition} \label{Proposition: Kolmogorov-Riesz Theorem, total boundedness}
Let $1\le p<\infty$. A subset $\mc{G}$ of  $L^{p}(\mb{R}^{d})$ is totally bounded if and only if \\
(1) $\mc{G}$ is bounded; \\
(2) for every $\varepsilon>0$, there is some $R$ so that 
\begin{equation*}
\int_{|x|>R}|g(x)|^{p} \mr{d}x<\varepsilon^{p}, \quad \text{for all $g \in \mc{G}$;}
\end{equation*}
(3) for every $\varepsilon>0$, there is some $\delta>0$ so that 
\begin{equation*}
\int_{\mb{R}^{d}}|g(x+y)-g(x)|^{p} \mr{d} x<\varepsilon^{p}, \quad \text{for all $g \in \mc{G}$ and $y\in\mb{R}^{d}$ with $|y|<\delta$}.
\end{equation*}
\end{proposition}

\begin{proof}[Proof of Proposition \ref{Proposition: mcR is compact on mcB^(s+2)(mbR^(Nn))}]
Since $\mc{B}^{s+2}(\mb{R}^{Nn})$ is a Banach space, a subset is precompact if and only if it is totally bounded. Hence, it suffices to prove the total boundedness of 
\begin{equation*} 
\begin{aligned}
\mc{G} :=  \left\{\la\cdot\ra^{2+s}\mc{F}\left(\mc{R} u\right):\|u\|_{\mc{B}^{s+2}(\mb{R}^{Nn})} \le 1\right\} \subset L^{1}(\mb{R}^{Nn}),
\end{aligned} 
\end{equation*}
where we expressed the image of $\mc{R}$ in the frequency space. By Proposition \ref{Proposition: Kolmogorov-Riesz Theorem, total boundedness}, we only need to verify the conditions (1)-(3) of Proposition \ref{Proposition: Kolmogorov-Riesz Theorem, total boundedness} for $\mc{G}$. 

(1) Note that $V \in \mc{B}^{s}(\mb{R}^{Nn})$ may be settled as a special case of Assumption \ref{Assumption: V_i, v_ij in mcFL_s^1+mcFL_s^^(alpha^prime), 2+s-|s|-n/alpha>0} when $V_{i},V_{ij}=0$ and $\alpha=\infty$. The boundedness follows directly from $\mc{B}^{2+s}(\mb{R}^{nN})\hookrightarrow\mc{B}^{|s|}(\mb{R}^{nN})$ and Lemma \ref{Lemma: operator mcR is bounded from mcFL_(|s|+2sigma beta)^p to mcFL_(s-2(1-sigma)beta)^p}. 

(2) Denote $z(\xi)=\la\xi\ra^{2+s}\wh{u}(\xi)$. Since $V \in \mc{B}^{s}(\mb{R}^{Nn})$, for any $\varepsilon>0$, there exists $R_{1}>0$ such that 
\begin{equation*}
\begin{aligned} 
\int_{|\theta|>R_{1}} \la\theta\ra^{s} |\wh{V}(\theta)| \mr{d} \theta \le \varepsilon. 
\end{aligned}
\end{equation*}
Since $2+s>|s|$, there exists $R_{2}>0$ such that $\la R_{2}\ra^{|s|-2-s} \le \varepsilon.$ We choose $R = 2\max(R_{1},R_{2})$ and by (\ref{ineq: <xi> le 2^(|s|/2) <eta>^s <xi-eta>^|s|}), we obtain 
\begin{equation*} 
\begin{aligned}
\int_{|\xi|>R}\la\xi\ra^{2+s}\left|\wh{\mc{R} u}(\xi)\right| \mr{d}\xi & \le \tilde{\mu}_{\rho} \int_{|\xi|>R}\left|\int_{\mb{R}^{Nn}} \la \xi \ra^{s}\wh{V}(\theta) \wh{u}(\xi-\theta)  \mr{d} \theta\right| \mr{d}\xi \\
& \le 2^{|s|/2}\tilde{\mu}_{\rho}\int_{|\xi|>R}\int_{\mb{R}^{Nn}} \frac{\la\theta\ra^{s}|\wh{V}(\theta) z(\xi-\theta)|}{\la\xi-\theta\ra^{2+s-|s|}}  \mr{d}\theta \mr{d}\xi . 
\end{aligned} 
\end{equation*}
We split the inner integral into two parts on regions $\{|\theta|>R/2\}$ and $\{|\theta|\le R/2\}$. Note that $|\xi|>R$ and $|\theta|\le R/2$ implies $|\xi-\theta|\ge R/2$. By Young's inequality, we obtain
\begin{equation*}
\begin{aligned} 
& \int_{|\xi|>R} \int_{|\theta|>R/2} \frac{\la\theta\ra^{s}|\wh{V}(\theta) z(\xi-\theta)|}{\la\xi-\theta\ra^{2+s-|s|}} \mr{d}\theta \mr{d}\xi \le  \|z\|_{L^{1}(\mb{R}^{nN})}\int_{|\theta|>R/2} \la\theta\ra^{s} |\wh{V}(\theta)| \mr{d} \theta ,  \\
& \int_{|\xi|>R} \int_{|\theta|\le R/2} \frac{\la\theta\ra^{s}|\wh{V}(\theta) z(\xi-\theta)|}{\la\xi-\theta\ra^{2+s-|s|}} \mr{d}\theta \mr{d}\xi \le  \la R_{2}\ra^{|s|-2-s}  \|z\|_{L^{1}(\mb{R}^{nN})} \|\la\cdot\ra^{s}\wh{V}\|_{L^{1}(\mb{R}^{nN})} .
\end{aligned}
\end{equation*}
Hence, using $\|z\|_{L^{1}(\mb{R}^{nN})} = \|u\|_{\mc{B}^{s+2}(\mb{R}^{Nn})} \le 1$, we verify the second condition by
\begin{equation*} 
\begin{aligned}
\int_{|\xi|>R}\la\xi\ra^{2+s}\left|\wh{\mc{R} u}(\xi)\right| \mr{d}\xi & \le 2^{|s|/2}\tilde{\mu}_{\rho} \left(1+\|V\|_{\mc{B}^{s}(\mb{R}^{n})}\right) \varepsilon.
\end{aligned} 
\end{equation*}
%and the arbitrariness of the choice of $\varepsilon$. \\

(3) From the verification of condition (2), for any $\varepsilon>0$, there exists $R>0$ such that 
\begin{equation*} 
\begin{aligned}
\int_{|\xi|>R}\la\xi\ra^{2+s}\left|\wh{\mc{R} u}(\xi)\right| \mr{d} \xi  <  \varepsilon,  \quad \text{for all $u$ with $\|u\|_{\mc{B}^{s+2}(\mb{R}^{Nn})} \le 1$.}
\end{aligned} 
\end{equation*}
Let $g\in \mc{G}$ and $g(\xi)=\la\xi\ra^{2+s} \wh{\mc{R} u}(\xi)$.  Denote $\xi_{y} = \xi+y$. Then,  if $|y| \le R$,  
\begin{equation} \label{ineq: decomposition for condition 3, |xi|>2R and |xi|le2R}
\begin{aligned}
\int_{\mb{R}^{nN}}\left|g(\xi_{y})-g(\xi)\right| \mr{d} \xi & \le  \int_{|\xi|>2R} \left(|g(\xi_{y})| + |g(\xi)|\right) \mr{d} \xi  +  \int_{|\xi|\le 2R} \left|g(\xi_{y})-g(\xi)\right| \mr{d} \xi \\
& \le  2\varepsilon  +  \int_{|\xi|\le 2R} \left|g(\xi_{y})-g(\xi)\right| \mr{d}\xi, 
\end{aligned} 
\end{equation}
where we used $|\xi_{y}| \ge |\xi| - |y| \ge R$ in the second inequality. We decompose
\begin{equation} \label{ineq: decomposition |xi|le2R part into I_1 + I_2}
\begin{aligned}
\int_{|\xi|\le 2R} \left|g(\xi_{y})-g(\xi)\right| \mr{d} \xi & = \int_{|\xi|\le 2R} \left|\frac{\la\xi_{y}\ra^{2+s}}{h(\xi_{y})-1+\rho} \wh{V u}(\xi_{y})-\frac{\la\xi\ra^{2+s}}{h(\xi)-1+\rho} \wh{V u}(\xi)\right| \mr{d} \xi  \\
& \le  \int_{|\xi|\le 2R} \left|\frac{\la\xi_{y}\ra^{2+s}}{h(\xi_{y})-1+\rho} -\frac{\la\xi\ra^{2+s}}{h(\xi)-1+\rho}\right| \left|\wh{V u}(\xi_{y})\right| \mr{d} \xi \\
&  \quad + \int_{|\xi|\le 2R} \frac{\la\xi\ra^{2+s}}{h(\xi)-1+\rho} \left|\wh{V u}(\xi_{y})-\wh{V u}(\xi)\right| \mr{d} \xi =: I_{1} + I_{2} . 
\end{aligned} 
\end{equation}
To control $I_{1}$, we estimate
\begin{equation*} 
\begin{aligned}
\left|\nabla \left(\frac{\la\xi\ra^{2+s}}{h(\xi)-1+\rho}\right)\right|  & \le \frac{(s+2)\la\xi\ra^{s}|\xi|}{h(\xi)-1+\rho} +\frac{4\pi^{2}\max_{1\le i\le N}\mu_{i}^{-1} |\xi|\la\xi\ra^{2+s}}{(h(\xi)-1+\rho)^{2}}  \\
& \le \left[s+2 + \frac{4\pi^{2}\tilde{\mu}_{\rho}}{\min_{1\le i\le N}\mu_{i}}\right] \tilde{\mu}_{\rho}\la\xi\ra^{s-1}. \\
\end{aligned} 
\end{equation*}
Then, the mean value theorem yields that there is a $\theta\in[0,1]$ such that
\begin{equation*} 
\begin{aligned}
I_{1} %& = \int_{|\xi|\le 2R} \left|\frac{\la\xi_{y}\ra^{2+s}}{h(\xi_{y})-1+\rho} -\frac{\la\xi\ra^{2+s}}{h(\xi)-1+\rho}\right| \left|\wh{V u}(\xi_{y})\right| \mr{d} \xi  \\ 
& \le \int_{|\xi|\le 2R} \left|\nabla \left(\frac{\la\xi_{\theta y}\ra^{2+s}}{h(\xi_{\theta y})-1+\rho}\right)\right| |y| \left|\wh{V u}(\xi_{y})\right| \mr{d} \xi  \\ 
& \le \left[s+2 + \frac{4\pi^{2}\tilde{\mu}_{\rho}}{\min_{1\le i\le N}\mu_{i}}\right] \tilde{\mu}_{\rho}|y| \int_{|\xi|\le 2R} \la\xi_{\theta y}\ra^{s} \left|\wh{V u}(\xi_{y})\right| \mr{d} \xi  \\ 
& \le \frac{2^{|s|/2}|y|}{\la R\ra^{-|s|}} \left[s+2 + \frac{4\pi^{2}\tilde{\mu}_{\rho}}{\min_{1\le i\le N}\mu_{i}}\right] \tilde{\mu}_{\rho} \int_{|\xi|\le 2R} \la\xi_{y}\ra^{s} \left|\wh{V u}(\xi_{y})\right| \mr{d} \xi,  \\ 
%& \lesssim_{s,\{\mu_{i}\},\rho}  |y|\la R\ra^{|s|} \|V u\|_{\mc{B}^{s}(\mb{R}^{n})} ,
\end{aligned} 
\end{equation*}
where in the third inequality we used the fact $\la\xi_{\theta y}\ra^{s} 
%\le 2^{|s|/2}\la\xi_{y}\ra^{s} \la(\theta-1) y\ra^{|s|} 
\le 2^{|s|/2}\la\xi_{y}\ra^{s} \la|y|\ra^{|s|}$ followed from (\ref{ineq: <xi> le 2^(|s|/2) <eta>^s <xi-eta>^|s|}). Hence, combining Lemma \ref{Lemma: multiply V, bounded from mcFL_(|s|+2sigma beta)^p to mcFL_(s-2(1-sigma)beta)^p}, we obtain
\begin{equation*} 
\begin{aligned}
I_{1} %& \le \frac{2^{|s|/2}|y|}{\la R\ra^{-|s|}} \left[s+2 + \frac{4\pi^{2}\tilde{\mu}_{\rho}}{\min_{1\le i\le N}\mu_{i}}\right] \tilde{\mu}_{\rho} \|V u\|_{\mc{B}^{s}(\mb{R}^{n})} \\
& \le \frac{2^{|s|/2}|y|}{\la R\ra^{-|s|}} \left[s+2 + \frac{4\pi^{2}\tilde{\mu}_{\rho}}{\min_{1\le i\le N}\mu_{i}}\right] \tilde{\mu}_{\rho} \|V\|_{\mc{B}^{s}(\mb{R}^{nN})}\|u\|_{\mc{B}^{s+2}(\mb{R}^{nN})} .
\end{aligned} 
\end{equation*}
We choose $\delta_{1}>0$ sufficiently small so that 
\begin{equation} \label{ineq: choice of sufficiently small delta_1}
\begin{aligned}
\frac{2^{|s|}\delta_{1}}{\la R\ra^{-|s|}} \left[s+2 + \frac{4\pi^{2}\tilde{\mu}_{\rho}}{\min_{1\le i\le N}\mu_{i}}\right] \tilde{\mu}_{\rho} \|V\|_{\mc{B}^{s}(\mb{R}^{nN})}\le \varepsilon .
\end{aligned} 
\end{equation}
Then, $I_{1}\le\varepsilon$ provided $|y|\le\delta_{1}$. 
Denote $w(\xi) = \la\xi\ra^{s}\wh{V}(\xi)$. Since $-s-1<0$, note that
\begin{equation*} 
\begin{aligned}
\left|\wh{V}(\theta+y) - \wh{V}(\theta)\right| %& = \left|\la\theta+y\ra^{-s}w(\theta+y) - \la\theta\ra^{-s}w(\theta)\right| \\
& \le  \left|\la\theta+y\ra^{-s} - \la\theta\ra^{-s}\right| |w(\theta+y)| + \la\theta\ra^{-s} |w(\theta+y)-w(\theta)| \\
& \le  |sy| \la|\theta|-|y|\ra^{-s-1} |w(\theta+y)| + \la\theta\ra^{-s} |w(\theta+y)-w(\theta)|  ,
\end{aligned} 
\end{equation*}
where we used mean value theorem in the second inequality. 
By the above estimate, to control $I_{2}$, we further decompose 
\begin{equation} \label{ineq: decomposition for term I_2}
\begin{aligned}
I_{2} %& = \int_{|\xi|\le 2R} \frac{\la\xi\ra^{2+s}}{h(\xi)-1+\rho} \left|\wh{V u}(\xi_{y})-\wh{V u}(\xi)\right| \mr{d} \xi \\
%& \le \max_{1\le i\le N}\left(\frac{\mu_{i}}{2\pi^{2}},\rho^{-1}\right)\int_{|\xi|\le 2R} \la\xi\ra^{s}\left|\wh{V u}(\xi_{y})-\wh{V u}(\xi)\right| \mr{d} \xi \\
& \le \tilde{\mu}_{\rho}\int_{|\xi|\le 2R} \int_{\mb{R}^{nN}} \la\xi\ra^{s} \left|\wh{V}(\theta+y) - \wh{V}(\theta)\right| \left|\wh{u}(\xi-\theta)\right|  \mr{d} \theta \mr{d} \xi \\ 
& \le \tilde{\mu}_{\rho}\left[|sy|\int_{|\xi|\le 2R} \int_{\mb{R}^{nN}}   \frac{\la\xi\ra^{s}|w(\theta+y)||z(\xi-\theta)|}{\la|\theta|-|y|\ra^{s+1} \la\xi-\theta\ra^{2+s}} \mr{d}\theta \mr{d}\xi \right. \\ 
& \quad \left. + \int_{|\xi|\le 2R} \int_{\mb{R}^{nN}} \frac{\la\xi\ra^{s}|w(\theta+y)-w(\theta)||z(\xi-\theta)|}{\la\theta\ra^{s}\la\xi-\theta\ra^{2+s}} \mr{d}\theta \mr{d}\xi \right] \\ 
& \le \tilde{\mu}_{\rho}\left[2^{|s|}|sy|\la R\ra^{|s|} \|V\|_{\mc{B}^{s}(\mb{R}^{nN})}\|u\|_{\mc{B}^{s+2}(\mb{R}^{nN})} \right. \\ 
& \quad \left. + \int_{|\xi|\le 2R} \int_{\mb{R}^{nN}} |w(\theta+y)-w(\theta)||z(\xi-\theta)| \mr{d}\theta \mr{d}\xi \right] , \\ 
\end{aligned} 
\end{equation}
where in the last inequality we used (\ref{ineq: <xi> le 2^(|s|/2) <eta>^s <xi-eta>^|s|}), Young's inequality and the fact
\begin{equation*} 
\begin{aligned}
\frac{\la\xi\ra^{s}}{\la|\theta|-|y|\ra^{s+1} \la\xi-\theta\ra^{2+s}} & \le \frac{2^{|s|/2}\la\theta\ra^{s}\la\xi-\theta\ra^{|s|}}{\la|\theta|-|y|\ra^{s+1} \la\xi-\theta\ra^{2+s}} \\
& \le \frac{2^{|s|}\la|\theta|-|y|\ra^{s}\la y\ra^{|s|}}{\la|\theta|-|y|\ra^{s+1}} \le 2^{|s|}\la y\ra^{|s|}\\
\end{aligned} 
\end{equation*}
obtained from (\ref{ineq: <xi> le 2^(|s|/2) <eta>^s <xi-eta>^|s|}). Since $C_{0}^{\infty}(\mb{R}^{d})$ is dense in $L^{1}(\mb{R}^{d})$ and $w\in L^{1}(\mb{R}^{d})$, we choose $\phi\in C_{0}^{\infty}(\mb{R}^{d})$ such that $\|w-\phi\|_{L^{1}(\mb{R}^{d})} < \varepsilon/\tilde{\mu}_{\rho}.$ Then, we take $0<\delta\le\delta_{1}$ small enough so that for any $|y|\le\delta$ and $\theta\in\mb{R}^{d}$, $|\phi(\theta+y)-\phi(\theta)|< \varepsilon/(M\tilde{\mu}_{\rho}),$ where $M=\int_{|\xi|\le 2R}\mr{d}\xi.$ %= \omega_{nN}(2R)^{nN}/(nN) 
The choice of $\phi$ and $|y|\le \delta$ guarantees 
\begin{equation*} 
\begin{aligned}
& \quad \int_{|\xi|\le 2R} \int_{\mb{R}^{nN}} |w(\theta+y)-w(\theta)||z(\xi-\theta)| \mr{d}\theta \mr{d}\xi \\ 
& \le  \int_{|\xi|\le 2R} \int_{\mb{R}^{nN}} |w(\theta+y)-\phi(\theta+y)||z(\xi-\theta)| \mr{d}\theta \mr{d}\xi \\ 
& \quad + \int_{|\xi|\le 2R} \int_{\mb{R}^{nN}} |\phi(\theta)-w(\theta)||z(\xi-\theta)| \mr{d}\theta \mr{d}\xi \\ 
& \quad + \int_{|\xi|\le 2R} \int_{\mb{R}^{nN}} |\phi(\theta+y)-\phi(\theta)||z(\xi-\theta)| \mr{d}\theta \mr{d}\xi \\ 
& \le  \left(2\|w-\phi\|_{L^{1}(\mb{R}^{d})} + \frac{\varepsilon}{M\tilde{\mu}_{\rho}}\int_{|\xi|\le 2R}\mr{d}\xi \right)\|u\|_{\mc{B}^{s+2}(\mb{R}^{nN})}, \\ 
\end{aligned} 
\end{equation*}
%where we used Young's inequality.
which implies
\begin{equation*} 
\begin{aligned}
\int_{|\xi|\le 2R} \int_{\mb{R}^{nN}} |w(\theta+y)-w(\theta)||z(\xi-\theta)| \mr{d}\theta \mr{d}\xi \le \frac{3\varepsilon}{\tilde{\mu}_{\rho}}. \\ 
\end{aligned} 
\end{equation*}
Substituting (\ref{ineq: choice of sufficiently small delta_1}) and the above estimate into (\ref{ineq: decomposition for term I_2}), we obtain $I_{2}\le 4\varepsilon.$ We conclude from the decompositions (\ref{ineq: decomposition for condition 3, |xi|>2R and |xi|le2R}), (\ref{ineq: decomposition |xi|le2R part into I_1 + I_2}) and the estimates for $I_{1}$, $I_{2}$ that 
$$\int_{\mb{R}^{nN}}\left|g(\xi_{y})-g(\xi)\right| \mr{d} \xi \le 7\varepsilon, \quad \text{for all $g\in \mc{G}$ and $|y|\le\delta$.}$$
Hence, we have verified condition (3) and complete the proof of the compactness of $\mc{R}$.
\end{proof}

\begin{proof}[Proof of Theorem \ref{Theorem: solvability of (mcH+rho I)u = f when V in mcB^s(mbR^(nN))}]
Applying $(\mc{H}_{0}+\rho I)^{-1}$ on both sides of $(\mc{H}+\rho I)u = f$, we obtain the equivalent equation (\ref{eq: equivalent eq for u*}). Theorem \ref{Theorem: solvability of (mcH+rho I)u = f when V in mcB^s(mbR^(nN))} follows immediately from the Fredholm alternative theorem and Proposition \ref{Proposition: mcR is compact on mcB^(s+2)(mbR^(Nn))}. Furthermore, when point (2) occurs, $(I+\mc{R})^{-1}$ is bounded on $\mc{B}^{s+2}(\mb{R}^{Nn})$. By Lemma \ref{Lemma: operator mcR is bounded from mcFL_(|s|+2sigma beta)^p to mcFL_(s-2(1-sigma)beta)^p}, we obtain 
\begin{equation*} 
\|u\|_{\mc{B}^{s+2}(\mb{R}^{Nn})} %\le \|(I + \mc{R})^{-1}\|_{\mc{B}^{s+2}(\mb{R}^{Nn})\to\mc{B}^{s+2}(\mb{R}^{Nn})}\|(\mc{H}_{0}+\rho I)^{-1}f\|_{\mc{B}^{s+2}(\mb{R}^{Nn})}
\le \tilde{\mu}_{\rho}\|(I + \mc{R})^{-1}\|_{\mc{B}^{s+2}(\mb{R}^{Nn})\to\mc{B}^{s+2}(\mb{R}^{Nn})}\|f\|_{\mc{B}^{s}(\mb{R}^{Nn})},
\end{equation*}
which completes the proof. 
\end{proof}

\section{Conclusion and future work} \label{Section: Conclusion and future work}
In this paper, we study the Barron regularity of Schr\"odinger eigenfunctions and the solvability of Schr\"odinger equations in Barron spaces. Under Assumption \ref{Assumption: V_i, v_ij in mcFL_s^1+mcFL_s^^(alpha^prime), 2+s-|s|-n/alpha>0}, we prove that all eigenfunctions $\psi\in \bigcap_{\gamma<s+2-n/\alpha} \mc{B}^{\gamma}(\mb{R}^{nN})$ and $\psi\in \mc{B}^{s+2}(\mb{R}^{nN})$ if $\alpha=\infty$. We obtain Simon's estimate~\cite[Theorem $1^{\prime}$]{Simon1974Pointwise}, regularity of eigenfunctions with inverse power potentials and Yserentant’s result~\cite[Theorem 4.6]{Yserentant2025regularity} as natural downstream corollaries. 
Under Assumption \ref{Assumption: V_i, v_ij in mcFL_s^1+mcFL_s^^(alpha^prime), 2+s-|s|-n/alpha>0} and $\rho$ large enough, we prove the existence, uniqueness and Barron regularity of the solution to $(\mc{H}+\rho I)u = f$. Additionally, when $V\in\mc{B}^{s}(\mb{R}^{nN})$, we prove solvability results purely in Barron spaces. Our Assumption \ref{Assumption: V_i, v_ij in mcFL_s^1+mcFL_s^^(alpha^prime), 2+s-|s|-n/alpha>0} greatly expands the range of potentials ensuring the Barron regularity and applies to many commonly used singular potentials. 

There are still some interesting problems to be further explored. In terms of regularity, what would happen in the borderline case $2+s-|s|-n/\alpha=0$. For electronic wave functions, can higher Barron regularity be proved by decomposing the unsmooth parts of wave functions, just as done for H\"older regularity in \cite{Fournais2005Sharp,Fournais2009Analytic}. 
In terms of approximation, since the eigenfunctions and solutions to the Schr\"odinger equation proved in this paper all enjoy Barron regularity with some positive index $t>0$, combining with the results in \cite{liao2025sharp}, they can be approximated on bounded domains by neural networks of width $W$ and depth $L$ with an error of $O(W^{-\min(tL, 1/2)})$ under $L^{p}$-norm ($1\le p\le\infty$). Better approximations may be expected if we take into account more refined regularity structures and symmetries. Furthermore, the numerical analysis of machine learning methods for solving Schr\"odinger equations and eigenvalue problems  with more general and practical potentials can be carried out on the basis of the regularity guarantee obtained in this paper.

\appendix
\section{Estimate of typical examples} \label{Appendix section: Estimate_Barron_smoothness}
First, we give a proof of Lemma \ref{Lemma: index condition s<n/alpha-t for |x|^(-t) to satisfy Assumption}, showing the index conditions for inverse power potentials and the Yukawa potential.
\begin{proof}[Proof of Lemma \ref{Lemma: index condition s<n/alpha-t for |x|^(-t) to satisfy Assumption}]
Denote $f(x) = |x|^{-t}$. Recall that $\wh{f}(\xi) = c_{t,n} |\xi|^{t-n}$ with $c_{t,n}$ in (\ref{eq: constant before the Fourier transform of |x|^(-t)}). We decompose $\wh{f}$ into $\wh{f}_{1}(\xi) = \wh{f}(\xi) 1_{\{|\xi| \le 1\}}$ and $\wh{f}_{2}(\xi) = \wh{f}(\xi) 1_{\{|\xi| > 1\}}$. By (\ref{ineq: norm of two parts of inverse power potential}), $\wh{f}_{1}\in L^{1}(\mb{R}^{n})$ has a compact support, which implies $f_{1}\in \mc{F}L_{s}^{1}(\mb{R}^n)$ for any $s\in\mb{R}$. On the other hand, $\la\cdot\ra^{s}\wh{f}_{2}\in L^{\alpha^{\prime}}(\mb{R}^n)$ if $(s+t-n)\alpha^{\prime} < -n$, i.e. $s<n/\alpha-t$. For the Yukawa potential $g(x) = \abs{x}^{-1}e^{-\mu\abs{x}}$, a direct calculation yields
\begin{equation*}
\wh{g}(\xi) = 4\pi (\mu^{2}+4\pi^2|\xi|^{2})^{-1} .
\end{equation*}
Then, a similar argument completes the proof.
\end{proof}

Next, we estimate the decay rate of Fourier transform of the eigenfunction $\psi(x) = e^{-\abs{x}^{\delta}}$ in Example \ref{Example: Sharpness of regularity of eigenfunctions}  to give its exact Barron regularity. 
\begin{lemma} \label{Lemma: decay rate of Fourier transform of e^(|x|^delta)}
Let $n\ge2$, $0<\delta<2$ and $\psi(x) = e^{-\abs{x}^{\delta}}$. Then, $\wh{\psi}$ is continuous and
\begin{equation*}
\begin{aligned}
\left|\wh{\psi}(\xi) - C_{1}(n,\delta)\abs{\xi}^{-\delta-n}\right| & \le C(n,\delta) |\xi|^{-2\delta-n} ,
\end{aligned}
\end{equation*}
where  $C(n,\delta)$ is a constant depending only on $n,\delta$ and $$C_{1}(n,\delta) = -\frac{\delta \Gamma((\delta+n) / 2)}{2 \pi^{n / 2+\delta} \Gamma(1-\delta / 2)} .$$
\end{lemma}
The above lemma states that $\wh{\psi}$ decays exactly at a rate of $\la\xi\ra^{-\delta-n}$. A typical situation is $\delta=1$. 
By~\cite[Theorem 1.14 with $\al=1/(2\pi)$]{SteinWeiss:1971}, $\wh{\psi}$ has the expilicit formula 
\[
\wh{\psi}(\xi) = 2^{n}\pi^{(n-1)/2} \Gamma((n+1)/2) \left(1+4 \pi^{2}|\xi|^{2}\right)^{-(n+1)/2}.
\]
\begin{proof}
It follows $\psi\in L^{1}(\mb{R}^{n})$ that $\wh{\psi}$ is continuous and bounded. We only need to prove the desired estimate for $|\xi|>1.$ Since $\psi$ is smooth and decays exponentially away from the origin, the decay of $\wh{\psi}$ is determined by its behavior near the origin. Note that $\abs{x}^{\delta}$ has a Fourier transform in the distributional sense
\begin{equation*}
\begin{aligned}
\mc{F}(\abs{\cdot}^{\delta}) & = \left(4 \pi^{2}|\xi|^{2}\right)^{-1} \mc{F}(\Delta(\abs{\cdot}^{\delta})) = -C_{1}(n,\delta)\abs{\xi}^{-\delta-n} . \\
%& = \frac{\delta(\delta+n-2)}{4 \pi^{2}|\xi|^{2}} \mc{F}(\abs{\cdot}^{\delta-2}) \\
%& = \frac{\delta(\delta+n-2)}{4 \pi^{2}|\xi|^{2}} \pi^{-n / 2+2-\delta} \Gamma((n-2+\delta) / 2) \Gamma(1-\delta / 2)^{-1} \abs{\xi}^{2-\delta-n} \\
%& = \frac{\delta \Gamma((\delta+n) / 2)}{2 \pi^{n / 2+\delta} \Gamma(1-\delta / 2)}  \abs{\xi}^{-\delta-n} .\\
\end{aligned}
\end{equation*}
Let $\chi$ be a smooth function supported in $|x|\le1$ and with value $1$ in $|x|\le1/2$.  We decompose
\begin{equation} \label{eq: decomposition of high order terms}
\begin{aligned}
\psi + \abs{x}^{\delta} & = (\psi+\abs{x}^{\delta}-1)\chi + \chi + \psi(1-\chi) +\abs{x}^{\delta}(1-\chi) ,
\end{aligned}
\end{equation}
where $\chi$ and $\psi(1-\chi)$ are Schwartz functions. Denote $g(x) = \psi(x)+\abs{x}^{\delta}-1.$ It remains to estimate the Fourier transforms of $g\chi$ and $\abs{x}^{\delta}(1-\chi)$.

For $|\xi|>1$, $\mc{F}(\abs{\cdot}^{\delta}(1-\chi))$ decays rapidly. We take $M = \lceil(2 \delta+n)/2\rceil +1$ and obtain
\begin{equation} \label{ineq: decay of Fourier trans. of |x|^delta(1-chi)}
\begin{aligned}
\left|\mc{F}(\abs{\cdot}^{\delta}(1-\chi))\right| & = \left(4 \pi^{2}|\xi|^{2}\right)^{-M} \left|\mc{F}[\Delta^{M}(\abs{\cdot}^{\delta}(1-\chi))] \right| \\
& \le \frac{1}{|\xi|^{2M}} \left(\int_{\frac{1}{2}\le|x|\le1}\left|\Delta^{M}\left[\abs{x}^{\delta}(1-\chi)\right]\right| \mr{d}x + C(n, \delta) \int_{|x| \ge 1}|x|^{\delta-2 M} \mr{d}x\right) \\
& \le C(n, \delta) |\xi|^{-2M} ,\\
\end{aligned}
\end{equation}
where $C(n, \delta)$ may be different from line to line, and the second inequality follows from that $\abs{x}^{\delta}(1-\chi)$ is smooth and $\delta-2 M<-n$.  

To estimate $\wh{g\chi}$, we take $0<\varepsilon<1/4$ and denote $\chi_{\varepsilon}(x) = \chi(x/\varepsilon).$ We decompose $\wh{g\chi}$ into $J_{1}(\xi):=\wh{g\chi_{\varepsilon}}(\xi)$ and $J_{2}(\xi):=\wh{g\chi}(\xi) - J_{1}(\xi)$. On the one hand, it follows from $0\le e^{-x}-1+x \le x^{2}/2$ for $x\ge0$ that
\begin{equation} \label{ineq: bound for |J_1(xi)|}
\begin{aligned}
\left|J_{1}(\xi)\right| & \le \int_{|x|\le\varepsilon}|g(x)|\mr{d}x \le \frac{1}{2} \int_{|x|\le\varepsilon}|x|^{2 \delta} \mr{d}x \\
&=\frac{\omega_{n-1}}{2(2 \delta+n)} \varepsilon^{2 \delta+n} .\\
\end{aligned}
\end{equation}
On the other hand, for any $\epsilon<1/4$, we claim 
\begin{equation*}
\begin{aligned}
\left|\nabla^{m} g(x)\right| \le C_{2}(n,\delta)|x|^{2\delta-m} , \quad \text{for all } \frac{\varepsilon}{2}\le|x|\le1, \  m\le 2M, \\
\end{aligned}
\end{equation*}
where $C_{2}(n,\delta)$ depends only on $n,\delta$. The case $m=0$ has been proved. A direct calculation yields for $1\le k\le M$,
\begin{equation*}
\begin{aligned}
& \nabla^{2 k-1} g(x)=\sum_{p=2}^{\infty} \frac{(-1)^{p}}{p!}\left[\prod_{l=0}^{k-2}(p \delta-2l)(p \delta-2l+n-2)\right] (p \delta-2k+2) |x|^{p\delta -2 k} x ,\\
& \nabla^{2 k} g(x)=\sum_{p=2}^{\infty} \frac{(-1)^{p}}{p!}\left[\prod_{l=0}^{k-1}(p \delta-2l)(p \delta-2l+n-2)\right]|x|^{p\delta -2 k},
\end{aligned}
\end{equation*}
where the series converge uniformly on $\varepsilon/2\le|x|\le1$. Hence, 
\iffalse
\begin{equation*}
\begin{aligned}
& |\nabla^{2 k-1} g(x)| \le \sum_{p=2}^{\infty} \frac{1}{p!}(p \delta+2k-2+n)^{2 k-1} |x|^{2\delta -2 k+1}, \\
& |\nabla^{2 k} g(x)| \le \sum_{p=2}^{\infty} \frac{1}{p!}(p \delta+2k+n)^{2k}|x|^{2\delta -2 k},
\end{aligned}
\end{equation*}
\fi
\begin{equation*}
\begin{aligned}
& |\nabla^{m} g(x)| \le \left[\sum_{p=2}^{\infty} \frac{1}{p!}(p \delta+m+n)^{m}\right] |x|^{2\delta -m},
\end{aligned}
\end{equation*}
for all $1\le m\le 2M$. The claim is proved with 
$$C_{2}(n,\delta) = \frac{1}{2} + \max_{1\le m\le 2M}\sum_{p=2}^{\infty} \frac{1}{p!}(p \delta+m+n)^{m}.$$
Now, we control $J_{2}$ by
\begin{equation*}
\begin{aligned}
\left|J_{2}(\xi)\right| & = \left(4 \pi^{2}|\xi|^{2}\right)^{-M} \left|\mc{F}[\Delta^{M}(g\chi(1-\chi_{\varepsilon}))] \right| \\
& \le \frac{1}{|\xi|^{2M}} \left(\int_{\frac{\varepsilon}{2}\le|x|\le\varepsilon}\left|\Delta^{M}\left[g(1-\chi_{\varepsilon})\right]\right| \mr{d}x +  \int_{\varepsilon\le|x|\le1}\left|\Delta^{M}(g\chi)\right| \mr{d}x\right) \\
& \le \frac{C(n, \delta)}{|\xi|^{2M}} \sum_{m=0}^{2M}\left(\int_{\frac{\varepsilon}{2}\le|x|\le\varepsilon}|\nabla^{m} g| |\nabla^{2M-m}(1-\chi_{\varepsilon})|  \mr{d}x  +  \int_{\varepsilon\le|x|\le1}|\nabla^{m} g| |\nabla^{2M-m}\chi| \mr{d}x\right) .
\end{aligned}
\end{equation*}
Note that $|\nabla^{m}(1-\chi_{\varepsilon})| \le C(n, \delta)\varepsilon^{-m}$ for $m\le 2M$. With the aid of the claim, we further bound 
\begin{equation*}
\begin{aligned}
\left|J_{2}(\xi)\right| & \le \frac{C(n, \delta)}{|\xi|^{2M}} \sum_{m=0}^{2M}\left(\varepsilon^{m-2M}\int_{\frac{\varepsilon}{2}\le|x|\le\varepsilon}|x|^{2\delta -m}   \mr{d}x  +  \int_{\varepsilon\le|x|\le1}|x|^{2\delta -m} \mr{d}x\right) \\
& \le C(n, \delta) \varepsilon^{2\delta+n-2M} |\xi|^{-2M}, \\
\end{aligned}
\end{equation*}
where the last inequality follows from $2M>2\delta+n$ and $\varepsilon<1$. Combining (\ref{ineq: bound for |J_1(xi)|}) and the above bound, we have
\begin{equation*}
\begin{aligned}
\left|\wh{g\chi}(\xi)\right| & \le C(n, \delta) \varepsilon^{2\delta+n} \left[1+(\varepsilon|\xi|)^{-2M} \right] \\
& \le C(n, \delta) |\xi|^{-2\delta-n}  ,
\end{aligned}
\end{equation*}
where we set $\varepsilon = |\xi|^{-1}/4$ in the last inequality. A combination of (\ref{eq: decomposition of high order terms}), (\ref{ineq: decay of Fourier trans. of |x|^delta(1-chi)}) and the above estimate completes the proof.
\end{proof}

%\bibliography{reference}
% \bib, bibdiv, biblist are defined by the amsrefs package.

\end{document}